\documentclass[11pt]{amsart}
\usepackage{amsmath, amssymb, amscd, mathrsfs, url, pinlabel,verbatim}
\usepackage[pagebackref]{hyperref}
\usepackage[margin=1.5in,marginparwidth=1in,centering,a4paper,dvips]{geometry}
\usepackage{color,dcpic,latexsym,graphicx,epstopdf,comment}
\usepackage[all]{xy}
\usepackage[dvipsnames]{xcolor}
\usepackage{upquote,tabularx,textcomp}
\usepackage[shortlabels]{enumitem}
\usepackage{mathtools}
\usepackage[color=blue!20!white,textsize=tiny]{todonotes}

\usepackage{tikz,tikz-cd,pgfplots}
\usetikzlibrary{decorations.markings,arrows,calc,intersections}

\title{Rational homology 3-spheres and $\SL(2,\C)$ representations}

\author{Sudipta Ghosh}
\address{Department of Mathematics\\University of Notre Dame}
\email{sghosh7@nd.edu}

\author{Steven Sivek}
\address{Department of Mathematics\\Imperial College London}
\email{s.sivek@imperial.ac.uk}

\author{Raphael Zentner}
\address{Department of Mathematical Sciences\\Durham University}
\email{raphael.zentner@durham.ac.uk}

\makeatletter
\newtheorem*{rep@theorem}{\rep@title}
\newcommand{\newreptheorem}[2]{%
\newenvironment{rep#1}[1]{%
 \def\rep@title{#2 \ref{##1}}%
 \begin{rep@theorem}}%
 {\end{rep@theorem}}}
\makeatother

\newtheorem {theorem}{Theorem}
\newreptheorem{theorem}{Theorem}
\newtheorem {lemma}[theorem]{Lemma}
\newtheorem {proposition}[theorem]{Proposition}
\newtheorem {corollary}[theorem]{Corollary}
\newtheorem {conjecture}[theorem]{Conjecture}

\numberwithin{equation}{section}
\numberwithin{theorem}{section}

\theoremstyle{definition}

\newtheorem{remark}[theorem]{Remark}
\newtheorem*{remark*}{Remark}

\newcommand{\Z}{\mathbb{Z}}

\newcommand{\R}{\mathbb{R}}
\newcommand{\C}{\mathbb{C}}

\newcommand{\F}{\mathbb{F}}
\newcommand{\Q}{\mathbb{Q}}

\newcommand{\Hom}{\operatorname{Hom}}

\newcommand{\Img}{\operatorname{Im}}
\newcommand{\ind}{\operatorname{ind}}

\newcommand{\cM}{\mathcal{M}}

\newcommand{\cP}{\mathcal{P}}
\newcommand{\cR}{\mathcal{R}}
\newcommand{\cS}{\mathcal{S}}

\newcommand{\cX}{\mathcal{X}}

\newcommand{\RP}{\mathbb{RP}}

\DeclareMathOperator{\Span}{span}
\newcommand{\irr}{\mathrm{irr}}

\newcommand\SU{\mathrm{SU}}
\newcommand\SL{\mathrm{SL}}
\newcommand\PSL{\mathrm{PSL}}
\newcommand\SO{\mathrm{SO}}

\newcommand{\ad}{\operatorname{ad}}

\newcommand{\coker}{\operatorname{coker}}
\newcommand{\pt}{\mathrm{pt}}
\newcommand{\kbsm}{\cS}

\makeatletter
\DeclareFontFamily{OMX}{MnSymbolE}{}
\DeclareSymbolFont{MnLargeSymbols}{OMX}{MnSymbolE}{m}{n}
\SetSymbolFont{MnLargeSymbols}{bold}{OMX}{MnSymbolE}{b}{n}
\DeclareFontShape{OMX}{MnSymbolE}{m}{n}{
    <-6>  MnSymbolE5
   <6-7>  MnSymbolE6
   <7-8>  MnSymbolE7
   <8-9>  MnSymbolE8
   <9-10> MnSymbolE9
  <10-12> MnSymbolE10
  <12->   MnSymbolE12
}{}
\DeclareFontShape{OMX}{MnSymbolE}{b}{n}{
    <-6>  MnSymbolE-Bold5
   <6-7>  MnSymbolE-Bold6
   <7-8>  MnSymbolE-Bold7
   <8-9>  MnSymbolE-Bold8
   <9-10> MnSymbolE-Bold9
  <10-12> MnSymbolE-Bold10
  <12->   MnSymbolE-Bold12
}{}

\let\llangle\@undefined
\let\rrangle\@undefined
\DeclareMathDelimiter{\llangle}{\mathopen}%
                     {MnLargeSymbols}{'164}{MnLargeSymbols}{'164}
\DeclareMathDelimiter{\rrangle}{\mathclose}%
                     {MnLargeSymbols}{'171}{MnLargeSymbols}{'171}
\makeatother

\newcounter{desccount}

\newcommand{\descref}[1]{\hyperref[#1]{#1}}

\usetikzlibrary{calc,intersections}
\tikzset{every picture/.style=thick}
\tikzset{link/.style = { white, double = black, line width = 1.75pt, double distance = 1.25pt, looseness=1.75 }}
\tikzset{crossing/.style = {draw, circle, dotted, minimum size=0.5cm, inner sep=0, outer sep=0}}
\pgfplotsset{compat=1.12}

\begin{document}

\begin{abstract}
We use instanton gauge theory to prove that if $Y$ is a closed, orientable $3$-manifold such that $H_1(Y;\Z)$ is nontrivial and either $2$-torsion or $3$-torsion, and if $Y$ is neither $\#^r \RP^3$ for some $r\geq 1$ nor $\pm L(3,1)$, then there is an irreducible representation $\pi_1(Y) \to \SL(2,\C)$.  We apply this to show that the Kauffman bracket skein module of a non-prime 3-manifold has nontrivial torsion whenever two of the prime summands are different from $\RP^3$, answering a conjecture of Przytycki (Kirby problem 1.92(F)) unless every summand but one is $\RP^3$.  As part of the proof in the $2$-torsion case, we also show that if $M$ is a compact, orientable $3$-manifold with torus boundary whose rational longitude has order 2 in $H_1(M)$, then $M$ admits a degree-1 map onto the twisted $I$-bundle over the Klein bottle.
\end{abstract}

\maketitle

\section{Introduction}

Given a 3-manifold $Y$, the $\SL(2,\C)$ character variety of the fundamental group $\pi_1(Y)$ contains a lot of information about the geometry and topology of $Y$.  For example, one can use the characters of irreducible representations to understand something about the hyperbolic structure on $Y$, if it exists, or to find incompressible surfaces in $Y$.  Before doing so, however, it is natural to ask whether there are any irreducible representations in the first place.  The third author recently used instanton gauge theory to say that in many cases, the answer is yes.

\begin{theorem}[{\cite[Theorem~9.4]{zentner}}] \label{thm:zentner-main}
Let $Y$ be an integer homology 3-sphere.  If $Y$ is not homeomorphic to $S^3$, then there is an irreducible representation $\pi_1(Y) \to \SL(2,\C)$.
\end{theorem}

The main results of this paper extend Theorem~\ref{thm:zentner-main} to all manifolds $Y$ for which $H_1(Y;\Z)$ is either $2$-torsion or $3$-torsion.

\begin{theorem} \label{thm:sl2-abelian-main}
Let $Y$ be a closed, orientable 3-manifold with $H_1(Y;\Z) \cong (\Z/2\Z)^r$ for some integer $r \geq 1$.  If $Y$ is not homeomorphic to $\#^r \RP^3$, then there is an irreducible representation $\pi_1(Y) \to \SL(2,\C)$.
\end{theorem}

\begin{theorem} \label{thm:3-torsion-main}
Let $Y$ be a closed $3$-manifold such that $H_1(Y;\Z) \cong (\Z/3\Z)^r$ for some $r \geq 1$.  If $Y$ is not homeomorphic to $\pm L(3,1)$, then there is an irreducible representation $\pi_1(Y) \to \SL(2,\C)$.
\end{theorem}

Although we will mostly describe applications of Theorem~\ref{thm:sl2-abelian-main}, it turns out that Theorem~\ref{thm:3-torsion-main} is slightly easier to prove.  In fact, the analogous result when $H_1(Y;\Z)$ is $p$-torsion for an odd prime $p$ follows from the case where $H_1(Y;\Z)$ is cyclic of order $p$; this is detailed in Theorems~\ref{thm:mod-p-to-p-torsion} and \ref{thm:p-torsion-strong}.  However, it is not always true that if $H_1(Y;\Z) \cong \Z/p\Z$, then either $Y$ is a lens space or there must be a representation $\pi_1(Y) \to \SL(2,\C)$ with non-abelian image; a construction due to Motegi \cite{motegi} (see Remark~\ref{rem:order-37}) provides counterexamples for many primes, starting with $p=37$.

In the following subsections we will provide some applications of Theorem~\ref{thm:sl2-abelian-main} to character varieties and skein modules of reducible 3-manifolds, and then we will give an outline of its proof.

\subsection{$\SL(2,\C)$ character varieties}

Given a 3-manifold $Y$, we can define its $\SL(2,\C)$ representation variety to be
\[ \cR(Y) = \Hom(\pi_1(Y), \SL(2,\C)). \]
We will say that $Y$ is \emph{$\SL(2,\C)$-reducible} if every $\rho \in \cR(Y)$ is reducible, or \emph{$\SL(2,\C)$-abelian} if every $\rho \in \cR(Y)$ has abelian image.  If $Y$ is $\SL(2,\C)$-abelian then it is $\SL(2,\C)$-reducible, though the converse need not be true.

The representation variety $\cR(Y)$ carries an action of $\SL(2,\C)$ by conjugation, and the $\SL(2,\C)$ character variety of $Y$ is the GIT quotient
\[ \cX(Y) = \cR(Y)\ /\!/\ \SL(2,\C). \] 
Culler and Shalen \cite{cs-splittings} showed that one can use ideal points of curves in $\cX(Y)$ to find incompressible surfaces in $Y$.  In the opposite direction, one can ask whether the existence of incompressible surfaces in $Y$ forces $\dim_\C \cX(Y)$ to be positive, and Motegi \cite{motegi} showed that this is not always true, but for essential spheres we have the following.

\begin{proposition} \label{prop:non-prime-curve}
Suppose that for $i=1,2$ there are representations
\[ \rho_i: \pi_1(Y_i) \to \SL(2,\C) \]
whose images are not central (i.e., not contained in $\{\pm1\}$).  Then $\dim_\C \cX(Y_1\#Y_2)$ is positive.
\end{proposition}

\begin{proof}
We write $\pi_1(Y_1\#Y_2) \cong \pi_1(Y_1) \ast \pi_1(Y_2)$ and consider the map $\SL(2,\C) \to \cR(Y_1\#Y_2)$ given by $A \mapsto \rho_1 \ast (A\rho_2A^{-1})$.  This has positive-dimensional image, even in the quotient $\cX(Y_1\#Y_2)$.
\end{proof}

Combining this observation with Theorem~\ref{thm:sl2-abelian-main}, we readily deduce the following.

\begin{theorem} \label{thm:dim-x-positive}
If $Y_1$ and $Y_2$ are closed, oriented 3-manifolds, and neither $Y_1$ nor $Y_2$ is homeomorphic to $\#^r \RP^3$ for any $r \geq 0$, then $\dim_\C \cX(Y_1\#Y_2)$ is positive.
\end{theorem}

\begin{proof}
If $Y_i \not\cong \#^r \RP^3$ for any $r$, then we can always find a representation $\rho_i: \pi_1(Y_i) \to \SL(2,\C)$ with non-central image.  Indeed, if $H_1(Y_i;\Z)$ is $2$-torsion then Theorem~\ref{thm:sl2-abelian-main} applies; if it is not $2$-torsion, then we can take $\rho_i$ to factor through $H_1(Y_i;\Z)$ and send a summand of the form $\Z$ or $\Z/n\Z$ ($n \geq 3$) to a non-central subgroup of $\SL(2,\C)$.  Now we apply Proposition~\ref{prop:non-prime-curve}.
\end{proof}

We remark that the condition $Y_i \not\cong \#^r \RP^3$ is necessary in Theorem~\ref{thm:dim-x-positive}, because we have
\[ \cX(Y\#\RP^3) \cong \cX(Y) \times \{\pm 1\} \]
and so taking connected sums with $\RP^3$ cannot change the dimension of $\cX(Y)$.

\subsection{Skein modules}

The Kauffman bracket skein module, defined by Przytycki \cite{przytycki-kbsm} and Turaev \cite{turaev-kbsm}, is a $\Z[A^{\pm1}]$-module $\kbsm(Y)$ associated to any oriented 3-manifold $Y$.  Relatively little is known about the structure of this invariant in general; it was only recently proved by Gunningham, Jordan, and Safronov \cite{gunningham-jordan-safronov} that if $Y$ is a closed, oriented 3-manifold then $\kbsm(Y)$ is finite-dimensional over $\Q(A)$.  Przytycki conjectured the following.

\begin{conjecture}[{\cite[Problem~1.92(F)]{kirby-list}}] \label{conj:kbsm-torsion}
If $Y \cong Y_1 \# Y_2$, where neither of the $Y_i$ is homeomorphic to $S^3$ with some number of disjoint balls removed, then $\kbsm(Y)$ has non-trivial torsion.
\end{conjecture}

By contrast, we know that $\kbsm(S^3) \cong \Z[A^{\pm1}]$ is freely generated by the empty link \cite[Theorem~12]{przytycki-kbsm}, while for lens spaces Hoste and Przytycki \cite{hoste-przytycki-lens} showed that $\kbsm(L(p,q))$ is a free module on $\lfloor p/2\rfloor+1$ generators.  (On the other hand, a \emph{non-separating} $S^2$ always leads to torsion in $\kbsm(Y)$, by a version of Dirac's belt trick \cite[\S4]{przytycki-fundamentals}.)  We remark that removing a ball from $Y$, or conversely filling in an $S^2$ component of $\partial Y$ with a ball, does not change $\kbsm(Y)$ up to isomorphism \cite[Proposition 4(d)]{przytycki-kbsm}.

We note the relevance of Theorem~\ref{thm:dim-x-positive} to Conjecture~\ref{conj:kbsm-torsion} via work of Bullock \cite[Corollary~1]{bullock}, who showed that if $\dim_\C \cX(Y) \geq 1$ then $\kbsm(Y)$ is infinitely generated.  Indeed, Przytycki \cite[Theorem~4.2(b)]{przytycki-fundamentals} proved that Conjecture~\ref{conj:kbsm-torsion} holds for a connected sum $Y=Y_1\#Y_2$ if for each $i$, there is a representation 
\[ \rho_i: \pi_1(Y_i) \to \SL(2,\C) \]
with non-central image.  (See also \cite[Theorem~4.4]{przytycki-algtop}.)  We have therefore shown the following, exactly as in Theorem~\ref{thm:dim-x-positive}.

\begin{theorem} \label{thm:skein-connected-sum}
Let $Y$ be an oriented $3$-manifold, and suppose that we can write
\[ Y \cong Y_1 \# Y_2 \]
where neither $Y_1$ nor $Y_2$ is homeomorphic to some $\#^r \RP^3$ ($r\geq 0$) minus a disjoint union of balls.  Then $\kbsm(Y)$ has non-trivial torsion.
\end{theorem}

In particular, the following conjecture would now imply Conjecture~\ref{conj:kbsm-torsion}.

\begin{conjecture} \label{conj:kbsm-rp3}
Suppose that $Y$ is a closed oriented $3$-manifold that is not homeomorphic to $S^3$.  Then $\kbsm(Y\#\RP^3)$ has non-trivial torsion.
\end{conjecture}

We note that at least the case $Y = \RP^3$ of Conjecture~\ref{conj:kbsm-rp3} is known: the Kauffman bracket skein module of $\RP^3\#\RP^3$ was completely determined by Mroczkowski \cite{mroczkowski-rp3}, who showed in \cite[Proposition~4.17]{mroczkowski-rp3} that $\kbsm(\RP^3\#\RP^3)$ does contain torsion.

\subsection{Outline of the proof of Theorem~\ref{thm:sl2-abelian-main}}

Just as for $\SL(2,\C)$, we will say that $Y$ is \emph{$\SU(2)$-abelian} if every $\rho: \pi_1(Y) \to \SU(2)$ has abelian image; in contrast with the $\SL(2,\C)$ case, this is the same as being $\SU(2)$-reducible.  We will use gauge theory to show that many of the 3-manifolds under consideration are not $\SU(2)$-abelian, which means that they are not $\SU(2)$-reducible and hence not $\SL(2,\C)$-reducible either.

With this in mind, we let $Y$ be an $\SL(2,\C)$-reducible $3$-manifold, and we suppose that $H_1(Y;\Z) \cong (\Z/2\Z)^r$ for some $r \geq 0$.  We can assume without loss of generality that $Y$ is prime, since otherwise each of its summands also is $\SL(2,\C)$-reducible with $2$-torsion homology.

Theorem~\ref{thm:sl2-abelian-main} follows quickly for several large classes of $3$-manifolds: Thurston proved that closed hyperbolic 3-manifolds are never $\SL(2,\C)$-reducible \cite[Proposition~3.1.1]{cs-splittings}, and among the prime Seifert fibered $3$-manifolds with $2$-torsion homology, work of the second and third author \cite{sz-menagerie} implies that among these only $\RP^3$ is $\SU(2)$-abelian.  Thus if $Y \not\cong \RP^3$ is $\SL(2,\C)$-reducible, then we use the geometrization theorem to conclude that $Y$ contains an incompressible torus, and this torus must be separating since $b_1(Y) = 0$. 

We now decompose $Y$ along this torus $T$, writing
\[ Y = M_1 \cup_T M_2 \]
where each $M_i$ is compact and irreducible with incompressible torus boundary.  Then we can write the $\SU(2)$ representation variety of $Y$ as a fiber product
\[ R(Y) = R(M_1) \times_{R(T)} R(M_2), \]
so it suffices to find representations $\rho_j: \pi_1(M_j) \to \SU(2)$ for $j=1,2$ whose restrictions to $\pi_1(T)$ coincide.  In fact, we need only find these up to conjugation, so we end up studying the images
\[ i_j^*: X(M_j) \to X(T) \]
of the respective $\SU(2)$ character varieties in the $\SU(2)$ character variety of the torus, known as the \emph{pillowcase} orbifold.  We will generally aim to show that these images intersect, since the points of intersection correspond to representations $\pi_1(Y) \to \SU(2)$; if we know that one of the images at such a point corresponds to an irreducible representation of $\pi_1(M_j)$, then the representation of $\pi_1(Y)$ will also be irreducible, as desired.

Each $M_i$ comes equipped with a distinguished peripheral curve up to orientation, namely the \emph{rational longitude} $\lambda_i$: this generates the kernel of the inclusion map $H_1(\partial M_j; \Q) \to H_1(M_j;\Q)$, but may be either zero or torsion in $H_1(M_j;\Z)$.  If $\lambda_2$ is nullhomologous in $M_2$ then there is a standard degree-1 map that pinches $M_2$ onto a solid torus (see Proposition~\ref{prop:solid-torus-pinch}), and hence there is a degree-1 map
\[ Y \to M_1(\lambda_2) \]
onto the Dehn filling of $M_1$ along the slope $\lambda_2 \subset T$.  This induces a surjection $\pi_1(Y) \to \pi_1(M_1(\lambda_2))$, so if $Y$ is $\SL(2,\C)$-reducible then $M_1(\lambda_2)$ must be as well.  Similarly, if $[\lambda_1] = 0$ in $H_1(M_1;\Z)$ then we deduce that $M_2(\lambda_1)$ is also $\SL(2,\C)$-reducible.

By choosing an appropriate Dehn filling of $M_j$ we may write it as the complement of a nullhomologous knot $K_j$ in a closed $3$-manifold $Y_j$, with meridian $\mu_j \subset \partial M_j$, such that each $H_1(Y_j;\Z)$ is $2$-torsion and one of the following applies:
\begin{enumerate}
\item both of the $K_j$ are nullhomologous, with longitudes $\lambda_j$, and we glue $\partial M_1 \xrightarrow{\cong} \partial M_2$ so that
\begin{enumerate}
\item $\mu_1 \sim \lambda_2$ and $\lambda_1 \sim \mu_2$, or \label{eq:intro-case-1a}
\item $\mu_1 \sim \mu_2^{-1}$ and $\lambda_1 \sim \mu_2^2 \lambda_2$; \label{eq:intro-case-1b}
\end{enumerate}
\item or without loss of generality $[\lambda_1]$ is $2$-torsion in $H_1(M_1;\Z)$, and then $[\lambda_2] = 0$ in $H_1(M_2;\Z)$.  In this case we glue $\partial M_1 \xrightarrow{\cong} \partial M_2$ so that $\mu_1 \sim \lambda_2$ and $\lambda_1 \sim \mu_2$. \label{eq:intro-case-2}
\end{enumerate}
(This list is shown to be exhaustive in \S\ref{ssec:basis-inessential} and \S\ref{ssec:basis-essential}.)  Cases \ref{eq:intro-case-1a} and \ref{eq:intro-case-2} are handled similarly, so we will summarize the arguments in cases \ref{eq:intro-case-1a}, \ref{eq:intro-case-2}, and \ref{eq:intro-case-1b} in that order below.

\vspace{1em}\noindent
{\bf Case~\ref{eq:intro-case-1a}} (Theorem~\ref{thm:splice-rep}): The above discussion says that each of
\begin{align*}
Y_1 &= M_1(\mu_1) = M_1(\lambda_2), \\
Y_2 &= M_2(\mu_2) = M_2(\lambda_1)
\end{align*}
is $\SL(2,\C)$-reducible.  In particular we can follow work of Lidman, Pinz\'on-Caicedo, and the third author \cite{lpcz} essentially verbatim to construct an irreducible representation $\rho: \pi_1(Y) \to \SU(2)$, giving a contradiction.  The rough idea is that by using work of \cite{zentner}, we know that each of the images $i_j^*\big(X(M_j)\big)$ must contain a closed essential curve in the twice-punctured pillowcase, and the gluing map guarantees that these two curves will intersect.

The only change from \cite{lpcz} is that we replace Floer's instanton homology for homology 3-spheres with the \emph{irreducible instanton homology} of each $Y_j$.  This invariant is generated as a complex by gauge equivalence classes of irreducible flat connections on the trivial $\SU(2)$-bundle $P \to Y_j$, and the theory works in exactly the same way when $H_1(Y_j;\Z)$ is $2$-torsion, because the reducible flat connections on $P$ all have central holonomy.  See \S\ref{sec:instanton} for further discussion.

\vspace{1em}\noindent
{\bf Case~\ref{eq:intro-case-2}} (Theorem~\ref{thm:splice-essential-rep}): This is similar to case~\ref{eq:intro-case-2}, but a priori we do not know that $Y_2$ is $\SL(2,\C)$-abelian: we cannot pinch $M_1$ onto a solid torus, because the class $[\lambda_1] \in H_1(M_1;\Z)$ is $2$-torsion rather than zero.  In \S\ref{sec:pinching}, we construct a replacement that should be of independent interest.
\begin{proposition}[Proposition~\ref{prop:klein-bottle-pinch}] \label{prop:main-klein-bottle-pinch}
Let $M$ be a compact, oriented $3$-manifold with torus boundary, and suppose that the rational longitude $\lambda_M \subset \partial M$ has order $2$ in $H_1(M)$.  Then there is a degree-1 map
\[ f: M \to N, \]
where $N$ is the twisted $I$-bundle over the Klein bottle, such that $f$ restricts to a homeomorphism $\partial M \to \partial N$ sending $\lambda_M$ to a rational longitude $\lambda_N \subset \partial N$.
\end{proposition}
\noindent Using Proposition~\ref{prop:main-klein-bottle-pinch}, we see that $N \cup_T M_2$ is $\SU(2)$-abelian if $Y$ is; this is enough to deduce that $Y_2$ is $\SU(2)$-abelian and understand the image $i_2^*\big(X(M_2)\big) \subset X(T)$ exactly as in case~\ref{eq:intro-case-1a}.  This leads us to an irreducible $\SU(2)$ representation of $\pi_1(N \cup_T M_2)$, and hence of $\pi_1(Y)$.

\vspace{1em}\noindent
{\bf Case~\ref{eq:intro-case-1b}} (Theorem~\ref{thm:gluing-rep}): Here the $\lambda_j$ are both nullhomologous again, but the analogous degree-1 maps from $Y$ have targets $(Y_j)_2(K_j)$ rather than $Y_j$.  This means that the $2$-surgeries on $K_j$ are $\SL(2,\C)$-reducible, and if one of them is toroidal then we may replace $Y$ with it and repeat.  We apply the following theorem of Rong \cite{rong} to say that this process must terminate after finitely many iterations:
\begin{theorem}[{\cite[Theorem~3.9]{rong}}] \label{thm:rong}
Suppose we have an infinite sequence of closed, oriented $3$-manifolds and degree-1 maps between them, of the form
\[ M_1 \xrightarrow{f_1} M_2 \xrightarrow{f_2} M_3 \xrightarrow{f_3} \cdots. \]
Then the map $f_i$ is a homotopy equivalence for all sufficiently large $i$.
\end{theorem}
\noindent
(We omit Rong's hypothesis that the $M_i$ belong to a set $\mathscr{G}_c$ of $3$-manifolds satisfying the geometrization conjecture, as this is now a theorem.)  Thus we can freely assume that the $\SL(2,\C)$-reducible $2$-surgeries are atoroidal.

Now we must have $(Y_j)_2(K_j) \cong \#^{n_j} \RP^3$ for some $n_j$.  With this simplification at hand, we prove in Theorem~\ref{thm:gluing-rep} that such $Y$ cannot be $\SU(2)$-abelian.  The key idea is to examine the subset
\[ R'_j = \{ \rho: \pi_1(Y_j \setminus N(K_j)) \to \SU(2) \mid \rho(\mu_j^2\lambda_j) = -1 \} \]
of each representation variety $R(M_j)$.  These $\rho$ do not descend to representations of
\[ \pi_1( (Y_j)_2(K_j) ) \cong \pi_1( \#^{n_j} \RP^3 ) \cong (\Z/2\Z)^{\ast n_j}, \]
but their \emph{adjoint} representations do, and we can understand the representation variety $\Hom( (\Z/2\Z)^{\ast n_j}, \SO(3) )$ explicitly enough to see that each path component of $R'_j$ contains an abelian representation.  This tells us in Proposition~\ref{prop:lines-of-slope-2} that for each $j$, the image $i_j^*\big(X(M_j)\big) \subset X(T)$ meets the line corresponding to the condition $\rho(\mu_j^2\lambda_j) = -1$ in a \emph{connected} arc.  It also contains an essential closed curve as in the previous cases, as well as the image of this curve under an involution of $X(T)$.  All of this ensures that the images $i_j^*(X(M_j))$ are too large to avoid each other, and where they intersect we get an irreducible representation after all, completing the proof.

\begin{remark}
If $Y$ is toroidal and $H_1(Y;\Z)$ is $2$-torsion then one might expect there to be an irreducible representation $\pi_1(Y) \to \SU(2)$, as shown for homology spheres in \cite{lpcz,bs-splicing}, but we do not prove this here.  The issue is that in case~\ref{eq:intro-case-1b}, we only get an $\SU(2)$ representation once we have reduced to the case where $\pi_1( (Y_j)_2(K_j) )$ is generated by elements of order $2$.  The reduction process gets stuck if $(Y_j)_2(K_j)$ is hyperbolic: in this case the degree-1 map $Y \to (Y_j)_2(K_j)$ tells us that $Y$ is not $\SL(2,\C)$-reducible and we stop there, but we cannot conclude that $Y$ is not $\SU(2)$-abelian because we do not know whether $(Y_j)_2(K_j)$ is.
\end{remark}

\begin{remark}
The proof of Theorem~\ref{thm:3-torsion-main}, carried out in Sections~\ref{sec:p-torsion} and \ref{sec:3-torsion}, is similar enough to that of Theorem~\ref{thm:sl2-abelian-main}, so we will not outline it here.  We note only that it can be reduced to an analogue of Case~\ref{eq:intro-case-1b}, namely Theorem~\ref{thm:glue-slope-3}, whose analysis is simpler because (unlike $\#^r \RP^3$) a nontrivial connected sum of order-3 lens spaces is not $\SL(2,\C)$-reducible.
\end{remark}

\subsection*{Organization}

In Section~\ref{sec:instanton} we discuss the needed background from instanton Floer homology, including some non-vanishing results; this includes a generalization of the usual surgery exact triangle, Theorem~\ref{thm:exact-triangle-2-torsion}, the details of which we postpone to Appendix~\ref{sec:exact-triangle-proof}.  Then in Section~\ref{sec:pillowcase} we use this to investigate the $\SU(2)$ character varieties of knot complements and their images in the \emph{pillowcase}, the $\SU(2)$ character variety of $T^2$.  In Section~\ref{sec:pinching} we study the twisted $I$-bundle over the Klein bottle in depth, and construct ``pinch'' maps  onto it from knot manifolds with rational longitudes of order 2 (Proposition~\ref{prop:klein-bottle-pinch}).

Building on this, Sections~\ref{sec:splice-torsion} and \ref{sec:glue-zhs} prove the existence of irreducible $\SU(2)$-representations for various toroidal 3-manifolds built by gluing together the complements of knots in $3$-manifolds whose homology is $2$-torsion; these are Theorems~\ref{thm:splice-rep}, \ref{thm:splice-essential-rep}, and \ref{thm:gluing-rep}, respectively.  Notably, in Subsection~\ref{ssec:splice-essential} we apply the pinch maps of Proposition~\ref{prop:klein-bottle-pinch} to study the case where one of the knots has a homologically essential rational longitude.

Finally, in Section~\ref{sec:knot-exteriors} we prove that if $Y$ is a toroidal $3$-manifold whose first homology is $p$-torsion for some prime $p$, then $Y$ can be decomposed into a union of knot complements in one of a few standard ways; when $p=2$ these are precisely the forms studied in Sections~\ref{sec:splice-torsion} and \ref{sec:glue-zhs}.  Thus allows us to complete the proof of Theorem~\ref{thm:sl2-abelian-main}, which we do in Section~\ref{sec:main-proof}.  In Section~\ref{sec:p-torsion} we study the case where $H_1(Y)$ is instead $p$-torsion for some odd prime $p > 2$, and we conclude by applying this in Section~\ref{sec:3-torsion} to prove Theorem~\ref{thm:3-torsion-main}.

\subsection*{Acknowledgments}

We thank Ali Daemi, Tye Lidman, and Mike Miller Eismeier for helpful conversations about the irreducible instanton homology of $3$-manifolds $Y$ such that $H_1(Y)$ is $2$-torsion.  We also thank Rhea Palak Bakshi and Renaud Detcherry for discussions about the relation between character varieties and torsion in skein modules.  We are grateful to the Max Planck Institute for Mathematics for hosting all three of us for the bulk of this work.

\section{Instanton Floer homology} \label{sec:instanton}

Let $I_*(Y)$ denote Floer's instanton homology \cite{floer-instanton}, associated to any integer homology 3-sphere, and $I^w_*(Y)$ the variant assigned to a Hermitian line bundle $w \to Y$ such that $c_1(w)$ has odd evaluation on some homology class.  (This means that we fix a $\mathrm{U}(2)$-bundle $E \to Y$ and an isomorphism $\wedge^2E \cong w$, and let $I^w_*(Y)$ be the $\SO(3)$ instanton homology of the admissible bundle $\ad(E) \to Y$, following \cite[\S5.6]{donaldson-book}.)  Kronheimer and Mrowka proved the following important result en route to their proof of the property P conjecture.

\begin{theorem}[\cite{km-p}] \label{thm:km-zero-surgery}
Let $K \subset S^3$ be a nontrivial knot.  Then $I^w_*(S^3_0(K)) \neq 0$.
\end{theorem}

Here, we cap off a Seifert surface $\Sigma$ for $K$ to get a closed surface $\hat\Sigma$ generating $H_1(Y_0(K)) \cong \Z$, and then we let $w \to Y_0(K)$ be the unique non-trivial line bundle with $\langle c_1(w), [\hat\Sigma]\rangle = 1$.  (In general we will write $w$ for both the line bundle and its first Chern class, and this should not cause any confusion.)  In fact, we have the following more general nonvanishing result.

\begin{theorem}[{\cite[Theorem~7.21]{km-excision}}] \label{thm:irreducible-nonzero}
Let $Y$ be an irreducible $3$-manifold, and suppose that there is a line bundle $w \to Y$ and an embedded surface $R \subset Y$ such that $w\cdot R = 1$.  Then $I^w_*(Y) \neq 0$.
\end{theorem}

Theorem~\ref{thm:km-zero-surgery} was crucial in the proof of Theorem~\ref{thm:zentner-main} in \cite{zentner}.  Lidman, Pinz\'on-Caicedo, and Zentner \cite[Theorem~1.3]{lpcz} used Theorem~\ref{thm:irreducible-nonzero} to generalize this to knots $K \subset Y$ with irreducible, boundary-incompressible exterior, proving for such $K$ that if $Y$ is an integer homology sphere with $I_*(Y) = 0$ then $I^w_*(Y_0(K)) \neq 0$, and they used this to produce $\SU(2)$ representations for toroidal homology spheres.  We would like to further generalize it to other $Y$ in order to prove Theorem~\ref{thm:sl2-abelian-main}, but the problem is that $I_*(Y)$ does not make sense unless $Y$ is a homology sphere.

In this section we will discuss a version of Floer's instanton homology for 3-manifolds whose first homology is $2$-torsion.  This does not seem to have appeared explicitly in the literature in this form, but we make no claim of originality here; these ideas appear in recent work of Daemi, Lidman, and Miller Eismeier \cite[\S2.1]{dlme}, and have been elaborated on in much greater detail by Daemi and Miller Eismeier \cite{daemi-miller-eismeier} under the name of \emph{irreducible} instanton homology.

By way of motivation, Floer originally defined $I_*(Y)$ for a homology sphere $Y$ \cite{floer-instanton} in terms of a chain complex generated by gauge equivalence classes of \emph{irreducible} flat connections on the trivial $\SU(2)$-bundle $P \to Y$.  This construction ignores the trivial connection $\theta$ completely, except as a way of lifting the relative $\Z/8\Z$ grading to an absolute one.  The reason that we can safely omit $\theta$ is that it has central holonomy.

As an example of why this matters, we recall that $d^2 = 0$ because the matrix coefficients $\langle d^2 a, b\rangle$ count pairs of rigid flowlines, meaning ASD connections on $\R\times P$ belonging to $0$-dimensional moduli spaces, of the form
\[ (A_1,A_2) \in \hat\cM_0(a,c) \times \hat\cM_0(c,b) \]
as $c$ ranges over generators of the chain complex.  This count is meant to equal the number of points in the boundary of the compactification of a 1-dimensional moduli space $\hat\cM_1(a,b)$, which is zero, and indeed this is the case as long as no sequence in $\hat\cM_1(a,b)$ limits to a broken flowline $a \to \theta \to b$ that breaks at the omitted connection $\theta$.  We rule out this problematic case by observing that rigid flowlines $a \to \theta$ and $\theta \to b$ can only be glued into a moduli space $\hat\cM(a,b)$ of dimension at least $4$, because the gluing map in this case involves an extra $\SO(3)$ factor coming from the isotropy group of $\theta$.  The proof that $I_*(Y)$ is an invariant similarly only relies on the fact that $\theta$ is central.

Having said this, we can define $I_*(Y)$ for manifolds with $H_1(Y;\Z) \cong (\Z/2\Z)^r$ for some $r \geq 0$, by repeating the material in \cite[\S5]{donaldson-book} essentially verbatim.  We need only observe that all of the reducible flat connections on $P$ have central holonomy, due to the fact that every homomorphism $(\Z/2\Z)^r \to \SU(2)$ has image in the center $\{\pm1\}$.

\begin{theorem} \label{thm:instanton-2-torsion}
Let $Y$ be a closed, oriented rational homology 3-sphere, and suppose that $H_1(Y;\Z)$ is $2$-torsion.  Then there is an \emph{irreducible instanton homology} group
\[ I_*(Y) \]
defined exactly as in \cite{floer-instanton,donaldson-book}.  It is the homology of a chain complex whose generators are gauge equivalence classes of \emph{irreducible} flat connections on the trivial $\SU(2)$-bundle $P\to Y$, and its differential counts anti-self-dual connections on the product $\R\times P \to \R \times Y$.
\end{theorem}

\begin{remark}
For other rational homology spheres the story is much more complicated, and we will say nothing more about it here.  See \cite[\S7]{daemi-miller-eismeier} for details.
\end{remark}

The key property we will need from irreducible instanton homology is a surgery exact triangle, which goes back to Floer \cite{floer-surgery,braam-donaldson} for the case of knots in homology spheres.

\begin{theorem} \label{thm:exact-triangle-2-torsion}
Let $Y$ be a closed, oriented 3-manifold such that $H_1(Y;\Z)$ is $2$-torsion, and let $K \subset Y$ be a nullhomologous knot.  Then there is an exact triangle
\[ \cdots \to I_*(Y) \to I^w_*(Y_0(K)) \to I_*(Y_1(K)) \to \cdots, \]
where the Hermitian line bundle $w \to Y_0(K)$ has $c_1(w)$ Poincar\'e dual to a meridian of $K$.
\end{theorem}

We note in Theorem~\ref{thm:exact-triangle-2-torsion} that $H_1(Y) \cong H_1(Y_1(K))$ is $2$-torsion, so Theorem~\ref{thm:instanton-2-torsion} says that all of the groups in the exact triangle are well-defined.  The proof of Theorem~\ref{thm:exact-triangle-2-torsion} follows an argument given by Scaduto in \cite{scaduto}, after one checks that the relevant compactifications of moduli spaces do not include broken flowlines with reducible connections in the middle.  We discuss the details in Appendix~\ref{sec:exact-triangle-proof}.

\begin{theorem} \label{thm:exact-triangle-1/n}
Let $Y$ be a closed, oriented $3$-manifold such that $H_1(Y;\Z)$ is $2$-torsion, and let $K \subset Y$ be a nullhomologous knot.  For any $n\in\Z$, there is an exact triangle
\[ \cdots I_*(Y_{1/n}(K)) \to I^w_*(Y_0(K)) \to I_*(Y_{1/(n+1)}(K)) \to \cdots, \]
where the Hermitian line bundle $w \to Y_0(K)$ has $c_1(w)$ Poincar\'e dual to a meridian of $K$.
\end{theorem}

\begin{proof}
We let $Y' = Y_{1/n}(K)$, with $K' \subset Y'$ the core of this surgery.  Then $H_1(Y') \cong H_1(Y)$ is $2$-torsion, and $1$-surgery on $K'$ is the same as Dehn filling the exterior of $K'$ along the curve $\mu_{K'}\lambda_{K'} = (\mu_K\lambda_K^n)(\lambda_K) = \mu_K\lambda_K^{n+1}$, which produces $Y_{1/(n+1)}(K)$.  The desired triangle is thus the result of applying Theorem~\ref{thm:exact-triangle-2-torsion} to the pair $(Y',K')$.
\end{proof}

With all of this at hand, we can now provide the desired generalization of \cite[Theorem~1.3]{lpcz}.

\begin{theorem} \label{thm:2-torsion-zero-surgery}
Let $Y$ be a closed, orientable, $\SU(2)$-abelian $3$-manifold, and suppose that $H_1(Y;\Z)$ is $2$-torsion.  Let $K \subset Y$ be a nullhomologous knot with irreducible, boundary-incompressible exterior.  Then $I^w_*(Y_0(K)) \neq 0$, where $w$ is Poincar\'e dual to a meridian of $K$.
\end{theorem}

\begin{proof}
We repeat the proof of \cite[Theorem~1.3]{lpcz} verbatim, including the details here for convenience.  Since $Y$ is $\SU(2)$-abelian, there are no irreducible flat connections on the product $\SU(2)$-bundle over $Y$, so $I_*(Y) = 0$.  Supposing that $I^w_*(Y_0(K)) = 0$ as well, we apply Theorem~\ref{thm:exact-triangle-2-torsion} to get $I_*(Y_1(K)) = 0$, and then Theorem~\ref{thm:exact-triangle-1/n} with $n=1,2,3$ in succession to get
\[ I_*(Y_{1/2}(K)) = I_*(Y_{1/3}(K)) = I_*(Y_{1/4}(K)) = 0. \]
Then Gordon \cite[Corollary~7.3]{gordon} showed that $Y_{1/4}(K) \cong Y_1(K_{2,1})$, where $K_{2,1}$ denotes the $(2,1)$-cable of $K$, so we apply Theorem~\ref{thm:exact-triangle-2-torsion} to get
\[ I^w_*(Y_0(K_{2,1})) = 0. \]
But $Y_0(K_{2,1})$ is irreducible, because it can be built by gluing two irreducible $3$-manifolds -- the exterior of $K$ and the $0$-surgery on the $(2,1)$-cable knot in $S^1\times D^2$ -- along their incompressible boundaries.  Thus $I^w_*(Y_0(K_{2,1}))$ is nonzero, by Theorem~\ref{thm:irreducible-nonzero}, and we have a contradiction.  We conclude that $I^w_*(Y_0(K)) \neq 0$ after all.
\end{proof}

\section{Closed curves in the pillowcase} \label{sec:pillowcase}

\subsection{The pillowcase}

Here we review basic facts about the pillowcase, following \cite[\S3.1--3.2]{lpcz}.  Given a manifold $Y$, we define its $\SU(2)$-representation variety 
\[ R(Y) = \Hom(\pi_1(Y), \SU(2)), \]
and let $R^\irr(Y)$ denote the subspace consisting of irreducible representations.  (We recall that an $\SU(2)$ representation is irreducible if and only if its image is non-abelian.)  These both carry an action of $\SU(2)$ by conjugation, and we define the character varieties
\begin{align*}
X(Y) &= R(Y) / \SU(2), \\
X^\irr(Y) &= R^\irr(Y) / \SU(2)
\end{align*}
as the quotients by this action. Note that we use a plain font ($R$, $X$) for the $\SU(2)$ representation and character varieties, in contrast to the calligraphic $\cR$ and $\cX$ for their $\SL(2,\C)$ counterparts.

If $K$ is a nullhomologous knot in a 3-manifold $Y$, with exterior $E_K = Y \setminus N(K)$, then the inclusion $i: \partial E(K) \hookrightarrow E(K)$ induces a map
\[ i^*: X(E_K) \to X(\partial E_K) \cong X(T^2). \]
Letting $\mu,\lambda$ be a meridian--longitude basis of $\pi_1(\partial E_K)$, every representation $\rho$ of either $\pi_1(E_K)$ or $\pi_1(T^2)$ is conjugate to one in which
\begin{align*}
\rho(\mu) &= \begin{pmatrix} e^{i\alpha} & 0 \\ 0 & e^{-i\alpha} \end{pmatrix}, &
\rho(\lambda) &= \begin{pmatrix} e^{i\beta} & 0 \\ 0 & e^{-i\beta} \end{pmatrix},
\end{align*}
for some $\alpha, \beta \in \R/2\pi\Z$, and these coordinates are almost unique: the only ambiguity is that the representations corresponding to $(\alpha,\beta)$ and $(-\alpha,-\beta)$ are conjugate to each other.  Thus the pair $\mu,\lambda$ leads to an identification
\[ X(T^2) = \frac{(\R/2\pi\Z) \times (\R/2\pi\Z)}{(\alpha,\beta) \sim (-\alpha,-\beta)}, \]
and this quotient orbifold is called the \emph{pillowcase}.  See Figure~\ref{fig:pillowcase-example} for an example.

\begin{figure}
\begin{tikzpicture}[style=thick]
\begin{scope}[xscale=1.5,yscale=2]
  \draw plot[mark=*,mark size = 0.5pt] coordinates {(0,0)(2,0)(2,3)(0,3)} -- cycle; 
  \begin{scope}[decoration={markings,mark=at position 0.55 with {\arrow[scale=1]{>>}}}]
    \draw[postaction={decorate}] (0,0) -- (0,1.5);
    \draw[postaction={decorate}] (0,3) -- (0,1.5);
  \end{scope}
  \begin{scope}[decoration={markings,mark=at position 0.575 with {\arrow[scale=1]{>>>}}}]
    \draw[postaction={decorate}] (2,0) -- (2,1.5);
    \draw[postaction={decorate}] (2,3) -- (2,1.5);
  \end{scope}
  \begin{scope}[decoration={markings,mark=at position 0.75 with {\arrow[scale=1]{>>>>}}}]
    \draw[postaction={decorate}] (2,3) -- (0,3);
    \draw[postaction={decorate}] (2,0) -- (0,0);
  \end{scope}
  \draw[dotted] (0,1.5) -- (2,1.5);
  \draw[thin,|-|] (0,-0.2) node[below] {\small$0$} -- node[midway,inner sep=1pt,fill=white] {$\alpha$} ++(2,0) node[below] {\small$\vphantom{0}\pi$};
  \draw[thin,|-|] (-0.2,0) node[left] {\small$0$} -- node[midway,inner sep=1pt,fill=white] {$\beta$} ++(0,3) node[left] {\small$2\pi$};

  \begin{scope}[color=blue, style=ultra thick]
    \draw (0,0) -- (2,0);
    \draw (1/3,3) -- (1,0) (1,3) -- (5/3,0);
  \end{scope}
  \node at (2.75,1.5) {\Large$\cong$};
\end{scope}
\begin{scope}[xscale=1.5,yscale=2,xshift=3.5cm]
  \draw (0,0.75) rectangle (2,2.25);
  \draw[very thin] (0,1.5) to[bend right=10] (2,1.5);
  \draw[very thin,dashed] (0,1.5) to[bend left=5] (2,1.5);
  \begin{scope}[color=blue, style=ultra thick]
    \draw[densely dotted] (1/3,0.75) -- (2/3,2.25);
    \draw (2/3,2.25) -- (1,0.75);
    \draw[densely dotted] (1,0.75) -- (4/3,2.25);
    \draw (4/3,2.25) -- (5/3,0.75);
    \draw (0,0.75) -- ++(2,0);
  \end{scope}
  \end{scope}
\end{tikzpicture}
\caption{The image $i^*(X(E_K))$ in the pillowcase, where $E_K$ is the exterior of the right-handed trefoil in $S^3$.}
\label{fig:pillowcase-example}
\end{figure}
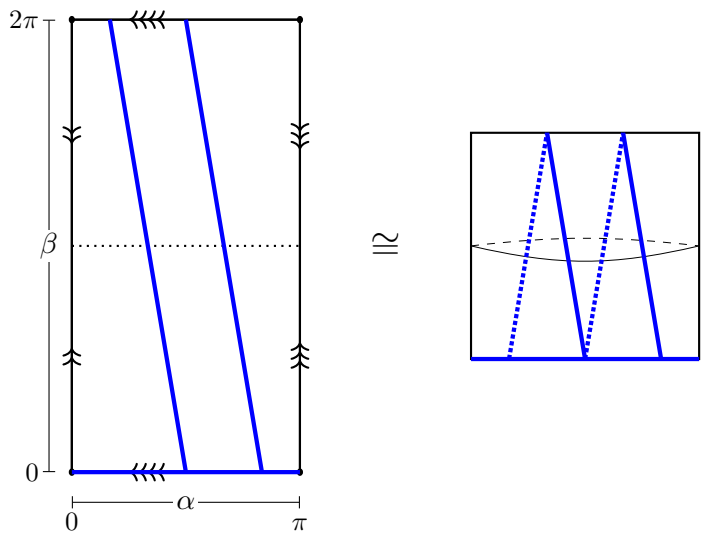

The following is one of the key technical results of \cite{zentner}, though it is only applied there to non-trivial knots in $S^3$.

\begin{proposition} \label{prop:curve-in-pillowcase}
Let $K$ be a nullhomologous knot in a 3-manifold $Y$, and let $w \in H^2(Y_0(K);\Z)$ be Poincar\'e dual to a meridian of $K$.  Suppose that $I^w_*(Y_0(K)) \neq 0$, and that the pillowcase image $i^*(X(E_K))$ does not contain the points
\[ P = (0,\pi), \ Q = (\pi,\pi) \in X(T^2). \]
Then there is a topologically embedded curve $C \subset i^*(X(E_K))$ that is homologically essential in
\[ X(T^2) \setminus \{P,Q\} \cong (0,1) \times S^1. \]
\end{proposition}

\begin{proof}
This is proved in \cite[\S7]{zentner}; we sketch the argument here.  We first observe that $I^w_*(Y_0(K))$ is generated as a chain complex by gauge equivalence classes of flat connections on the associated $\SO(3)$ bundle over $Y_0(K)$ that do lift to $\SU(2)$ connections over $E_K$, but that do not lift over all of $Y_0(K)$ because the lifted connections over $E_K$ have holonomy $-1$ along $\lambda$.  Equivalently, these are conjugacy classes of representations
\[ \rho: \pi_1(E_K) \to \SU(2) \]
such that $\rho(\lambda) = -1$.  Thus the complex used to define $I^w_*(Y_0(K))$ is generated by the points of $X(E_K)$ whose images lie on the line segment $L_\pi = \{\beta \equiv \pi\pmod{2\pi}\}$ in the pillowcase.

The next step is to show as in \cite[Theorem~7.2]{zentner} that if
\[ \gamma: [0,1] \to X(T^2) \]
is a topologically embedded path from $\gamma(0)=P$ to $\gamma(1)=Q$ that avoids the line $L_0 = \{\beta \equiv 0\pmod{2\pi}\}$, then $\gamma$ intersects the image $i^*(X(E_K))$.   Now the chain complex for $I^w_*(Y_0(K))$ is generated by the intersection of $i^*(X(E_K))$ with one such path, namely the line $L_\pi$.  Supposing we have another such path $\gamma$ that avoids $i^*(X(E_K))$ completely, then since $i^*(X(E_K))$ is compact it must actually be disjoint from an open neighborhood $U$ of this path.  Now $I^w_*(Y_0(K))$ is defined using a certain Chern--Simons functional, and \cite[Theorem~4.2]{zentner} says that we can modify it using holonomy perturbations so that $I^w_*(Y_0(K))$ is instead defined by the intersection of $i^*(X(E_K))$ with a path that is arbitrarily $C^0$-close to $\gamma$.  We take this path to lie in $U$, and then the intersection is empty, so $I^w_*(Y_0(K))$ is the homology of the zero complex and this is a contradiction.  So every such $\gamma$ must intersect $i^*(X(E_K))$.

Now just as in the proof of \cite[Theorem~7.1]{zentner} we know that $\Gamma = i^*(X(E_K))$ is an embedded finite graph in the pillowcase $X(T^2) \cong S^2$.  The graph $\Gamma$ contains the entire line $L_0 = \{\beta\equiv 0\pmod{2\pi}\}$, as the image of the reducible characters of $\pi_1(E_K)$, but by assumption it contains neither $P$ nor $Q$, so the above argument says that $P$ and $Q$ lie in different components of the complement $S^2 \setminus \Gamma$.  We use \cite[Lemma~7.3]{zentner} to conclude that $\Gamma$ contains a topologically embedded, homologically essential curve in $S^2 \setminus \{P,Q\}$.
\end{proof}

The following lemma will be useful in conjunction with Proposition~\ref{prop:curve-in-pillowcase}, in order to understand when the essential curve $C$ can pass through the corners of the pillowcase.

\begin{lemma} \label{lem:corners-limit}
Suppose that $H_1(Y;\Z)$ is 2-torsion, and let $K \subset Y$ be a nullhomologous knot.  If either $(0,0)$ or $(\pi,0)$ is a limit point of the image $i^*(X^\irr(E_K))$, then there must be a representation $\rho: \pi_1(E_K) \to \SU(2)$ with non-abelian image such that $\rho(\mu) = \rho(\lambda) = 1$.  In particular, neither $Y$ nor any Dehn surgery $Y_{p/q}(K)$ is $\SU(2)$-abelian.
\end{lemma}

\begin{proof}
Suppose we have a sequence of irreducible representations $\rho_n: \pi_1(E_K) \to \SU(2)$ such that the images
\[ i^*([\rho_n]) = (\alpha_n,\beta_n) \in X(T^2) \]
converge to either $(0,0)$ or $(\pi,0)$.  If their limit is $(\pi,0)$, then since $H_1(E_K) \cong H_1(Y) \oplus \Z$ with the $\Z$ summand generated by $\mu$, we can define a character
\[ \chi: \pi_1(E_K) \twoheadrightarrow H_1(E_K) \to \{\pm 1\} \]
which sends $H_1(Y)$ to $+1$ and $\mu$ to $-1$, and thus has central image.  In particular each
\[ \rho_n' = \chi \cdot \rho_n: \pi_1(E_K) \to \SU(2) \]
is an irreducible representation as well, and since $\rho'_n(\mu) = \chi(\mu)\rho(\mu) = -\rho_n(\mu)$ but $\rho'_n(\lambda) = \rho(\lambda)$, we have
\[ i^*([\rho'_n]) = (\alpha_n-\pi, \beta_n) \to (0,0). \]
Thus we may as well assume that $(\alpha_n,\beta_n) \to (0,0)$.  Moreover, since the $\SU(2)$ representation variety $R(E_K)$ is compact, we can pass to a subsequence to assume that the $\rho_n$ converge in $R(E_K)$; their limit is a representation
\[ \rho: \pi_1(E_K) \to \SU(2) \]
with $i^*([\rho]) = (0,0)$ and thus $\rho(\mu) = \rho(\lambda) = 1$.

The limiting representation $\rho$ factors as a composition
\[ \pi_1(E_K) \twoheadrightarrow \frac{\pi_1(E_K)}{\llangle \mu \rrangle} \cong \pi_1(Y) \xrightarrow{\rho_Y} \SU(2), \]
in which the last map $\rho_Y$ has the same image as $\rho$ itself.  If this image is abelian then $\rho_Y$ further factors through $H_1(Y)$; the latter is $2$-torsion, and $-1$ is the only order-2 element of $\SU(2)$, so then the image of $\rho$ lies in the center $\{\pm1\}$ of $\SU(2)$.  We will show that this is impossible, arguing along the same lines as in \cite[Lemma~3.1]{lpcz}, and this will imply that $\rho$ must not have abelian image after all.

To prove that $\rho$ cannot have central image, we think of it as a point of the $\SL(2,\C)$ representation variety $\cR(E_K)$, which is an affine variety over $\C$, and then since $\ad\rho$ is trivial we can identify
\[ T_\rho \cR(E_K) \cong H^1(E_K; \mathfrak{sl}(2,\C)_{\ad\rho}) \cong H^1(E_K; \C^3) \cong \C^3. \]
We have a finite-to-one (in fact, injective) morphism $f: \SL(2,\C) \to \cR(E_K)$, defined by sending $A \in \SL(2,\C)$ to the unique representation
\[ \rho_A: \pi_1(E_K) \twoheadrightarrow H_1(E_K;\Z) \cong H_1(Y) \oplus \Z \xrightarrow{\phi_A} \SL(2,\C) \]
such that $\phi_A|_{H_1(Y)} = \rho|_{H_1(Y)}$ and $\phi_A(\mu) = A$.  Then $f(1) = \rho_1 = \rho$ and 
\[ \dim_\C (\SL(2,\C)) = 3 = \dim_\C T_\rho \cR(E_K), \]
where on the left side we view $\SL(2,\C)$ as a complex variety and compute its dimension at the identity $1\in f^{-1}(\rho)$.  Thus \cite[Lemma~2.5]{lubotzky-magid} says that $\Img(f)$ contains a neighborhood of $\rho$ in $\cR(E_K)$, all of whose points are non-singular.  But then $\rho$ has a neighborhood in $R(E_K) \subset \cR(E_K)$ consisting only of points in $\Img(f)$, all of which have abelian image, and this contradicts the assumption that $\rho$ is a limit of irreducible representations.

In summary, we have shown that the representation $\rho$ must have non-abelian image, with $\rho(\mu) = \rho(\lambda) = 1$.  Now for any slope $\frac{p}{q}$, including $\frac{p}{q} = \frac{1}{0}$, we have $\rho(\mu^p\lambda^q) = 1$ and so $\rho$ factors as a composition
\[ \pi_1(E_K) \twoheadrightarrow \frac{\pi_1(E_K)}{\llangle \mu^p\lambda^q\rrangle} \cong \pi_1(Y_{p/q}(K)) \xrightarrow{\rho_{p/q}} \SU(2). \]
The map $\rho_{p/q}$ has the same image as $\rho$ itself, so its image is non-abelian and thus $Y_{p/q}(K)$ is not $\SU(2)$-abelian.
\end{proof}

\subsection{The cut-open pillowcase} \label{ssec:cut-open}

In some cases we can say more about the pillowcase image of $X(E_K)$ and can use this to simplify the statement of Proposition~\ref{prop:curve-in-pillowcase}.  For example:

\begin{lemma} \label{lem:alpha-pi}
Let $K$ be a nullhomologous knot in an $\SU(2)$-abelian 3-manifold $Y$, and fix a representation $\rho: \pi_1(E_K) \to \SU(2)$.  Suppose that $i^*([\rho])$ has coordinates $(\alpha,\beta)$ in the pillowcase, where $\alpha \in \pi\Z$.  Then $\rho$ has abelian image and $\beta\equiv 0\pmod{2\pi}$.
\end{lemma}

\begin{proof}
The claim that $\beta\equiv 0$ will follow from knowing that $\Img(\rho)$ is abelian: then $\rho$ factors through $H_1(E_K;\Z)$, and the homology class $[\lambda]$ is zero, so we must have $\rho(\lambda) = 1$.  Thus we focus on the claim that $\rho$ has abelian image.

Suppose first that $\alpha \equiv 0 \pmod{2\pi}$.  Then $\rho(\mu) = 1$, and so $\rho$ descends to a representation
\[ \rho_Y: \pi_1(Y) \cong \frac{\pi_1(E_K)}{\llangle\mu\rrangle} \to \SU(2), \]
which must then have abelian image.  But $\rho$ has the same image as $\rho_Y$, so $\Img(\rho)$ is abelian as well.

In the remaining case, we have $\alpha \equiv \pi \pmod{2\pi}$, so $\rho(\mu) = -1$.  Then we can multiply by a central character $\chi: \pi_1(E_K) \to \{\pm1\}$ with $\chi(\mu) = -1$, just as in the proof of Lemma~\ref{lem:corners-limit}, to replace $\rho$ with $\rho'$ such that $\rho'(\mu) = 1$.  By the previous case we know that $\rho'$ has abelian image, hence so does $\rho$.
\end{proof}

Lemma~\ref{lem:alpha-pi} sometimes allows us to replace the pillowcase with the cut-open pillowcase
\[ \cP = [0,\pi] \times (\R/2\pi\Z). \]
In particular, the natural quotient map $\cP \to X(T^2)$ glues each point $(0,\beta)$ to $(0,2\pi-\beta)$, and $(\pi,\beta)$ to $(\pi,2\pi-\beta)$, so it is one-to-one except at points of the form $(\alpha,\beta)$ with $\alpha \in \pi\Z$ but $\beta \not\in \pi\Z$.  Lemma~\ref{lem:alpha-pi} says that if $Y$ is $\SU(2)$-abelian then $i^*(X(E_K))$ avoids the images of such points, so it lifts uniquely to $\cP$.  Thus for $\SU(2)$-abelian $Y$ we have a well-defined map
\[ j: X(E_K) \to \cP. \]

The following is now a quick application of Proposition~\ref{prop:curve-in-pillowcase}, generalizing \cite[Theorem~3.3]{lpcz}.

\begin{theorem} \label{thm:pillowcase-loop}
Let $K \subset Y$ be a nullhomologous knot in an $\SU(2)$-abelian 3-manifold, and suppose that $I^w_*(Y_0(K)) \neq 0$, where $w$ is Poincar\'e dual to a meridian of $K$ in $Y_0(K)$.  Then the image $j(X(E_K)) \subset \cP$ must contain a topologically embedded curve that is homologically essential in $H_1(\cP;\Z) \cong \Z$.
\end{theorem}

\begin{proof}
The pillowcase image $i^*(X(E_K)) \subset X(T^2)$ does not contain the points $P=(0,\pi)$ or $Q=(\pi,\pi)$, by Lemma~\ref{lem:alpha-pi}.  Thus we can apply Proposition~\ref{prop:curve-in-pillowcase} to find an embedded curve
\[ C \subset i^*(X(E_K)) \]
that is homologically essential in $X(T^2) \setminus \{P,Q\}$.  Lemma~\ref{lem:alpha-pi} says that $C$ actually lies in
\[ X(T^2) \setminus \big( \{0,\pi\} \times (0,2\pi) \big), \]
where it is still homologically essential, and the inclusion of the latter into $\cP$ is a homotopy equivalence taking $C$ to its image $j(C)$, so $j(C)$ is a homologically essential curve in $\cP$.
\end{proof}

We can now deduce the following generalization of the main result of \cite{km-su2}.

\begin{theorem} \label{thm:km-su2-generalized}
Let $Y$ be an $\SU(2)$-abelian 3-manifold such that $H_1(Y)$ is 2-torsion, and let $K \subset Y$ be a nullhomologous knot with irreducible, boundary-incompressible complement.  Then for any $r\in\Q$ with $0 < |r| \leq 2$, there is a representation
\[ \rho: \pi_1(Y_r(K)) \to \SU(2) \]
with non-abelian image.
\end{theorem}

\begin{proof}
Proposition~\ref{thm:2-torsion-zero-surgery} tells us that $I^w_*(Y_0(K)) \neq 0$, where $w$ is Poincar\'e dual to a meridian of $K$.  By Theorem~\ref{thm:pillowcase-loop}, we can thus find a continuous path
\[ \gamma: [0,1] \to [0,\pi] \times [0,2\pi] \]
such that if we write $\gamma(t) = (\alpha_t, \beta_t)$, then
\begin{itemize}
\item $\beta_0 = 0$, $\beta_1 = 2\pi$, and $0 < \beta_t < 2\pi$ for $0 < t < 1$;
\item for each $t$, there is a representation $\rho_t: \pi_1(E_K) \to \SU(2)$ with
\begin{align*}
\rho_t(\mu) &= \begin{pmatrix} e^{i\alpha_t} & 0 \\ 0 & e^{-i\alpha_t} \end{pmatrix}, &
\rho_t(\lambda) &= \begin{pmatrix} e^{i\beta_t} & 0 \\ 0 & e^{-i\beta_t} \end{pmatrix};
\end{align*}
\item and $\rho_t$ is irreducible for $0 < t < 1$, since $0 < \beta_t < 2\pi$ implies that $\rho_t(\lambda) \neq 1$.
\end{itemize}
Since $(\alpha_t,\beta_t) \to (\alpha_0,0)$ as $t \searrow 0$, and since $R(E_K)$ is compact, some subsequence of the irreducibles
\[ \{ \rho_t \mid 0 < t < 1 \} \subset X^\irr(E_K) \]
converges to a representation $\bar\rho_0 \in R(E_K)$ with $j([\bar\rho_0]) = (\alpha_0,0)$.  Since $H_1(Y)$ is 2-torsion, we can apply Lemma~\ref{lem:corners-limit} to say that $\alpha_0$ is neither $0$ nor $\pi$.  The same argument says that $0 < \alpha_1 < \pi$ as well.

Now suppose without loss of generality that $0 < r \leq 2$, and write $r = \frac{p}{q}$ in lowest terms, so that $0 < p \leq 2q$.  We note for each $t \in [0,1]$ that
\[ \rho_t(\mu^p\lambda^q) = \begin{pmatrix} e^{i(p\alpha_t+q\beta_t)} & 0 \\ 0 & e^{-i(p\alpha_t+q\beta_t)} \end{pmatrix}, \]
and that $\alpha_0 < \pi$ and $p \leq 2q$ imply that
\[ p\alpha_0 + q\beta_0 = p\alpha_0 < p\pi \leq 2q\pi, \]
while $\alpha_1 > 0$ tells us that
\[ p\alpha_1 + q\beta_1 = p\alpha_1 + 2q\pi > 2q\pi. \]
Thus by continuity there is some $t\in(0,1)$ such that $p\alpha_t+q\beta_t = 2q\pi$, and then $\rho_t$ is an irreducible representation satisfying $\rho_t(\mu^p\lambda^q) = 1$, so it descends to the desired representation of $\pi_1(Y_r(K))$.
\end{proof}

\begin{remark}
It is not clear to us whether the hypotheses of Theorem~\ref{thm:km-su2-generalized} should imply the existence of a non-abelian representation $\pi_1(Y_0(K)) \to \SU(2)$, even when $Y=S^3$.  This is equivalent to there being an irreducible $\rho: \pi_1(E_K) \to \SU(2)$ with pillowcase image $i^*([\rho]) = (\alpha,0)$ for some $\alpha$.  If no such $\rho$ exists, then the representation $\bar\rho_0 \in R(E_K)$ constructed in the proof of Theorem~\ref{thm:km-su2-generalized} is a reducible limit of irreducible representations, and this implies that the Alexander polynomial satisfies $\Delta_K(e^{2i\alpha_0}) = 0$, cf.~\cite[Theorem~19]{klassen} in the case $Y=S^3$ or \cite[Theorem~2.7]{heusener-porti-suarez} more generally.

On the other hand, these hypotheses do imply that $I^w_*(Y_0(K)) \neq 0$, and hence there is an irreducible representation $\pi_1(Y_0(K)) \to \SO(3)$ that does not lift to an $\SU(2)$ representation.
\end{remark}

\section{Pinching and the twisted $I$-bundle over the Klein bottle} \label{sec:pinching}

In this section we will construct and study some degree-1 maps between compact $3$-manifolds with torus boundary.  As a warm-up exercise, we recall the well-known construction of ``pinching'' maps onto solid tori here; after doing so, we will study the twisted $I$-bundle over the Klein bottle in some detail, culminating in the construction of pinching maps onto it in Proposition~\ref{prop:klein-bottle-pinch}.  We will repeatedly make use of the following claim.

\begin{lemma} \label{lem:pinch-sphere}
Let $X$ be a compact $n$-manifold, and suppose we have a continuous map $f: \partial X \to S^{n-1}$.  Then $f$ can be extended to a continuous map $\tilde{f}: X \to D^n$, with $\tilde{f}^{-1}(\partial D^n) = \partial X$.
\end{lemma}

\begin{proof}
We identify a collar neighborhood $[0,1] \times \partial X$ of the boundary $\{1\}\times \partial X$, and then set
\[ \tilde{f}(t, x) = t\cdot f(x) \]
for all $(t,x) \in [0,1] \times \partial X$.  Then $\tilde{f}(\{0\}\times \partial X) = \{0\}$, so we extend $\tilde{f}$ to the rest of $X$ by setting $\tilde{f}(y) = 0$ for all $y \not\in [0,1]\times \partial X$.
\end{proof}

Lemma~\ref{lem:pinch-sphere} allows us to construct pinching maps onto solid tori as follows.

\begin{proposition} \label{prop:solid-torus-pinch}
Let $M$ be a compact, oriented $3$-manifold with torus boundary, and let $\lambda \subset \partial M$ be an essential curve that bounds a properly embedded, orientable surface $F \subset M$.  Then there is a degree-1 map
\[ f: M \to S^1\times D^2 \]
that restricts to a homeomorphism $\partial M \to S^1 \times \partial D^2$ and sends $\lambda$ to $\{\pt\} \times \partial D^2$.
\end{proposition}

\begin{proof}
We first define $f$ on $\partial M$ by choosing a homeomorphism $\partial M \cong S^1 \times \partial D^2$ that sends $\lambda$ to $\{\pt\} \times \partial D^2$, and then extend it to a homeomorphism between collar neighborhoods of both boundaries.  Now $f$ is defined on a collar neighborhood of $\partial F$ in $F$, and it sends the boundary of this collar to a circle in $\{\pt\} \times D^2$ that bounds a disk, so Lemma~\ref{lem:pinch-sphere} lets us extend $f$ across all of $F$.  Again we extend this to a collar neighborhood of $F$, so now $f$ is defined on $N(\partial M \cup F)$.  The boundary of this domain is sent to a 2-sphere that bounds a 3-ball (namely, the boundary component of $N\big((S^1 \times \partial D^2) \cup (\{\pt\}\times D^2)\big)$ that lies on the interior of $S^1\times D^2$), so we use Lemma~\ref{lem:pinch-sphere} to extend $f$ to the rest of $M$ and we are done.
\end{proof}

In the rest of a section we will work with \emph{rational longitudes}, so in order to define them we must first recall a standard fact about $3$-manifolds that we will use frequently in \S\ref{sec:knot-exteriors}.  If $M$ is a compact orientable 3-manifold with boundary, the ``half lives half dies'' principle (see for example \cite[Lemma~3.5]{hatcher-3}) says that over any field $\F$, the map
\[ i_*: H_1(\partial M;\F) \to H_1(M;\F) \]
has rank $\frac{1}{2}\dim H_1(\partial M;\F)$, which is $1$ if $\partial M$ is a torus.  (The orientability is needed to ensure that $M$ satisfies Poincar\'e--Lefschetz duality over $\F$.)  Applying this over $\F=\Q$, we deduce that there is a primitive integral class $\lambda \in H_1(\partial M;\Z)$ that generates the kernel of $i_*$ over $\Q$, and it is unique up to sign.  We call this the \emph{rational longitude} of $M$.  While $\lambda$ need not be nullhomologous in $M$, the integral class $i_*(\lambda)$ is always torsion.

\subsection{The twisted $I$-bundle over the Klein bottle}

We define an annulus
\[ A = [-1,1] \times (\R/2\pi\Z) \]
and a pair of orientation-preserving homeomorphisms $A\to A$ by the formulas
\begin{align*}
\phi(r,\theta) &= (-r,-\theta), \\
\tau(r,\theta) &= (r, \theta - \pi(r+1)).
\end{align*}
We note that the homeomorphism $\tau$ is a Dehn twist about the core $c = \{0\} \times (\R/2\pi\Z)$.

\begin{figure}
\begin{tikzpicture}
\draw[fill=gray!50,fill opacity=0.5, even odd rule] (0,0) ellipse (3 and 1) ellipse (0.75 and 0.25);
\draw[blue,very thick] (0,0) ellipse (1.875 and 0.625);
\draw[blue,very thick,fill=blue!50,fill opacity=0.5] (-1.875,0) -- ++(0,3) arc (180:0:1.875 and 0.625) -- ++(0,-3) arc (0:180:1.875 and 0.625);
\foreach \i in {-3,-0.75,0.75,3} { \draw (\i,0) -- ++(0,3); }
\draw (-0.75,0) arc (180:0:0.75 and 0.25);
\draw[fill=white,fill opacity=0.5] (-0.75,3) -- ++(0,-3) arc (180:360:0.75 and 0.25) -- ++(0,3) arc (360:180:0.75 and 0.25);
\draw[fill=gray!50,fill opacity=0.5, even odd rule] (0,3) ellipse (3 and 1) ellipse (0.75 and 0.25);
\draw[fill=white] (0,3) ellipse (0.75 and 0.25);
\draw[->] (0.75,0) arc (360:270:0.75 and 0.25);
\foreach \i in {0,15,...,345} { \draw[very thin] (0,3) ++ (\i:1 and 1/3) -- ++(\i:1.8 and 0.6); }
\foreach \j in {0,60,...,300} {
  \draw[thin,yscale=1/3,domain=0:360,smooth,samples=20,variable=\i] plot ({\i+\j}:{0.75+2.25*\i/360});
}
\draw[blue,very thick] (0,3) ellipse (1.875 and 0.625);
\draw[blue,very thick,fill=blue!50,fill opacity=0.5] (-1.875,3) -- ++(0,-3) arc (180:360:1.875 and 0.625) -- ++(0,3) arc (360:180:1.875 and 0.625);
\draw[fill=white,fill opacity=0.5] (-3,3) -- ++(0,-3) arc (180:360:3 and 1) -- ++(0,3) arc (360:180:3 and 1);
\node[left] at (-3,3) {$A\times\{1\}$};
\node[left] at (-3,0) {$A\times\{0\}$};
\node[blue,right] at (1.875,1.5) {$B$};
\draw[looseness=1,-latex] (0,3) ++ (30:3.3 and 1.1) -- ++(60:0.3 and 0.1) to[out=30,in=105] (4.5,3) to[out=285,in=75] node[midway,right] {$\psi_n$} ++(0,-3) to[out=255,in=-30] ($(330:3.3 and 1.1)+(300:0.3 and 0.1)$) -- ++(120:0.3 and 0.1);
\draw[->] (3,3) arc (0:90:3 and 1);
\draw[-Stealth] (-3,0) arc (180:270:3 and 1);
\draw[-Stealth] (-0.75,3) arc (180:90:0.75 and 0.25);
\end{tikzpicture}
\caption{The twisted $I$-bundle over the Klein bottle $B$, shown here as the mapping torus of $\psi_n: A \to A$, together with the fibration of the top and bottom annuli into intervals.}
\label{fig:klein-bottle-nbhd}
\end{figure}
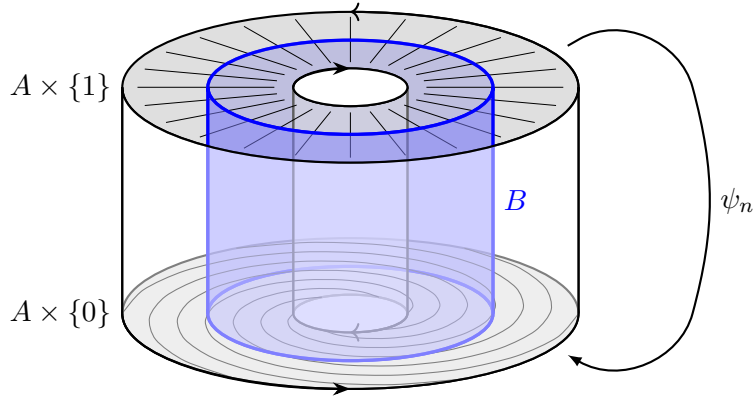

\begin{lemma} \label{lem:klein-bottle-mapping-torus}
Define a diffeomorphism $\psi_n: A \to A$ for each $n\in\Z$ by $\psi_n = \tau^n \circ \phi$.  Then the mapping torus $M_{\psi_n}$ is homeomorphic to the twisted $I$-bundle over the Klein bottle for all $n$, and if $\lambda \subset \partial M_{\psi_n}$ is the rational longitude then each annulus fiber generates $H_2(M_{\psi_n},\partial M_{\psi_n}) \cong \Z$ and has boundary homologous in $\partial M_{\psi_n}$ to $2\lambda$.
\end{lemma}

\begin{proof}
We first observe that $\psi_n$ fixes the curve $c = \{0\}\times(\R/2\pi\Z)$ setwise, since both $\phi$ and $\tau$ do, but that it reverses the orientation of $c$: we have
\[ \psi_n(r,\theta) = \tau^n(\phi(r,\theta)) = \tau^n(-r,-\theta) = (-r, -\theta - n\pi(1-r)) \]
for all $r$ and $\theta$, and this sends the circle $\{r=0\}$ to itself.  Thus the mapping torus
\[ M_{\psi_n|_c} = \frac{c \times [0,1]}{(x,1) \sim (\psi_n(x), 0)} \]
of $\psi_n|_c$ is homeomorphic to a Klein bottle.

Next, the mapping torus of $\psi_n$ on all of $A$ is by definition
\[ M_{\psi_n} = \frac{A \times[0,1]}{(x,1) \sim (\psi_n(x),0)}, \]
and the Klein bottle $B = M_{\psi_n|_c}$ is a submanifold of $M_{\psi_n}$, identified as the image of $\{r=0\} \times [0,1]$ inside $A\times [0,1]$.  We can check that the projection map
\begin{align*}
\pi: A \times [0,1] &\to c \times [0,1] \\
\big((r,\theta), t\big) &\mapsto \big((0, \theta + (1-t)n\pi r), t\big)
\end{align*}
fixes all points of $c\times[0,1]$, i.e., where $r=0$, and that the fiber over each point is an interval:
\[ \pi^{-1}\big( (0,\theta), t \big) = \{((s,\theta-(1-t)n\pi s),t) \mid s \in [-1,1]\}. \]
Moreover, the monodromy $\psi_n$ identifies the fibers at $t=1$ with fibers at $t=0$: we have
\begin{align*}
\big((r,\theta),1\big) \sim (\psi_n(r,\theta), 0)&= \big((-r, -\theta - n\pi(1-r)), 0\big) \\
&= \big((-r, (-\theta - n\pi) - n\pi (-r)), 0 \big) \\
&\in \pi^{-1}\big( (0,-\theta-n\pi), 0 \big).
\end{align*}
Thus $\pi$ descends to a fibration $\pi: M_{\psi_n} \to B$ with interval fibers.  The total space is orientable whereas $B$ is not, so it must be the twisted $I$-bundle over $B$, as claimed.

Finally, consider the fiber $A_1 = A \times \{1\}$ of $M_{\psi_n}$.  This fiber is primitive
as an element of
\[ H_2(M_{\psi_n},\partial M_{\psi_n}) \cong H^1(M_{\psi_n}) \cong H^1(B) \cong \Z, \]
and hence generates it, because it has a single transverse point of intersection with the closed curve
\[ (0, -\tfrac{n\pi}{2}) \times [0,1] \subset \frac{A\times[0,1]}{(x,1)\sim(\psi_n(x),0)} = M_{\psi_n}. \]
(We note that this curve is closed because $\psi_n(0,-\frac{n\pi}{2}) = (0,-\frac{n\pi}{2})$.)  We orient the components of
\[ \partial A_1 = \big(\{\pm 1\} \times (\R/2\pi\Z)\big) \times \{1\} \]
as the boundary of $A_1$.  Then these components are isotopic to each other as oriented curves in the torus $\partial M_{\psi_n}$, because $\lambda = \big( \{+1\} \times (\R/2\pi\Z) \big) \times \{1\}$ is identified with $\big(\{-1\} \times (\R/2\pi\Z) \big) \times \{0\}$ in an orientation-reversing way.  In particular $\lambda$ is a rational longitude for $M_{\psi_n}$, and $\partial A_1$ is homologous in $\partial M_{\psi_n}$ to $2\lambda$ as claimed.
\end{proof}

The twisted $I$-bundle over the Klein bottle is depicted as a mapping torus in Figure~\ref{fig:klein-bottle-nbhd}.  With this construction at hand, we now study its $\SU(2)$ character variety.

\begin{proposition} \label{prop:klein-bottle-pillowcase}
Let $N$ be the twisted $I$-bundle over the Klein bottle, with rational longitude $\lambda_0$.  Then there is a unique peripheral curve $\mu_0 \subset \partial N$ with Dehn filling $N(\mu_0) \cong \RP^3\#\RP^3$, and $\mu_0$ is dual to $\lambda_0$.  Every other $\SU(2)$-abelian Dehn filling of $N$ has cyclic fundamental group, namely $\Z$ or $\Z/4k\Z$ for some integer $k \geq 1$.

Viewing $\SU(2)$ as the unit quaternions, every representation
\[ \rho: \pi_1(N) \to \SU(2) \]
is conjugate to one with $(\rho(\mu_0),\rho(\lambda_0))$ equal to either $(1,\pm1)$ or $(-1, e^{it})$ for some $t \in \R/2\pi\Z$, and every such value is realized by some $\rho$.  The image of $\rho$ is non-abelian if and only if $\rho(\lambda_0) \neq \pm1$.
\end{proposition}

\begin{proof}
Taking $n=0$ in Lemma~\ref{lem:klein-bottle-mapping-torus}, the mapping torus $M_{\psi_0} = M_\phi \cong N$ has fundamental group
\[ \pi_1(M_\phi) \cong \pi_1(B) \cong \langle a,b \mid aba^{-1} = b^{-1} \rangle, \]
where $a$ is identified with the section
\[ (0,0) \times [0,1] \subset \frac{A\times [0,1]}{(x,1)\sim(\phi(x),0)} \]
of $M_\phi \to S^1$, and $b$ is identified with a core circle $c \times \{\frac{1}{2}\}$ of one of the annulus fibers.  Then $b$ is isotopic in $M_\phi$ to the rational longitude $\lambda_0 \subset \partial M_\phi$, while $a^2$ is isotopic to a dual peripheral curve
\[ \mu_0 = \{(1,0),(-1,0)\} \times [0,1] \subset M_\phi. \]
Thus $\pi_1(M_\phi)$ has peripheral subgroup $\langle \mu_0,\lambda_0 \rangle = \langle a^2, b\rangle$.

We first claim that the Dehn filling $N(\mu_0)$ is $\RP^3\#\RP^3$.  Indeed, by viewing $N$ as the mapping torus of $\phi: A \to A$, one can see that the annuli
\[ \left( [-1,1] \times \{\theta\} \right) \times [0,1] \subset A \times [0,1], \qquad \theta=0\text{ or }\pi \]
give rise to a pair of M\"obius bands in $N$, with tubular neighborhoods
\[ \frac{([-1,1] \times I) \times [0,1]}{(x,1) \sim (\phi(x),0)}, \qquad I = (-\tfrac{\pi}{2},\tfrac{\pi}{2}) \text{ or } (\tfrac{\pi}{2},\tfrac{3\pi}{2}),\]
whose boundaries are parallel copies of $\mu_0$.  The meridional disks in the Dehn filling solid torus complete each of these M\"obius bands to real projective planes in $N(\mu_0)$, and their tubular neighborhoods to punctured copies of $\RP^3$, producing the desired identification $N(\mu_0) \cong \RP^3 \# \RP^3$.

Next, we will show that $\mu_0$ is unique.  Given any other slope $\alpha = \mu_0^p\lambda_0^q = a^{2p}b^q$ in $\partial N$, we must have $q \neq 0$ and $\gcd(p,q)=1$; we take $q \geq 1$ without loss of generality.  If $q=1$ then we compute that
\[ \pi_1(N(\alpha)) \cong \langle a,b \mid aba^{-1}=b^{-1},\ b=a^{-2p}\rangle \cong \Z/4|p|\Z, \]
which is cyclic of order $4|p|$ if $p \neq 0$, and is $\Z$ otherwise (corresponding to $N(\lambda_0) \cong S^1\times S^2$).  If $q \geq 2$ and $p$ is odd then we can define a non-abelian representation
\[ \pi_1(N(\alpha)) \cong \langle a,b \mid aba^{-1} = b^{-1},\ a^{2p}b^q=1 \rangle \to \SU(2)\]
by sending $a\mapsto j$ and $b \mapsto e^{i\pi/q}$, so $N(\alpha)$ is not $\SU(2)$-abelian.  Similarly if $q \geq 3$ and $p$ is even then we can send $a \mapsto j$ and $b \mapsto e^{i\cdot 2\pi/q}$, and this is also non-abelian; the case where $q=2$ and $p$ is even does not occur because $p$ and $q$ are coprime.  Thus every $\SU(2)$-abelian Dehn filling of $N$ other than $N(\mu_0) \cong \RP^3\#\RP^3$ has fundamental group $\Z$ or $\Z/4|p|\Z$ for some $p$.

Finally, we consider an arbitrary representation $\rho: \pi_1(M_\phi) \to \SU(2)$, which must satisfy $\rho(aba^{-1}) = \rho(b^{-1})$.  If $\rho(a) = \pm1$ then $\rho$ is reducible and $\rho(b) = \rho(b^{-1})$ implies that $\rho(\lambda_0) = \rho(b) = \pm1$, while $\rho(\mu_0) = \rho(a^2) = 1$.  Otherwise up to conjugation we have $\rho(a) = e^{js}$ for some $s \not\in \pi\Z$; the relation $\rho(aba^{-1}) = \rho(b^{-1})$ implies that $\rho(a) = \pm j$ and that $\rho(b)$ has zero $j$-component, so up to another conjugation we can further arrange that $\rho(a) = j$ and $\rho(b) = e^{it}$ for some $t$, and any value of $t$ works.  In this case we have $\rho(\mu_0) = \rho(a^2) = -1$ and $\rho(\lambda_0) = \rho(b) = e^{it}$, and $\rho$ is irreducible unless $\rho(b) = e^{it}$ commutes with $\rho(a)=j$, i.e., unless $\rho(b)=\pm1$.
\end{proof}

\subsection{Pinching maps for rational longitudes of order 2}

In this subsection we will construct degree-1 maps from compact manifolds $M$ with torus boundary onto the twisted $I$-bundle over the Klein bottle.  In contrast to Proposition~\ref{prop:solid-torus-pinch}, which only works when the rational longitude of $M$ is nullhomologous, here we require the rational longitude to have order $2$.

\begin{proposition} \label{prop:klein-bottle-pinch}
Let $M$ be a compact, oriented $3$-manifold with torus boundary, and suppose that the rational longitude $\lambda_M \subset \partial M$ has order $2$ in $H_1(M)$.  Then there is a degree-1 map
\[ f: M \to N, \]
where $N$ is the twisted $I$-bundle over the Klein bottle, such that $f$ restricts to a homeomorphism $\partial M \to \partial N$ sending $\lambda_M$ to a rational longitude $\lambda_N \subset \partial N$.
\end{proposition}

\begin{proof}
Using Lemma~\ref{lem:klein-bottle-mapping-torus}, we realize $N$ as the mapping torus of some self-diffeomorphism
\[ \psi_n: A \to A \]
of the annulus; we will choose the integer $n\in\Z$ later.  We will also let
\[ F \subset M \]
be a connected, properly embedded rational Seifert surface, whose boundary is two disjoint copies $\lambda^0_M, \lambda^1_M \subset \partial M$ of the rational longitude for $M$.  We will construct the map $f$ in stages: first we define it on $\partial M$, then we extend it to a rational Seifert surface $F'$ constructed by stabilizing $F$, and then we extend it across the remainder of $M$.

This last step requires substantially more care than did the pinching maps onto solid tori in Proposition~\ref{prop:solid-torus-pinch}: letting $M_0$ denote the remaining portion of $M$, we will want to send $M_0$ into $N$ minus a neighborhood of an annulus fiber, i.e., a solid torus.  In order to extend our initial map $\partial M_0 \to S^1 \times S^1$ to $M_0 \to S^1\times D^2$, we must arrange for some curve $\gamma \subset \partial M_0$ that is nullhomologous in $M_0$ to be sent to $\{\pt\} \times S^1$, so that we can collapse a surface in $M_0$ with boundary $\gamma$ to $\{\pt\} \times D^2$.  By contrast, the target in the analogous step of Proposition~\ref{prop:solid-torus-pinch} was a solid torus minus a disk fiber, which is a ball, and we could just apply Lemma~\ref{lem:pinch-sphere} to extend $\partial M_0 \to S^2$ to $M_0 \to D^3$ without any extra hypotheses.

We fix points $p_0 \in \lambda^0_M$ and $p_1 \in \lambda^1_M$, and a properly embedded, oriented arc $\alpha \subset F$ from $p_0$ to $p_1$.  Identifying a closed tubular neighborhood of $F$ as $F\times [-1,1] \subset M$, with $F = F \times \{0\}$, and letting 
\[ E_F = M \setminus \big(F \times(-1,1)\big) \]
be the exterior of $F$, we build a closed curve $c \subset \partial E_F$ as the union of the oriented arcs
\[ \alpha_{\pm} = \alpha \times \{\pm1\} \subset F \times \{\pm1\} \]
with a pair of arcs in $\partial M \cap E_F$ from $p_1 \times \{1\}$ to $p_0 \times \{-1\}$, and from $p_1\times\{-1\}$ to $p_0 \times \{1\}$.  See the top row of Figure~\ref{fig:rational-surface-curve}.
\begin{figure}
\begin{tikzpicture}
\begin{scope}
\foreach \p in {(1.5,3.25),(1.5,-1)} {\draw \p -- ++(3,0); }
\draw (0,4) -- ++(3,0);
\draw (0,-0.25) -- ++(1.45,0) ++(0.1,0) -- ++(1.45,0);
\foreach \i/\y in {0/1,1/2.75} {
  \shade[left color=gray, right color=white, opacity=0.5, shading angle=145] (0,\y) -- ++(1.5,-0.75) -- ++(3,0) -- coordinate[midway] (alpha\i) ++(-1.5,0.75) -- ++(-3,0);
  \draw[very thick] (0,\y) node[left,inner sep=1pt] {$\lambda^{\i}_M$} -- coordinate[midway] (p\i) ++(1.5,-0.75) -- ++(3,0) ++(-1.5,0.75) -- ++(-3,0);
  \draw[fill=black] (p\i) circle (0.05) node[below left, inner sep=1pt] {\small$p_{\i}$};
  \node[right] at (alpha\i) {$\alpha$};
}
\begin{scope}[every path/.style={ultra thick}]
  \draw[->] (p0) -- ++(1.5,0);
  \draw (p0) -- ++(3,0);
  \draw (p1) ++(3,0) -- ++(-3,0);
  \draw[->] (p1) ++(3,0) -- ++(-1.5,0);
\end{scope}
\draw (0,-0.25) -- ++(0,4.25) -- ++(1.5,-0.75) -- ++(0,-4.25) -- ++(-1.5,0.75);
\end{scope}
\begin{scope}[xshift=6.5cm]
\foreach \p in {(1.5,3.25),(1.5,-1)} {\draw \p -- ++(3,0); }
\draw (0,4) -- ++(3,0);
\draw (0,-0.25) -- ++(1.45,0) ++(0.1,0) -- ++(1.45,0);
\foreach \j/\dy in {m/-0.25,p/0.25} {
  \foreach \i/\y in {0/1,1/2.75} {
    \shade[left color=gray, right color=white, opacity=0.5, shading angle=145] (0,\y)  ++(0,\dy) -- ++(1.5,-0.75) -- ++(3,0) -- coordinate[midway] (alpha\i) ++(-1.5,0.75) -- ++(-3,0);
    \draw[thin] (0,\y) ++ (0,\dy) -- coordinate[midway] (p\i) ++(1.5,-0.75) -- ++(3,0) ++(-1.5,0.75) -- ++(-3,0);
  }
  \begin{scope}[every path/.style={ultra thick,blue}]
    \draw (p0) coordinate (p0\j) -- ++(3,0);
    \draw (p1) coordinate (p1\j) ++(3,0) -- ++(-3,0);
  \end{scope}
}
\foreach \y in {1,2.75} {
  \draw[thin,fill=gray, fill opacity=0.6] (0,\y) ++(0,0.25) -- ++(1.5,-0.75) coordinate (nw) -- ++(0,-0.5) -- ++(-1.5,0.75) -- cycle;
  \shade[left color=gray, right color=white, opacity=0.6] (nw) rectangle ++(3,-0.5);
  \draw[thin] (nw) -- ++(3,0) ++(0,-0.5) -- ++(-3,0);
}
\draw (0,-0.25) -- ++(0,4.25) -- coordinate[midway] (topmid) ++(1.5,-0.75) -- ++(0,-4.25) -- ++(-1.5,0.75) coordinate[midway] (botmid);
\begin{scope}[every path/.style={ultra thick,blue}]
  \draw (p0p) -- coordinate[midway] (midarrow) (p1m);
  \draw (p0m) -- coordinate[pos=0.3] (botarrow) (botmid);
  \draw (p1p) -- coordinate[pos=0.7] (toparrow) (topmid);
  \draw [->] (p1m) -- (midarrow);
  \draw [->] (botmid) -- (botarrow);
  \draw [->] (p1p) -- (toparrow);
  \path (p1p) ++(3,0) node[right] {$c$};
\end{scope}
\end{scope}
\begin{scope}[yshift=-6cm]
\foreach \p in {(1.5,3.25),(1.5,-1)} {\draw \p -- ++(3,0); }
\draw (0,4) -- ++(3,0);
\draw (0,-0.25) -- ++(1.45,0) ++(0.1,0) -- ++(1.45,0);

\foreach \j/\dy in {m/-0.25,p/0.25} {
  \foreach \i/\y in {0/1,1/2.75} {
    \path (0,\y) ++ (0,\dy) -- coordinate[midway] (p\i\j) ++(1.5,-0.75);
  }
}
\path (0,-0.25) -- ++(0,4.25) -- coordinate[midway] (topmid) ++(1.5,-0.75) -- ++(0,-4.25) -- ++(-1.5,0.75) coordinate[midway] (botmid);
\path (p0p) -- coordinate[midway] (midarrow) (p1m);
\path (p0m) -- coordinate[pos=0.3] (botarrow) (botmid);
\path (p1p) -- coordinate[pos=0.7] (toparrow) (topmid);

\begin{scope}[every path/.style={thick,blue}]
  \draw (botmid) -- (p0m) -- ++(3,0);
  \draw[red, ultra thick] (botmid) ++ (0.25,0) -- ($(p0m)+(0.25,-0.25)$) -- ++(2.75,0);
\end{scope}

\foreach \i/\y in {0/1} {
  \shade[left color=gray, right color=white, opacity=0.75, shading angle=145] (0,\y) -- ++(1.5,-0.75) -- ++(3,0) -- coordinate[midway] (alpha\i) ++(-1.5,0.75) -- ++(-3,0);
  \draw (0,\y) -- coordinate[midway] (p\i) ++(1.5,-0.75) -- ++(3,0) ++(-1.5,0.75) -- ++(-3,0);
}

\begin{scope}[every path/.style={thick,blue}]
  \draw (p1m) ++(3,0) -- ++(-3,0) -- (p0p) -- ++(3,0);
  \draw[red, ultra thick] (p0p) ++(3,0.25) -- ++(-2.75,0) -- ($(p1m)+(0.25,-0.25)$) -- ++(2.75,0);
\end{scope}

\foreach \i/\y in {1/2.75} {
  \shade[left color=gray, right color=white, opacity=0.75, shading angle=145] (0,\y) -- ++(1.5,-0.75) -- ++(3,0) -- coordinate[midway] (alpha\i) ++(-1.5,0.75) -- ++(-3,0);
  \draw (0,\y) -- coordinate[midway] (p\i) ++(1.5,-0.75) -- ++(3,0) ++(-1.5,0.75) -- ++(-3,0);
}

\begin{scope}[every path/.style={thick,blue}]
  \draw (p1p) ++(3,0) -- ++(-3,0) -- (topmid);
  \draw[red, ultra thick] (p1p) ++(3,0.25) node[right, inner sep=1pt] {$c'_{\vphantom{!}}$} -- ++(-2.75,0) -- ($(topmid)+(0.25,0)$);
\end{scope}

\draw (0,-0.25) -- ++(0,4.25) -- coordinate[midway] (topmid) ++(1.5,-0.75) -- ++(0,-4.25) -- ++(-1.5,0.75) coordinate[midway] (botmid);
\end{scope}
\begin{scope}[xshift=6.5cm,yshift=-6cm]
\foreach \p in {(1.5,3.25),(1.5,-1)} {\draw \p -- ++(3,0); }
\draw (0,4) -- ++(3,0);
\draw (0,-0.25) -- ++(1.45,0) ++(0.1,0) -- ++(1.45,0);

\foreach \j/\dy in {m/-0.25,p/0.25} {
  \foreach \i/\y in {0/1,1/2.75} {
    \path (0,\y) ++ (0,\dy) -- coordinate[midway] (p\i\j) ++(1.5,-0.75);
  }
}
\path (0,-0.25) -- ++(0,4.25) -- coordinate[midway] (topmid) ++(1.5,-0.75) -- ++(0,-4.25) -- ++(-1.5,0.75) coordinate[midway] (botmid);
\path (p0p) -- coordinate[midway] (midarrow) (p1m);
\path (p0m) -- coordinate[pos=0.3] (botarrow) (botmid);
\path (p1p) -- coordinate[pos=0.7] (toparrow) (topmid);

\begin{scope}
  \draw[thin,blue] (botmid) -- (p0m) -- ++(3,0);
  \draw[thin] ($(botmid)+(0.55,0)$) to[bend right=25] ++(-0.3,0);
  \shade[left color=gray, right color=white, opacity=0.75] ($(botmid)+(0.55,0)$) to[bend left=25] ++(-0.3,0) -- ($(p0m)+(0.25,-0.45)$) arc (180:90:0.2) -- ++(2.55,0) -- ++(0,-0.3) -- ++(-2.3,0) arc (90:180:0.15) -- ($(botmid)+(0.55,0)$);
  \draw ($(botmid)+(0.55,0)$) to[bend left=25] ++(-0.3,0) -- ($(p0m)+(0.25,-0.45)$) arc (180:90:0.2) -- ++(2.55,0) coordinate (rightend) ++(0,-0.3) -- ++(-2.3,0) arc (90:180:0.15) -- ($(botmid)+(0.55,0)$);
  \path (rightend) ++ (0,-0.15) node[right, inner sep=1pt] {$F'$};
\end{scope}

\foreach \i/\y in {0/1} {
  \shade[left color=gray, right color=white, opacity=0.75, shading angle=145] (0,\y) -- ++(1.5,-0.75) -- ++(3,0) -- coordinate[midway] (alpha\i) ++(-1.5,0.75) -- ++(-3,0);
  \draw (0,\y) -- coordinate[midway] (p\i) ++(1.5,-0.75) -- ++(3,0) ++(-1.5,0.75) -- ++(-3,0);
}

\begin{scope}
  \draw[thin,blue] (p1m) ++(3,0) -- ++(-1.1,0) to[bend left=20] ++(-0.5,0) -- ++(-1.4,0) -- (p0p) -- ++(3,0);
  \shade[left color=gray, right color=white, opacity=0.75] (p0p) ++(3,0.25) -- ++(-2.55,0) arc (270:180:0.2) -- ($(p1m)+(0.25,-0.45)$) arc (180:90:0.2) -- ++(2.55,0) -- ++(0,-0.25) -- ($(p1m)+(0.65,-0.5)$) arc (90:270:0.125) -- ($(p0p)+(3,0.5)$);
  \draw (p0p) ++(3,0.25) -- ++(-2.55,0) arc (270:180:0.2) -- ($(p1m)+(0.25,-0.45)$) arc (180:90:0.2) -- ++(2.55,0)
    ++(0,-0.25) -- ($(p1m)+(0.65,-0.5)$) arc (90:270:0.125) -- ($(p0p)+(3,0.5)$);
\end{scope}

\foreach \i/\y in {1/2.75} {
  \shade[left color=gray, right color=white, opacity=0.75, shading angle=145] (0,\y) -- ++(1.5,-0.75) -- ++(3,0) -- coordinate[midway] (alpha\i) ++(-1.5,0.75) -- ++(-3,0);
  \draw (0,\y) -- coordinate[midway] (p\i) ++(1.5,-0.75) -- ++(3,0) ++(-1.5,0.75) -- ++(-3,0);
}

\begin{scope}
  \draw[thin,blue] (p1p) ++(3,0) -- ++(-1.2,0) ++(-0.3,0) -- ++(-1.5,0) -- (topmid);
  \path (topmid) ++ (0.25,0) coordinate (tmoffset);
  \draw (tmoffset) to [bend left=25] ++(0.3,0);
  \path (tmoffset) to[bend right=25] ++(0.3,0) -- ($(p1p)+(0.55,0.7)$) arc (180:270:0.15) -- ++(2.3,0) -- ++(0,-0.3) -- ++(-1.2,0) -- ++(0,-0.5) coordinate (hole) to[bend left=25] ++(-0.3,0) -- ++(0,0.5) -- ++(-1.05,0) arc (270:180:0.2) -- (tmoffset);
  \draw[thin,fill=white] (hole) to[bend left=25] ++(-0.3,0) to[bend left=25] ++(0.3,0);
  \shade[left color=gray, right color=white, opacity=0.75, shading angle=145] (tmoffset) to[bend right=25] ++(0.3,0) -- ($(p1p)+(0.55,0.7)$) arc (180:270:0.15) -- ++(2.3,0) -- ++(0,-0.3) -- ++(-1.2,0) -- ++(0,-0.5) coordinate (hole) to[bend left=25] ++(-0.3,0) -- ++(0,0.5) -- ++(-1.05,0) arc (270:180:0.2) -- (tmoffset);
  \draw (tmoffset) to[bend right=25] ++(0.3,0) -- ($(p1p)+(0.55,0.7)$) arc (180:270:0.15) -- ++(2.3,0) ++(0,-0.3) -- ++(-1.2,0) -- ++(0,-0.5) ++(-0.3,0) -- ++(0,0.5) -- ++(-1.05,0) arc (270:180:0.2) -- (tmoffset);
\end{scope}

\draw (0,-0.25) -- ++(0,4.25) -- coordinate[midway] (topmid) ++(1.5,-0.75) -- ++(0,-4.25) -- ++(-1.5,0.75) coordinate[midway] (botmid);
\end{scope}
\end{tikzpicture}
\caption{Top left: the rational Seifert surface $F \subset M$.  Top right: the curve $c$ in the boundary of the exterior $E_F$.  Bottom left: the push-off $c' \subset M$ of the curve $c \subset \partial E_F$.  Bottom right: stabilizing $F$ to get a new rational Seifert surface $F' \subset M$.}
\label{fig:rational-surface-curve}
\end{figure}
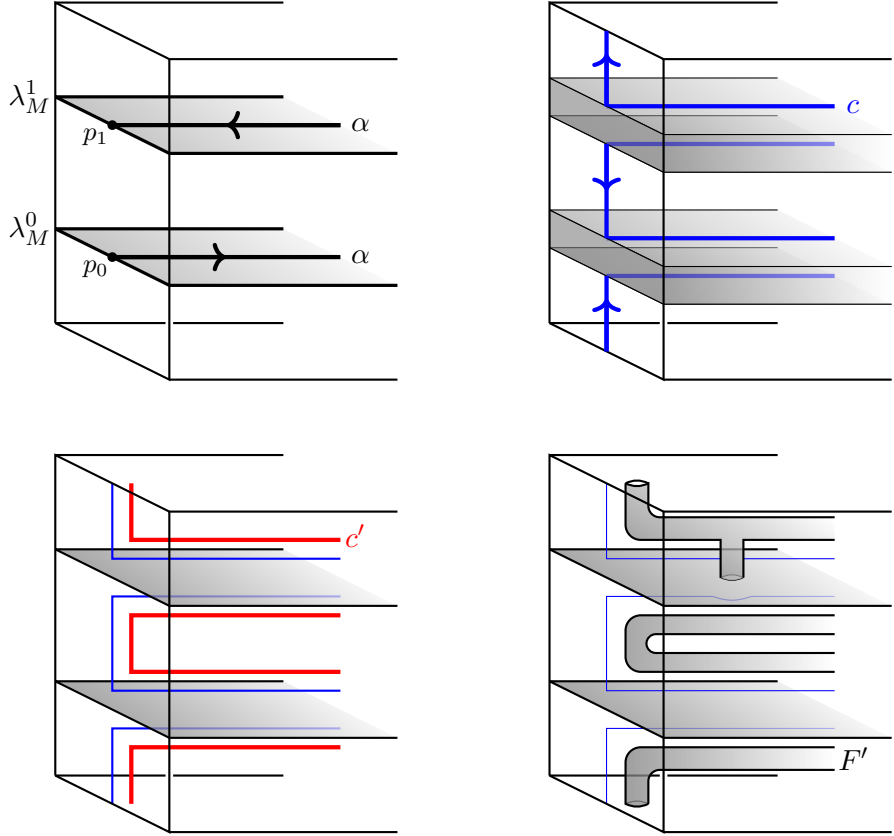

Next, we take a collar neighborhood $\partial E_F \times [-1,0] \subset E_F$ of the boundary of $E_F$, which we identify in these coordinates as $\partial E_F \times \{0\}$.  Then $c = c\times \{0\}$ and $c' = c\times \{-1\}$ cobound an annulus in $E_F$, namely the product $c\times[-1,0]$.  We take an arc $\beta$ connecting $c'$ to $F$ in the interior of $M$, chosen so that $\beta$ intersects the annulus $c \times [-1,0]$ in a separating arc.  Then we stabilize $F$ to get a new rational Seifert surface $F'$, with $g(F') = g(F)+1$, by attaching the boundary of a small tubular neighborhood of $c' \cup \beta$, as shown in the bottom row of Figure~\ref{fig:rational-surface-curve}; we also perturb the arc $\alpha_- \subset F\times\{-1\}$ slightly so that it avoids this neighborhood.  The end result is that we have a properly embedded disk $D$ in the exterior $E_{F'} \cong M \setminus \big(F'\times(-1,1)\big)$ of $F'$, consisting of the annulus $c\times[-1,0] \subset E_F$ minus a neighborhood of the arc $\beta$.  The intersection
\[ \partial D \cap \partial M = c \cap \partial M \]
consists of the two chosen arcs from $p_1 \times \{\pm1\}$ to $p_0 \times \{\mp1\}$, and the rest of $\partial D$ consists of a pair of properly embedded arcs
\begin{align*}
\alpha'_+ \times \{+1\} &\subset F' \times \{+1\}, \\
\alpha'_- \times \{-1\} &\subset F' \times \{-1\}
\end{align*}
from $p_1 \times \{\pm1\} \in \lambda^1_M \times \{\pm1\}$ to $p_0 \times \{\pm1\} \in \lambda^0_M \times \{\pm1\}$.

We are now ready to construct the desired map $f: M \to N$, where $N$ is the mapping torus
\[ M_{\psi_n} = \frac{A \times [0,1]}{(x,1) \sim (\psi_n(x),0)} \]
as described in Lemma~\ref{lem:klein-bottle-mapping-torus}.  We start by choosing a map
\[ g: (F',\partial F') \to (A, \partial A) \]
as follows: we choose an identification of $\partial F'$ with the two components of $\partial A$, and then extend this by sending the arc $\alpha'_- \subset F'$ homeomorphically onto some properly embedded arc $\gamma$ connecting the components of $\partial A$.  We extend this to collar neighborhoods of each, getting a partially defined homeomorphism
\[ g: N(\partial F' \cup \alpha'_-) \xrightarrow{\cong} N(\partial A \cup \gamma) \]
as shown in Figure~\ref{fig:seifert-to-annulus}.  This sends the circle
\[ \partial\left( N(\partial F' \cup \alpha'_-) \right) \setminus \partial F' \]
homeomorphically to a circle in $A$ that bounds a disk, and this circle bounds the portion of $F'$ on which $g$ has not yet been defined, so we now use Lemma~\ref{lem:pinch-sphere} to extend $g$ to the rest of $F'$.
\begin{figure}
\begin{tikzpicture}
\path (-2,0.35) coordinate (lhole) ++(2,0) coordinate (rhole);
\begin{scope}
\clip (lhole) -- ++(0,-0.5) -- ++(3,0) -- ++(0,3) -- ++(-3,0) -- cycle;
\draw[line width=4pt,blue] (lhole) to[bend left=30] ++(1,0);
\draw[line width=4pt,blue] (rhole) to[bend left=30] ++(1,0);
\end{scope}
\draw[fill=gray!20, fill opacity=0.5] (rhole) to[bend right=30] ++(1,0) -- ++(0,2) arc (0:180:1.5) -- ++(0,-2) to[bend right=30] ++(1,0) -- ++(0,0.75) arc (180:0:0.5) -- ++(0,-0.75);
\begin{scope}
\clip (lhole) ++(-0.4pt,0) -- ++(0,-0.5) -- ++(3,0) -- ++(0.8pt,0) -- ++(0,3) -- ++(-3,0) -- ++(-0.4pt,0) -- cycle;
\draw[line width=4pt, blue] (lhole) ++(-0.1,0.05) to[bend right=30] ++(1.1,-0.05) -- ++(0,0.75) arc (180:0:0.5) node[midway,above,inner sep=0] {\small$\alpha'_-$} -- ++(0,-0.75) to[bend right=30] ++(1.1,0.05);
\begin{scope}
\draw (-1,2.6) coordinate (hole) to[bend right=60] ++(1,0);
\clip (hole) to[bend right=60] ++(1,0) -- ++(0,1) -- ++(-1,0) -- ++(0,-1);
\draw[fill=white] (hole) ++ (0,-0.25) to[bend left=60] ++(1,0) -- ++(0,-1) -- ++(-1,0) -- ++(0,1);
\draw (hole) to[bend right=60] ++(1,0);
\end{scope}
\end{scope}
\draw[->] (1.5,2) -- ++(1,0);
\draw[line width=4pt, blue, fill=gray!20, even odd rule] (5,2) coordinate (acenter) circle (0.75) circle (2);
\draw[line width=4pt, blue] (acenter) ++(-0.75,0) -- node[midway,above,inner sep=1pt] {\small$\gamma$} ++(-1.25,0);
\end{tikzpicture}
\caption{Constructing a degree-1 map $g: (F',\partial F') \to (A, \partial A)$.}
\label{fig:seifert-to-annulus}
\end{figure}
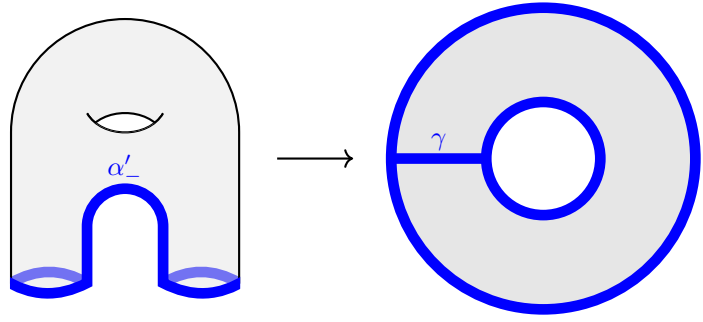

We now define $f: M \to N$ on the union of $\partial M$ and the rational Seifert surface $F'$ as follows.  We first choose a homeomorphism
\[ f|_{\partial M}: \partial M \to \partial N \]
that takes the two rational longitudes $\lambda^i_M$ to the components of
\[ \partial A \times \{1\} \subset \frac{A \times [0,1]}{(x,1) \sim (\psi_n(x),0)} \cong N. \]
Having done so, we use the above map $g: (F',\partial F') \to (A,\partial A)$ to set
\[ f(x) = (g(x),1) \in \frac{A \times [0,1]}{(x,1) \sim (\psi_n(x),0)} \]
for all $x \in F'$.  We can extend $f$ to a collar neighborhood of $\partial M$, and then to a neighborhood $F' \times [-1,1]$ of $F'$ such that
\begin{align*}
f(F' \times \{1\}) &\subset A \times \{\epsilon\}, \\
f(F' \times \{-1\}) &\subset A \times \{1-\epsilon\}
\end{align*}
where $\epsilon > 0$ is small (say $\epsilon=\frac{1}{10}$ for concreteness).
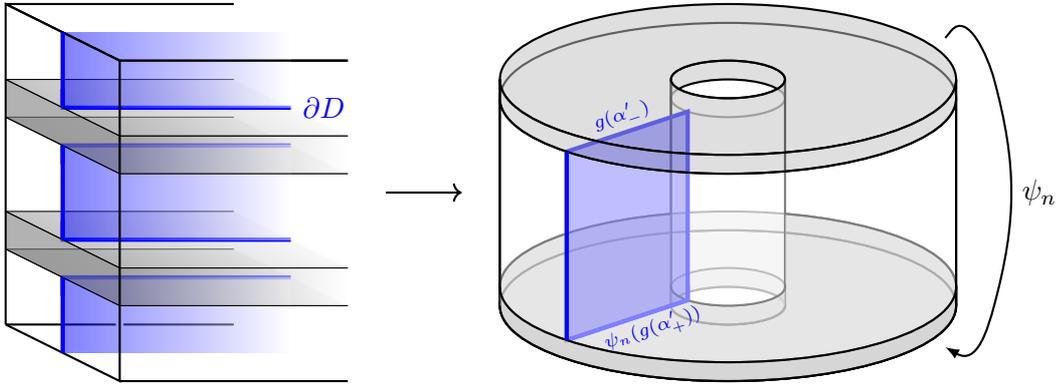
\begin{figure}
\begin{tikzpicture}
\begin{scope}
\foreach \p in {(1.5,3.25),(1.5,-1)} {\draw \p -- ++(3,0); }
\draw (0,-0.25) -- ++(1.45,0) ++(0.1,0) -- ++(1.45,0);
\foreach \j/\dy in {m/-0.25,p/0.25} {
  \foreach \i/\y in {0/1,1/2.75} {
    \shade[left color=gray, right color=white, opacity=0.5, shading angle=145] (0,\y)  ++(0,\dy) -- ++(1.5,-0.75) -- ++(3,0) -- coordinate[midway] (alpha\i) ++(-1.5,0.75) -- ++(-3,0);
    \draw[thin] (0,\y) ++ (0,\dy) -- coordinate[midway] (p\i) ++(1.5,-0.75) -- ++(3,0) ++(-1.5,0.75) -- ++(-3,0);
  }
  \begin{scope}[every path/.style={ultra thick,blue}]
    \draw (p0) coordinate (p0\j) ++(3,0) -- ++(-3,0) -- ++(0,\dy);
    \draw (p1) coordinate (p1\j) ++(3,0) -- ++(-3,0) -- ++(0,\dy);
  \end{scope}
}
\shade[left color=blue, right color=white, fill opacity=0.5] (p0m) rectangle ++(3,-1);
\foreach \y in {1} {
  \draw[thin,fill=gray, fill opacity=0.6] (0,\y) ++(0,0.25) -- ++(1.5,-0.75) coordinate (nw) -- ++(0,-0.5) -- ++(-1.5,0.75) -- cycle;
  \shade[left color=gray, right color=white, opacity=0.6] (nw) rectangle ++(3,-0.5);
  \draw[thin] (nw) -- ++(3,0) ++(0,-0.5) -- ++(-3,0);
}
\shade[left color=blue, right color=white, fill opacity=0.5] (p0p) rectangle ($(p1m)+(3,0)$);
\foreach \y in {2.75} {
  \draw[thin,fill=gray, fill opacity=0.6] (0,\y) ++(0,0.25) -- ++(1.5,-0.75) coordinate (nw) -- ++(0,-0.5) -- ++(-1.5,0.75) -- cycle;
  \shade[left color=gray, right color=white, opacity=0.6] (nw) rectangle ++(3,-0.5);
  \draw[thin] (nw) -- ++(3,0) ++(0,-0.5) -- ++(-3,0);
}
\shade[left color=blue, right color=white, fill opacity=0.5] (p1p) rectangle ++(3,1);
\draw (0,4) -- ++(3,0) (0,4) -- (1.5,3.25) -- ++(3,0);

\draw (0,-0.25) -- ++(0,4.25) -- coordinate[midway] (topmid) ++(1.5,-0.75) -- ++(0,-4.25) -- ++(-1.5,0.75) coordinate[midway] (botmid);
\begin{scope}[every path/.style={ultra thick,blue}]
  \draw (p0p) -- coordinate[midway] (midarrow) (p1m);
  \draw (p0m) -- coordinate[pos=0.3] (botarrow) (botmid);
  \draw (p1p) -- coordinate[pos=0.7] (toparrow) (topmid);
  \path (p1p) ++(3,0) node[right] {$\partial D$};
\end{scope}
\end{scope}
\draw[->] (5,1.5) -- ++(1,0);
\begin{scope}[xshift=9.5cm]
\draw (-3,2.75) arc (180:0:3 and 1);
\draw[fill=gray!50,fill opacity=0.5] (-0.75,2.75) arc (180:0:0.75 and 0.25) -- ++(0,0.25) arc (0:180:0.75 and 0.25) ++(0,-0.25);
\draw[fill=gray!50,fill opacity=0.5, even odd rule] (0,0) ellipse (3 and 1) ellipse (0.75 and 0.25);
\draw[fill=gray!50,fill opacity=0.5] (-3,0) ++(0,0.25) arc (180:0:3 and 1) -- ++(0,-0.25) arc (0:180:3 and 1);
\foreach \i in {-3,-0.75,0.75,3} { \draw (\i,0) -- ++(0,3); }
\draw (-0.75,2.75) arc (180:360:0.75 and 0.25);
\draw[fill=gray!50,fill opacity=0.5] (0.75,0) arc (0:180:0.75 and 0.25) -- ++(0,0.25) arc (180:0:0.75 and 0.25) ++(0,-0.25);
\draw[fill=white,fill opacity=0.5] (-0.75,3) -- ++(0,-3) arc (180:360:0.75 and 0.25) -- ++(0,3) arc (360:180:0.75 and 0.25);
\draw[fill=gray!50,fill opacity=0.5] (0.75,0) arc (360:180:0.75 and 0.25) -- ++(0,0.25) arc (180:360:0.75 and 0.25) ++(0,-0.25);
\begin{scope}
  \clip (-3.1,1.3) -- (-0.75,1.3) -- (-0.75,0.25) arc (180:360:0.75 and 0.25) -- (0.75,1.3) -- (3.1,1.3) -- (3.1,0.2) arc(360:180:3.1 and 1) -- cycle;
  \draw[fill=gray!50,fill opacity=0.5, even odd rule] (0,0.25) ellipse (3 and 1) ellipse (0.75 and 0.25);
\end{scope}
\draw[blue,fill=blue,fill opacity=0.5,ultra thick] (0,0.25) ++ (225:0.75 and 0.25) -- ++(0,2.5) -- coordinate[midway] (aminus) ++(225:2.25 and 0.75) -- ++(0,-2.5) -- coordinate[pos=0.67] (aplus) ++(45:2.25 and 0.75);
\draw[fill=gray!50,fill opacity=0.5, even odd rule] (0,3) ellipse (3 and 1) ellipse (0.75 and 0.25);
\draw[blue,ultra thick] (0,0.25) ++(225:2.97 and 0.99) -- ++(45:2.22 and 0.74) -- ++(0,0.2);
\draw[fill=white,fill opacity=0.5] (-3,2.75) -- ++(0,-2.75) arc (180:360:3 and 1) -- ++(0,2.75) arc (360:180:3 and 1);
\node[above,blue,inner sep=1.5pt,rotate=18.435] at (aminus) {\tiny$g(\alpha'_-)$};
\node[below,blue,inner sep=1.5pt,rotate=18.435] at (aplus) {\tiny$\psi_n(g(\alpha'_+))$};
\draw[fill=gray!50,fill opacity=0.5] (-3,0) ++(0,0.25) arc (180:360:3 and 1) -- ++(0,-0.25) arc (360:180:3 and 1);
\draw[looseness=1,-latex] (0,3) ++ (30:3.3 and 1.1) -- ++(75:0.3 and 0.1) to[out=50,in=105] (3.5,3) to[out=285,in=75] node[midway,right] {$\psi_n$} ++(0,-3) to[out=255,in=-30] ($(330:3.3 and 1.1)+(300:0.3 and 0.1)$) -- ++(120:0.3 and 0.1);
\draw[fill=gray!50,fill opacity=0.5] (-3,3) arc (180:360:3 and 1) -- ++(0,-0.25) arc (360:180:3 and 1) -- ++(0,0.25);
\draw[blue,ultra thick] (0,2.75) ++ (225:3 and 1) -- ++(0,-2.5);
\end{scope}
\end{tikzpicture}
\caption{The map $f: M \to N$, as partially defined on a neighborhood of $\partial M \cup F'$.}
\label{fig:partial-pinch-map}
\end{figure}%
This is illustrated in Figure~\ref{fig:partial-pinch-map}.

We note that so far the image of $f$, which has been defined on a neighborhood of $\partial M \cup F'$, is the union of
\[ A \times \left([0,\epsilon] \cup [1-\epsilon,1]\right) \]
and a neighborhood of $\partial N$.  The complement of that image is a solid torus $V \subset N$, and the properly embedded disk $D \subset E_{F'}$ has its boundary $\partial D$ sent to an essential curve in $\partial V$, consisting of
\begin{itemize}
\item one arc in each component of $\partial A \times [\epsilon,1-\epsilon]$;
\item the image $g(\alpha'_-)\times\{1-\epsilon\} = \gamma \times \{1-\epsilon\}$ of the arc $\alpha'_-\times \{-1\} \subset F'\times\{-1\}$;
\item the image $\psi_n(g(\alpha'_+)) \times \{\epsilon\}$ of the arc $\alpha'_+ \times \{+1\} \subset F'\times\{+1\}$.
\end{itemize}
  The curve $f(\partial D)$ may not bound a disk in $V$, but if we change the parameter $n$ in the monodromy $\psi_n$, then its intersection with $A \times \{\epsilon\}$ changes by the corresponding number of Dehn twists along the core of that annulus.  Thus by a suitable choice of $n$ we can arrange for $f(\partial D)$ to be nullhomologous in $V$, hence null-homotopic in $V$; we apply a further homotopy, supported away from $\partial M$, so that $f|_{\partial D}$ is a homeomorphism sending $\partial D$ to the boundary of a properly embedded disk in $V$.  We extend $f$ across $D$ by sending it homeomorphically to that disk, and then further extend $f$ to a collar neighborhood of $D$.

At this point we have defined $f$ on a neighborhood of $\partial M \cup F' \cup D$, and the boundary of the subdomain where $f$ remains undefined is sent to a $2$-sphere in $V \subset N$ that bounds a ball.  We thus apply Lemma~\ref{lem:pinch-sphere} again to extend $f$ to the rest of $M$, and this completes the proof.
\end{proof}

\section{Splicing knots in manifolds with $2$-torsion homology} \label{sec:splice-torsion}

In this section we study images of character varieties in the pillowcase to understand what happens when we splice the complements of knots in $3$-manifolds whose homology is $2$-torsion.

\subsection{The nullhomologous case} \label{ssec:splice-nullhomologous}

Our main result here is a generalization of \cite[Theorem~8.3]{zentner}, which describes the case where $Y_1 \cong Y_2 \cong S^3$, using the methods of \cite{lpcz}.

\begin{theorem} \label{thm:splice-rep}
Let $Y_1$ and $Y_2$ be closed, orientable $3$-manifolds such that $H_1(Y_1;\Z)$ and $H_1(Y_2;\Z)$ are both $2$-torsion, and let $K_1 \subset Y_1$ and $K_2 \subset Y_2$ be non-trivial nullhomologous knots with irreducible complements.  We splice their exteriors $E_{K_1} = Y_1 \setminus N(K_1)$ and $E_{K_2} = Y_2 \setminus N(K_2)$ to form a closed 3-manifold
\[ Y = E_{K_1} \cup_\partial E_{K_2}, \]
gluing the meridian and longitude $\mu_1$ and $\lambda_1$ in $\partial E_{K_1}$ to the longitude and meridian $\lambda_2$ and $\mu_2$ in $\partial E_{K_2}$, respectively.  Then there is a representation
\[ \rho: \pi_1(Y) \to \SU(2) \]
with non-abelian image.
\end{theorem}

\begin{proof}
Proposition~\ref{prop:solid-torus-pinch} gives us degree-1 maps
\[ Y \to E_{K_1}(\lambda_2) \cong Y_1 \quad\text{and}\quad Y \to E_{K_2}(\lambda_1) \cong Y_2, \]
which induce surjections $\pi_1(Y) \to \pi_1(Y_i)$ for $i=1,2$.  If there is some non-abelian representation $\pi_1(Y_i) \to \SU(2)$ then we can compose it with the surjection from $\pi_1(Y)$ to get the desired $\rho: \pi_1(Y) \to \SU(2)$.  Thus we may assume from now on that both $Y_1$ and $Y_2$ are $\SU(2)$-abelian.

Since neither $K_1$ nor $K_2$ is unknotted, we apply Theorem~\ref{thm:2-torsion-zero-surgery} to see that their zero-surgeries have non-trivial instanton homology: if $w_\ell \in H^2( (Y_\ell)_0(K_\ell); \Z)$ is Poincar\'e dual to a meridian of $K_\ell$, then
\[ I^{w_\ell}_*\big((Y_\ell)_0(K_\ell)\big) \neq 0 \]
for $\ell=1,2$.  Noting that each $Y_\ell$ is $\SU(2)$-abelian, Theorem~\ref{thm:pillowcase-loop} says that the pillowcase images
\[ j(X(E_{K_\ell})) \subset \cP \]
contain homologically essential loops.  This means that for $\ell=1,2$ we can find continuous paths
\[ \gamma^\ell_t = (\alpha^\ell_t,\beta^\ell_t): [0,1] \to [0,\pi] \times [0,2\pi] \]
such that
\begin{itemize}
\item $\beta^\ell_0 = 0$, $\beta^\ell_1 = 2\pi$, and $0 < \beta^\ell_t < 2\pi$ for $0 < t < 1$;
\item $0 < \alpha^\ell_t < \pi$ for $0 < t < 1$;
\item for each $t\in[0,1]$, there is a representation $\rho^\ell_t: \pi_1(E_{K_\ell}) \to \SU(2)$ such that
\begin{align*}
\rho^\ell_t(\mu_\ell) &= \begin{pmatrix} e^{i\alpha^\ell_t} & 0 \\ 0 & e^{-i\alpha^\ell_t} \end{pmatrix}, &
\rho^\ell_t(\lambda_\ell) &= \begin{pmatrix} e^{i\beta^\ell_t} & 0 \\ 0 & e^{-i\beta^\ell_t} \end{pmatrix};
\end{align*}
\item for $0 < t < 1$, each $\rho^\ell_t$ is irreducible, since $\rho^\ell_t(\lambda_\ell) \neq 1$.
\end{itemize}
Since $H_1(Y_1;\Z)$ is $2$-torsion, Lemma~\ref{lem:corners-limit} also tells us that
\begin{itemize}
\item $0 < \alpha^1_t < \pi$ for all $t \in [0,1]$,
\end{itemize}
since $(\alpha^1_0,\beta^1_0) = (\alpha^1_0,0)$ is a limit point of $(\alpha^1_t,\beta^1_t) \in j(X^\irr(E_{K_1}))$ as $t$ approaches $0$ from above, and likewise for $(\alpha^1_1,\beta^1_1)$.

Now if we let $\tau = \inf\{ t \in[0,1] \mid \beta^2_t = \pi \}$, so $0 < \tau < 1$, then the transposed path
\[ \tilde\gamma^2_t = (\beta^2_t,\alpha^2_t): [0,\tau] \to [0,\pi] \times [0,2\pi] \]
starts on the line $\{0\} \times [0,2\pi]$ and ends on the line $\{\pi\} \times [0,2\pi]$, with second coordinate $\alpha^2_t \in (0,\pi)$ for all $t \in (0,\tau]$.  Then $\tilde\gamma^2_t$ separates the rectangle $[0,\pi]\times[0,2\pi]$, with the subsets
\[ (0,\pi) \times \{0\} \quad\text{and}\quad (0,\pi) \times \{2\pi\} \]
in different path components of the complement of its image.  These subsets contain $\gamma^1_0 = (\alpha^1_0,\beta^1_0)$ and $\gamma^1_1 = (\alpha^1_1,\beta^1_1)$ respectively, so $\tilde\gamma^2$ must intersect the path $\gamma^1$ at some point
\[ \gamma^1_t = (\alpha^1_t,\beta^1_t), \quad 0 < t < 1. \]
Taking $\tilde{t} \in [0,\tau]$ so that $\gamma^1_t = \tilde\gamma^2_{\tilde{t}}$, it follows that
\begin{align*}
\rho^1_t(\mu_1) = \rho^2_{\tilde{t}}(\lambda_2) &= \begin{pmatrix} e^{i\alpha^1_t} & 0 \\ 0 & e^{-i\alpha^1_t} \end{pmatrix}, \\
\rho^1_t(\lambda_1) = \rho^2_{\tilde{t}}(\mu_2) &= \begin{pmatrix} e^{i\beta^1_t} & 0 \\ 0 & e^{-i\beta^1_t} \end{pmatrix}. 
\end{align*}
Thus $\rho^1_t$ and $\rho^2_{\tilde{t}}$ agree on the torus $\partial E_{K_1} = \partial E_{K_2}$ inside $Y$, and so they glue together to give a representation
\[ \rho: \pi_1(Y) \to \SU(2). \]
This representation restricts to $\pi_1(E_{K_1})$ as the irreducible $\rho^1_t$, where $0 < t < 1$, and its restriction to $\pi_1(E_{K_2})$ is likewise irreducible since $\rho(\lambda_2) = \rho(\mu_1) = \rho^1_t(\mu_1)$ is not the identity.  Thus $\rho$ is irreducible as well.
\end{proof}

\subsection{The homologically essential case} \label{ssec:splice-essential}

In this subsection, we consider what happens if we splice two knot complements where one of the knots is homologically essential.  We will ultimately prove the following.

\begin{theorem} \label{thm:splice-essential-rep}
Let $K_1 \subset Y_1$ and $K_2 \subset Y_2$ be knots in rational homology spheres such that $H_1(Y_1;\Z)$ and $H_1(Y_2;\Z)$ are both $2$-torsion.  Suppose that the exteriors $E_{K_1}$ and $E_{K_2}$ are irreducible, and that
\begin{itemize}
\item The rational longitude $\lambda_1 \subset \partial E_{K_1}$ has order $2$ in $H_1(E_{K_1};\Z)$;
\item $K_2$ is nullhomologous, with irreducible, boundary-incompressible complement.
\end{itemize}
We form a closed $3$-manifold
\[ Y = E_{K_1} \cup_{\partial} E_{K_2} \]
by splicing the exteriors along their boundaries so that $\mu_1 \sim \lambda_2$ and $\lambda_1 \sim \mu_2$.  Then there is a representation
\[ \rho: \pi_1(Y) \to \SU(2) \]
with non-abelian image.
\end{theorem}

Unlike in Theorem~\ref{thm:splice-rep}, one key obstacle is that we cannot make use of a degree-1 pinching map
\[ Y \to E_{K_2}(\lambda_1) \cong Y_2, \]
because we may not be able to collapse $E_{K_1}$ onto a solid torus.  However, since the rational longitude $\lambda_{K_1}$ has order $2$ in homology, we can use Proposition~\ref{prop:klein-bottle-pinch} to pinch it to the next best thing, the twisted $I$-bundle over the Klein bottle.

\begin{proof}[Proof of Theorem~\ref{thm:splice-essential-rep}]
Suppose to the contrary that $Y$ is $\SU(2)$-abelian.  We first note that by the Mayer--Vietoris sequence, the homology $H_1(Y)$ is isomorphic to
\[ \frac{H_1(E_{K_1}) \oplus H_1(E_{K_2})}{\mu_1\sim \lambda_2,\ \lambda_1\sim \mu_2} \cong \frac{H_1(Y_1) \oplus H_1(E_{K_2})}{\lambda_1 \sim \mu_2} \cong H_1(Y_1) \oplus H_1(Y_2), \]
where we first use the fact that $[\lambda_2] = 0$ in $H_1(E_{K_2})$, and then that $H_1(E_{K_2}) \cong H_1(Y_2) \oplus \Z$ with the $\Z$ summand generated by $[\mu_2]$.  In particular $H_1(Y;\Z)$ is $2$-torsion.

Now we apply Proposition~\ref{prop:klein-bottle-pinch} to construct a degree-1 map
\[ Y \to N \cup_\partial E_{K_2}, \]
where $N$ is the twisted $I$-bundle over the Klein bottle.  The rational longitude $\lambda_0$ of $N$ is still glued to $\mu_2$, since this map preserves rational longitudes, but a priori we only know that some curve $\mu \subset \partial N$ that is dual to $\lambda_0$ has been glued to $\lambda_2$.  However, we can now pinch $E_{K_2}$ to a solid torus as in Proposition~\ref{prop:solid-torus-pinch}, so we have a composition of degree-1 maps
\[ Y \to N \cup_\partial E_{K_2} \to N(\lambda_2) = N(\mu). \]
This induces a surjection on $\pi_1$ and hence on $H_1$, so we conclude that $H_1(N(\mu))$ is $2$-torsion since $H_1(Y)$ is; and that both 
\[ Y' = N \cup_\partial E_{K_2} \]
and $N(\mu)$ are $\SU(2)$-abelian, since $Y$ is.  Since $N(\mu)$ is $\SU(2)$-abelian and its first homology is $2$-torsion, Proposition~\ref{prop:klein-bottle-pillowcase} says that $\mu$ is the unique slope $\mu_0$ such that $N(\mu_0) \cong \RP^3\#\RP^3$.

We now claim that $Y_2$ must be $\SU(2)$-abelian.  We know that $Y'$ is $\SU(2)$-abelian, and since the slopes $\mu_0$ and $\lambda_0$ are glued to $\lambda_2$ and $\mu_2$ respectively, we have
\[ \pi_1(Y') = \frac{\langle a,b \mid aba^{-1}=b^{-1}\rangle \ast \pi_1(E_{K_2})}{a^2=\mu_0\sim \lambda_2,\ b=\lambda_0\sim\mu_2}. \]
Suppose that there is a non-abelian representation $\pi_1(Y_2) \to \SU(2)$, or equivalently some non-abelian $\rho: \pi_1(E_{K_2}) \to \SU(2)$ with $\rho(\mu_2) = 1$.  Then since every element of $\SU(2)$ has a square root, we can extend $\rho$ to $\pi_1(Y')$ by setting $\rho(b) = 1$ and letting $\rho(a)$ be some square root of $\rho(\lambda_2)$.  This contradicts the fact that $Y'$ is $\SU(2)$-abelian, so $Y_2$ is $\SU(2)$-abelian after all.

Now we recall from Proposition~\ref{prop:klein-bottle-pillowcase} that for every $t \in \R/2\pi\Z$, there is a representation $\rho_t: \pi_1(N) \to \SU(2)$ with
\begin{align*}
\rho_t(\mu_0) &= -1, &
\rho_t(\lambda_0) &= \begin{pmatrix} e^{it} & 0 \\ 0 & e^{-it} \end{pmatrix},
\end{align*}
and that this has non-abelian image if $t \not\in \pi\Z$.  If for some $t$ we can find a representation $\rho_{K_2}: \pi_1(E_{K_2}) \to \SU(2)$ with
\begin{align*}
\rho_{K_2}(\mu_2) &= \begin{pmatrix} e^{it} & 0 \\ 0 & e^{-it} \end{pmatrix}, &
\rho_{K_2}(\lambda_2) &= -1,
\end{align*}
then we could glue this to the corresponding $\rho_t$ to get a representation $\rho': \pi_1(Y') \to \SU(2)$.  But then $\rho_{K_2}$ has non-abelian image, since otherwise we would have $\rho_{K_2}(\lambda_2) = 1$, and so $\rho'$ must be non-abelian as well.  This would also contradict the fact that $Y'$ is $\SU(2)$-abelian.

In conclusion, we have shown that if $Y$ is $\SU(2)$-abelian then there cannot be any representations $\rho: \pi_1(E_{K_2}) \to \SU(2)$ with $\rho(\lambda_2) = -1$.  But such representations up to conjugacy generate $I^w_*\big( (Y_2)_0(K_2) \big)$, where $w$ is dual to a meridian of $K_2$, so the latter invariant must be zero.  On the other hand, we know that $Y_2$ is $\SU(2)$-abelian, and that $K_2 \subset Y_2$ is nullhomologous with irreducible, boundary-incompressible complement, so Theorem~\ref{thm:2-torsion-zero-surgery} says that $I^w_*((Y_2)_0(K_2)) \neq 0$.  This is a contradiction, so we conclude that the spliced manifold $Y$ could not have been $\SU(2)$-abelian after all.
\end{proof}

\section{Gluing complements of knots in sums of $\RP^3$} \label{sec:glue-zhs}

Our goal in this somewhat lengthy section is to prove the following theorem.

\begin{theorem} \label{thm:gluing-rep}
Let $K_1 \subset Y_1$ and $K_2 \subset Y_2$ be nullhomologous knots in rational homology spheres whose $2$-surgeries satisfy
\begin{align*}
(Y_1)_2(K_1) &\cong \#^k \RP^3, &
(Y_2)_2(K_2) &\cong \#^\ell \RP^3
\end{align*}
for some integers $k, \ell \geq 1$, and suppose that their exteriors $E_{K_1}$ and $E_{K_2}$ are irreducible and not solid tori.  We form a closed $3$-manifold
\[ Y = E_{K_1} \cup_{\partial} E_{K_2} \]
by gluing the exteriors along their boundaries so that $\mu_1 \sim \mu_2^{-1}$ and $\lambda_1 \sim \mu_2^2 \lambda_2$.  Then there is a representation
\[ \pi_1(Y) \to \SU(2) \]
whose restrictions to each of $\pi_1(E_{K_1})$ and $\pi_1(E_{K_2})$ have non-abelian image.
\end{theorem}

The gluing map used to construct $Y$ in Theorem~\ref{thm:gluing-rep} is not arbitrary: it produces a toroidal $3$-manifold whose homology is $2$-torsion, and we will eventually see in \S\ref{sec:knot-exteriors} that for such manifolds, this is essentially the only gluing map we need to consider that is not of the form $(\mu_1,\lambda_1) \sim (\lambda_2,\mu_2)$.  The proof of Theorem~\ref{thm:gluing-rep} will occupy the next several subsections.

\subsection{Knots with $\#^n \RP^3$ surgeries}

Suppose that $Y$ is a rational homology 3-sphere, and $K \subset Y$ is a nullhomologous knot such that $Y_2(K) \cong \#^n \RP^3$ for some $n\geq 1$.  If $Y_0(K)$ is irreducible then we have $I^w_*(Y_0(K)) \neq 0$, which we can use to understand something about the $\SU(2)$ character variety of the complement of $K$.  We wish to understand exactly when this happens, so that we can almost always guarantee that $I^w_*(Y_0(K))$ will be nonzero.

\begin{proposition} \label{prop:rp3-surgery-zero-surgery}
Let $Y$ be a rational homology 3-sphere, and suppose for some nullhomologous knot $K \subset Y$ with irreducible exterior $E_K = Y \setminus N(K)$ that either
\begin{itemize}
\item $Y_2(K) \cong \#^n \RP^3$, where $n \geq 1$; or
\item $Y_p(K)$ is a lens space of order $p$ for some prime $p$.
\end{itemize}
Then $Y_0(K)$ is irreducible unless $(Y,K) \cong (S^3,U)$.
\end{proposition}

\begin{proof}
We first consider the case where $Y_2(K) \cong \#^n \RP^3$ for some $n \geq 2$.  In this case, the Dehn filling $E_K(\mu^2\lambda)$ produces the connected sum $\#^n \RP^3$, which is reducible.  Since $E_K$ is irreducible, any pair of slopes producing reducible fillings must have distance $1$ \cite[Theorem~1.2]{gordon-luecke-reducible}; but then $\lambda$ has distance $2$ from $\mu^2\lambda$, so $E_K(\lambda) \cong Y_0(K)$ must be irreducible as well.

From now on we suppose that $Y_p(K) \cong L(p,q)$ for some prime $p$, which may be either $2$ or odd, so that $Y$ is an integral homology sphere.  We will suppose that $Y_0(K)$ is reducible.  Then the exterior $E_K$ has a reducible Dehn filling (of slope $0$) and a Dehn filling with finite fundamental group (of slope $p$), and these filling slopes have distance $p \geq 2$.  A theorem of Boyer and Zhang \cite[Theorem~1.2]{boyer-zhang-seminorms} thus asserts that one of the following must hold: either
\begin{itemize}
\item $E_K$ is a simple (i.e., irreducible and atoroidal) Seifert fibered manifold, or
\item $E_K$ is a cable on the twisted $I$-bundle over the Klein bottle.
\end{itemize}
In the latter case any Dehn filling of $E_K$ must contain a Klein bottle, but we know that there are no embedded Klein bottles in $\RP^3$ \cite{bredon-wood}.  There are also no Klein bottles in a lens space $L(p,q)$ where $p$ is odd: any Klein bottle $B$ would be non-separating, hence $[B]$ would be a nonzero class in $H_2(L(p,q);\Z/2\Z) \cong 0$.  Thus $E_K$ must be Seifert fibered instead.  We refer to \cite{scott} for the facts about Seifert fibered 3-manifolds that we will use below.

If we fix a Seifert fibration on $E_K$, then it extends over any Dehn filling of $\partial E_K$ as long as the filling in question is not along the fiber slope.  In particular, the Seifert fibration extends over either $Y_p(K) \cong L(p,q)$ or $Y_0(K)$.  If it extends over $Y_0(K)$ then we know that the only non-prime Seifert fibered space is $\RP^3 \# \RP^3$, so since $Y_0(K)$ is not a rational homology sphere it must be prime; then $Y_0(K)$ is reducible by assumption and also prime, so it must be $S^1\times S^2$.  In either case, every Seifert fibration on $L(p,q)$ and on $S^1\times S^2$ has base orbifold homeomorphic to $S^2$, so the fibration on $E_K$ has base orbifold homeomorphic to a disk.

Next, we claim that $Y_0(K) \cong S^1\times S^2$.  We have already argued that this is the case if the Seifert fibration on $E_K$ extends over $Y_0(K)$.  If it does not, then the longitude of $K$ must have been the fiber slope, and since the base orbifold of $E_K$ is orientable, it follows that $Y_0(K)$ is a connected sum of lens spaces and copies of $S^1\times S^2$ \cite[Proposition~2]{heil}.  Then from $H_1(Y_0(K);\Z)\cong\Z$ we must have $Y_0(K) \cong S^1\times S^2$ as claimed.

We have shown that the core of $0$-surgery on $K \subset Y$ is a knot $K' \subset S^1\times S^2$ that admits an $L(p,q)$ surgery, and whose exterior is Seifert fibered.  Baker, Buck, and Lecuona \cite[Theorem~1.18]{baker-buck-lecuona} showed that the only such knots are $(a,b)$-torus knots in $S^1\times S^2$, and that if $a\geq 2$ then the corresponding lens spaces are $L(na^2,nab+1)$ for $n\in\Z$; but these do not have homology of prime order, hence cannot be $L(p,q)$.  Thus $K'$ must be isotopic to $S^1 \times \{\pt\} \subset S^1\times S^2$, and it follows that $(Y,K) \cong (S^3,U)$.
\end{proof}

\begin{corollary} \label{cor:su2-abelian-rp3-surgery}
Let $K \subset Y$ be a nullhomologous knot with irreducible complement in a rational homology sphere, and suppose that $Y_2(K) \cong \#^n\RP^3$ for some $n \geq 1$ but that $(Y,K) \not\cong (S^3,U)$.  Then $Y$ is not $\SU(2)$-abelian.
\end{corollary}

\begin{proof}
Proposition~\ref{prop:rp3-surgery-zero-surgery} says that $Y_0(K)$ is irreducible, so if $w \in H^2(Y_0(K);\Z)$ is Poincar\'e dual to a meridian of $K$, then $I^w_*(Y_0(K)) \neq 0$ by Theorem~\ref{thm:irreducible-nonzero}.  The homology of $Y$ is $2$-torsion since
\[ H_1(Y;\Z) \oplus (\Z/2\Z) \cong H_1(Y_2(K);\Z) \cong (\Z/2\Z)^n \]
is $2$-torsion, so if $Y$ were $\SU(2)$-abelian then Theorem~\ref{thm:km-su2-generalized} would give us a non-abelian representation $\pi_1(Y_2(K)) \to \SU(2)$.  But this is impossible, since every representation of
\[ \pi_1(Y_2(K)) \cong \pi_1(\#^n\RP^3) \cong (\Z/2\Z) \ast \dots \ast (\Z/2\Z) \]
into $\SU(2)$ must send each $\Z/2\Z$ factor into $\{\pm1\}$ and thus have central image, so $Y$ must not be $\SU(2)$-abelian after all.
\end{proof}

\subsection{The pillowcase image of a knot with a $\#^n\RP^3$ surgery}

Suppose that $Y$ is a rational homology sphere and that $K \subset Y$ is a nullhomologous knot with irreducible complement such that $Y_2(K) \cong \#^n\RP^3$, and that $(Y,K) \not\cong (S^3,U)$.  Since Corollary~\ref{cor:su2-abelian-rp3-surgery} tells us that $Y$ is not $\SU(2)$-abelian, we cannot argue as in Lemma~\ref{lem:alpha-pi} that the pillowcase image
\[ i^*(X(E_K)) \subset X(T^2) \]
avoids the lines $\{\alpha=0\}$ and $\{\alpha=\pi\}$: there may be a non-abelian representation $\pi_1(Y)\to \SU(2)$ that sends the homotopy class of the longitude $\lambda$ to something non-trivial.  In particular, it no longer makes sense to talk about the image $j(X(E_K))$ in the cut-open pillowcase.  Thus in what follows we will stick to the pillowcase, identified as
\begin{equation} \label{eq:pillowcase-quotient}
X(T^2) \cong \frac{(\R/2\pi\Z) \times (\R/2\pi\Z)}{(\alpha,\beta) \sim (-\alpha,-\beta)}.
\end{equation}
We will also describe it in terms of a fundamental domain for the above quotient, namely as
\begin{equation} \label{eq:pillowcase-domain}
X(T^2) \cong \frac{[0,\pi] \times [0,2\pi]}{\left\{\begin{array}{c}(0,\beta) \sim (0,2\pi-\beta), \\ (\pi,\beta) \sim (\pi,2\pi-\beta), \\ (\alpha,0) \sim (\alpha,2\pi) \end{array}\right\}},
\end{equation}
which equips it with a quotient map $[0,\pi] \times [0,2\pi] \to X(T^2)$.

\begin{proposition} \label{prop:lines-of-slope-2}
Let $K \subset Y$ be a nullhomologous knot in a rational homology sphere with $Y_2(K) \cong \#^n\RP^3$ for some $n \geq 0$.  Then the pillowcase image
\[ i^*(X(E_K)) \subset X(T^2) \]
does not contain any points $(\alpha,\beta)$ with $2\alpha+\beta \in 2\pi\Z$, except for $(0,0)$ and $(\pi,0)$.  Moreover, its intersection with the line
\[ \{ 2\alpha + \beta \equiv \pi \!\!\!\pmod{2\pi} \} \subset X(T^2) \]
is connected and contains the point $(\frac{\pi}{2},0)$.
\begin{figure}
\begin{tikzpicture}[style=thick]
\begin{scope}
  \draw plot[mark=*,mark size = 0.5pt] coordinates {(0,0)(3,0)(3,6)(0,6)} -- cycle; 
  \begin{scope}[decoration={markings,mark=at position 0.55 with {\arrow[scale=1]{>>}}}]
    \draw[postaction={decorate}] (0,0) -- (0,3);
    \draw[postaction={decorate}] (0,6) -- (0,3);
  \end{scope}
  \begin{scope}[decoration={markings,mark=at position 0.575 with {\arrow[scale=1]{>>>}}}]
    \draw[postaction={decorate}] (3,0) -- (3,3);
    \draw[postaction={decorate}] (3,6) -- (3,3);
  \end{scope}
  \begin{scope}[decoration={markings,mark=at position 0.75 with {\arrow[scale=1]{>>>>}}}]
    \draw[postaction={decorate}] (3,6) -- (0,6);
    \draw[postaction={decorate}] (3,0) -- (0,0);
  \end{scope}
  \draw[dotted] (0,3) -- (3,3);
  \draw[thin,|-|] (0,-0.4) node[below] {\small$0$} -- node[midway,inner sep=1pt,fill=white] {$\alpha$} ++(3,0) node[below] {\small$\vphantom{0}\pi$};
  \draw[thin,|-|] (-0.3,0) node[left] {\small$0$} -- node[midway,inner sep=1pt,fill=white] {$\beta$} ++(0,6) node[left] {\small$2\pi$};
  \begin{scope}[color=red,style=ultra thick,shorten >=-4pt, shorten <=-4pt]
    \draw[o-o] (0,6) -- node[above,sloped] {avoid this line} (3,0);
  \end{scope}
  \begin{scope}[color=blue,style=ultra thick]
    \draw[loosely dotted, very thick] (0,3) -- coordinate[pos=0.6] (x) (1.5,0) (1.5,6) -- coordinate[pos=0.4] (y) (3,3);
    \draw (x) -- node[above,sloped] {conn-} (1.5,0);
    \draw (1.5,6) -- node[above,sloped] {-ected} (y);
  \end{scope}
\end{scope}
\end{tikzpicture}
\caption{If $Y_2(K) \cong \#^n \RP^3$, then the image $i^*(X(E_K))$ must avoid the line $\{2\alpha+\beta\in 2\pi\Z\}$, except possibly at its endpoints, and its intersection with the line $\{2\alpha+\beta\equiv \pi\pmod{2\pi}\}$ must be connected.  (Note that the points $(\frac{\pi}{2},0)$ and $(\frac{\pi}{2},2\pi)$ are in fact the same.)}
\label{fig:slope-2}
\end{figure}
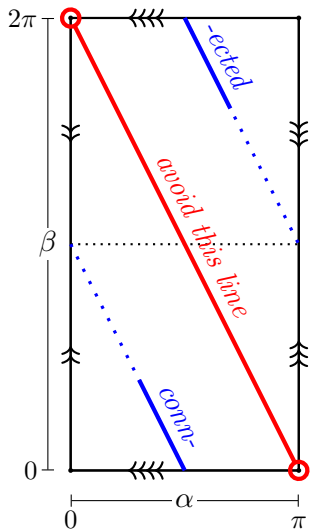
(See Figure~\ref{fig:slope-2}.)
\end{proposition}

\begin{proof}
Let $\rho: \pi_1(E_K) \to \SU(2)$ be a representation with $i^*([\rho]) = (\alpha,\beta)$ and $2\alpha+\beta \in \pi\Z$.  Then up to conjugacy we have
\[ \rho(\mu^2\lambda) = \begin{pmatrix} e^{i(2\alpha+\beta)} & 0 \\ 0 & e^{-i(2\alpha+\beta)} \end{pmatrix} = \pm1. \]
We will consider each value separately below.

If $\rho(\mu^2\lambda) = 1$, then $2\alpha+\beta\equiv 0 \pmod{2\pi}$, and $\rho$ factors through
\[ \frac{\pi_1(E_K)}{\llangle \mu^2\lambda \rrangle} \cong \pi_1(Y_2(K)) \cong \pi_1(\#^n\RP^3) \cong (\Z/2\Z)^{\ast n}. \]
Every homomorphism $(\Z/2\Z)^{\ast n} \to \SU(2)$ has central image, because each $\Z/2\Z$ factor must be sent to $\{\pm1\}$, so in particular this means that $\rho$ must have central image.  But then $\rho$ factors through $H_1(E_K;\Z)$, and thus it sends the nullhomologous $\lambda$ to $1$.  This is equivalent to $\beta\equiv 0\pmod{2\pi}$, and then $\alpha\in\pi\Z$ as claimed.

We assume from now on that $\rho(\mu^2\lambda) = -1$, so $2\alpha+\beta \equiv \pi \pmod{2\pi}$.  If the only such representations satisfy $(\alpha,\beta) = (\frac{\pi}{2},0)$ then there is nothing to show, so we will assume that $(\alpha,\beta)$ is different from $(\frac{\pi}{2},0)$.  In particular, since $\beta \not\equiv 0 \pmod{2\pi}$ we know that $\rho(\lambda) \neq 1$, and thus $\rho$ must have non-abelian image.

In this case, while $\rho$ itself no longer factors through $\pi_1(Y_2(K))$, the representation 
\[ \ad\rho: \pi_1(E_K) \to \SO(3) \]
does send $\mu^2\lambda$ to the identity.  This means that $\ad\rho$ factors as a composition
\[ \pi_1(E_K) \twoheadrightarrow \frac{\pi_1(E_K)}{\llangle \mu^2\lambda \rrangle} \cong (\Z/2\Z)^{\ast n} \xrightarrow{\phi} \SO(3). \]
It must also have non-trivial image, since otherwise the image of $\rho$ would have been abelian.

Now we let $x_1,\dots,x_n$ be generators of the $\Z/2\Z$ factors of $(\Z/2\Z)^{\ast n}$.  The map $\phi$ sends each $x_i$ to an element $\phi(x_i) \in \SO(3)$ of order at most $2$, hence either to the identity or to a $180$-degree rotation about some axis $L_i$; and it sends at least one $x_i$ to such a rotation, since $\ad\rho$ is non-trivial.  The space of such rotations is connected and homeomorphic to $\RP^2$, since each rotation is uniquely determined by its axis and vice versa.  Thus we can define a family of homomorphisms
\[ \phi_t: (\Z/2\Z)^{\ast n} \to \SO(3), \]
with $\phi_0 = \phi$ and $\phi_1$ having abelian image, as follows:
\begin{itemize}
\item If $\phi(x_i) = 1$ then we let $\phi_t(x_i) = 1$ for all $t \in [0,1]$.
\item If $\phi(x_i)$ is a $180$-degree rotation around an axis $L_i$, then we choose a path $\gamma_i: [0,1] \to \RP^2$ from $[L_i]$ to $[1:0:0]$ and let $\phi_t(x_i)$ be the $180$-degree rotation about $\gamma_i(t)$.
\end{itemize}
We see that $\phi_1$ has abelian image of order $2$, since it sends each $x_i$ to either $1$ or the $180$-degree rotation about the $x$-axis, and at least one of the $\phi_1(x_i)$ is a rotation.

The corresponding continuous family of homeomorphisms
\[ \bar\rho_t: \pi_1(E_K) \twoheadrightarrow \frac{\pi_1(E_K)}{\llangle \mu^2\lambda \rrangle} \cong (\Z/2\Z)^{\ast n} \xrightarrow{\phi_t} \SO(3) \]
satisfies $\bar\rho_t(\mu^2\lambda) = 1$ for all $t \in [0,1]$ by construction.  Moreover, we know that $\bar\rho_0 = \ad\rho$ lifts to a representation $\pi_1(E_K) \to \SU(2)$, so the obstruction $w_2(\bar\rho_0)$ to lifting must be zero, and then since $w_2(\bar\rho_t) = w_2(\bar\rho_0)$ by continuity, it follows that all of the $\bar\rho_t$ lift to a continuous family of representations
\[ \rho_t: \pi_1(E_K) \to \SU(2). \]

We now have $\ad \rho_t(\mu^2\lambda) = 1$, so $\rho_t(\mu^2\lambda)$ must be either $1$ or $-1$, and then
\[ \rho_t(\mu^2\lambda) = \rho_0(\mu^2\lambda) = -1 \]
for all $t$.  This says that the pillowcase images $i^*([\rho_t]) = (\alpha_t,\beta_t)$ all lie on the line $2\alpha+\beta \equiv \pi \pmod{2\pi}$.  Moreover, since the image in $\SO(3)$ of $\bar\rho_1 = \ad\rho_1$ has order $2$, it follows that the lift $\rho_1$ has cyclic image of order $4$ in $\SU(2)$.  But then $\rho_1$ has abelian image, so we must have $(\alpha_1,\beta_1) = (\frac{\pi}{2},0)$, and then the points $(\alpha_t,\beta_t)$ trace a continuous path in $i^*(X(E_K))$ from our original $(\alpha,\beta) = (\alpha_0,\beta_0)$ to $(\alpha_1,\beta_1) = (\frac{\pi}{2},0)$.

We conclude that the intersection
\[ i^*(X(E_K)) \cap \{ 2\alpha + \beta \equiv \pi \!\!\!\pmod{2\pi} \} \subset X(T^2) \]
is connected, since it contains a continuous path from every one of its points to $(\frac{\pi}{2},0)$.  It must also contain the point $(\frac{\pi}{2},0)$, as claimed, as the image of an abelian representation
\[ \pi_1(E_K) \twoheadrightarrow H_1(E_K;\Z) \cong H_1(Y) \oplus \Z \to \SU(2) \]
which is trivial on $H_1(Y)$ and sends the meridian generating the $\Z$ summand to $\left(\begin{smallmatrix} i & 0 \\ 0 & -i \end{smallmatrix}\right)$.
\end{proof}

\begin{proposition} \label{prop:avoid-slope-2}
Let $K \subset Y$ be a nullhomologous knot in a rational homology sphere with $Y_2(K) \cong \#^n\RP^3$ for some $n \geq 1$, and suppose that the exterior $E_K$ is irreducible and that $(Y,K) \not\cong (S^3,U)$.  Then the pillowcase image
\[ i^*(X(E_K)) \subset X(T^2) \]
satisfies exactly one of the following:
\begin{enumerate}
\item The image $i^*(X(E_K))$ contains the entire line $\{2\alpha+\beta \equiv \pi \pmod{2\pi}\}$.
\item The image $i^*(X(E_K))$ contains neither $P=(0,\pi)$ nor $Q=(\pi,\pi)$, and then it contains a homologically essential simple closed curve
\[ C \subset i^*(X(E_K)) \subset X(T^2) \setminus \{P,Q\} \]
that is disjoint from the line $\{2\alpha+\beta\in 2\pi\Z\}$.
\end{enumerate}
\end{proposition}

\begin{proof}
Proposition~\ref{prop:lines-of-slope-2} tells us that $i^*(X(E_K))$ contains all of $\{2\alpha+\beta \equiv \pi \pmod{2\pi}\}$ if and only if it contains both of the endpoints $P=(0,\pi)$ and $Q=(\pi,\pi)$, since its intersection with this line is connected.  We further observe that it contains $P$ if and only if it contains $Q$, since we can multiply a representation
\[ \rho: \pi_1(E_K) \to \SU(2) \]
with $\rho(\mu) = \pm1$ and $\rho(\lambda) = -1$ by a central character
\[ \chi: \pi_1(E_K) \twoheadrightarrow H_1(E_K) \cong H_1(Y) \oplus \Z \to \{\pm1\} \]
sending the meridian (as a generator of the $\Z$ summand) to $-1$ in order to get a new representation $\tilde\rho$ with $\tilde\rho(\mu) = \mp1$ and $\tilde\rho(\lambda) = -1$.  Thus if $i^*(X(E_K))$ does not contain the entire line $\{2\alpha+\beta\equiv\pi\}$, then it cannot contain either of $P$ or $Q$.

Now we suppose that $P,Q \not\in i^*(X(E_K))$.  Since $E_K$ is irreducible and $(Y,K) \not\cong (S^3,U)$, Proposition~\ref{prop:rp3-surgery-zero-surgery} and Theorem~\ref{thm:irreducible-nonzero} tell us that $I^w_*(Y_0(K)) \neq 0$, where $w \in H^2(Y_0(K))$ is Poincar\'e dual to a meridian of $K$.  Using the assumption that $P,Q \not\in i^*(X(E_K))$, we can now apply Proposition~\ref{prop:curve-in-pillowcase} to get the desired essential curve $C \subset i^*(X(E_K)) \setminus \{P,Q\}$.

Finally, suppose that the curve $C$ intersects the line $\{2\alpha+\beta\in 2\pi\Z\}$; according to Proposition~\ref{prop:lines-of-slope-2}, this can only happen at a point of the form $(k\pi,0)$ where $k$ is $0$ or $1$.  Parametrizing $C$ by a continuous, injective map
\[ f: \R/\Z \hookrightarrow i^*(X(E_K)) \hookrightarrow X(T^2) \]
so that $f(0) = (k\pi,0)$, we claim that there must be a sequence $t_n \to 0$ such that $f(t_n)$ is not on the line $\{\beta=0\}$: assuming otherwise, there is some $\epsilon > 0$ such that $f$ restricts to a continuous, injective map
\[ f|_{(-\epsilon,\epsilon)}: (-\epsilon,\epsilon) \hookrightarrow [0,\pi] \times \{0\} \hookrightarrow X(T^2), \]
with $0$ sent to an endpoint $(k\pi,0)$ of the line segment $[0,\pi] \times \{0\}$, and this is impossible.  Now since $f(t_n) \in i^*(X^\irr(E_K))$ and $t_n \to 0$, we see that $f(0) = (k\pi,0)$ is a limit point of the image $i^*(X^\irr(E_K))$.  But $H_1(Y)$ is $2$-torsion, since $H_1(Y) \oplus (\Z/2\Z) \cong (\Z/2\Z)^{\oplus n}$, so Lemma~\ref{lem:corners-limit} says this can only happen if $Y_2(K) \cong \#^n \RP^3$ is not $\SU(2)$-abelian, a contradiction.  We conclude that $C$ cannot pass through $(k\pi,0)$ after all.
\end{proof}

\subsection{Symmetries of the pillowcase} \label{ssec:pillowcase-symmetries}

We do not need to use instanton homology for the remainder of this section, since Propositions~\ref{prop:lines-of-slope-2} and \ref{prop:avoid-slope-2} will suffice for the proof of Theorem~\ref{thm:gluing-rep}.  Given a nullhomologous knot $K \subset Y$, we will therefore write
\[ I_K = i^*(X(E_K)) \subset X(T^2) \]
for the pillowcase image of the $\SU(2)$-character variety of $K$.

With the gluing map of Theorem~\ref{thm:gluing-rep} in mind, we now define a map
\[ \sigma: X(T^2) \to X(T^2) \]
in terms of the coordinates \eqref{eq:pillowcase-quotient}, by the formula
\[ \sigma(\alpha,\beta) = (-\alpha, 2\alpha + \beta) = (\alpha, 2\pi-(2\alpha+\beta)). \]
See Figure~\ref{fig:sigma-tau}.  It is straightforward to check that this is well-defined, that $\sigma^2 = \mathrm{Id}$ and thus $\sigma$ is a homeomorphism, and that
\begin{align*}
\sigma(0,\beta) &= (0,\beta), \\
\sigma(\pi,\beta) &= (\pi,\beta)
\end{align*}
for all $\beta$.  
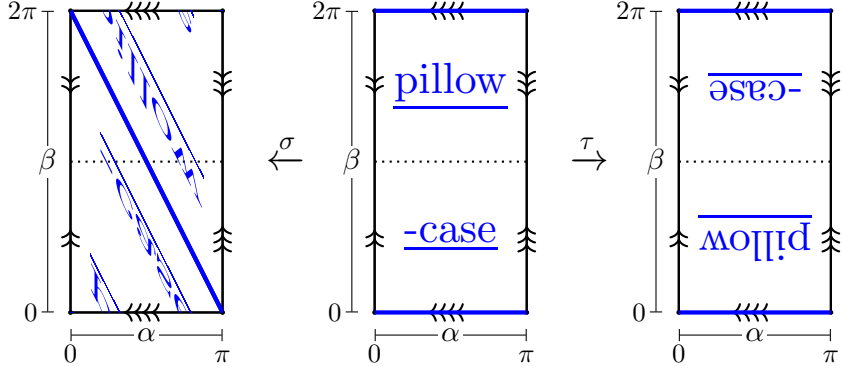
\begin{figure}
\begin{tikzpicture}[style=thick]
\foreach \i in {-4,0,4} {
\begin{scope}[xshift=\i cm]
  \draw plot[mark=*,mark size = 0.5pt] coordinates {(0,0)(2,0)(2,4)(0,4)} -- cycle;
  \begin{scope}[decoration={markings,mark=at position 0.55 with {\arrow[scale=1]{>>}}}]
    \draw[postaction={decorate}] (0,0) -- (0,2);
    \draw[postaction={decorate}] (0,4) -- (0,2);
  \end{scope}
  \begin{scope}[decoration={markings,mark=at position 0.575 with {\arrow[scale=1]{>>>}}}]
    \draw[postaction={decorate}] (2,0) -- (2,2);
    \draw[postaction={decorate}] (2,4) -- (2,2);
  \end{scope}
  \begin{scope}[decoration={markings,mark=at position 0.65 with {\arrow[scale=1]{>>>>}}}]
    \draw[postaction={decorate}] (2,4) -- (0,4);
    \draw[postaction={decorate}] (2,0) -- (0,0);
  \end{scope}
  \draw[dotted] (0,2) -- (2,2);
  \draw[thin,|-|] (0,-0.3) node[below] {\small$0$} -- node[midway,inner sep=1pt,fill=white] {$\alpha$} ++(2,0) node[below] {\small$\vphantom{0}\pi$};
  \draw[thin,|-|] (-0.3,0) node[left] {\small$0$} -- node[midway,inner sep=1pt,fill=white] {$\beta$} ++(0,4) node[left] {\small$2\pi$};
\end{scope}
}

\begin{scope}[every node/.style={blue,scale=1.5}]
  \node at (1,3) {\underline{pillow}};
  \node at (1,1) {\underline{-case}};
  \draw[blue,ultra thick] (0,0) -- (2,0) (0,4) -- (2,4);
\end{scope}

\node at (2.825,2.13) {\Large$\xrightarrow{\tau}$};
\begin{scope}[transform shape,rotate around={180:(3,2)}]
\begin{scope}[every node/.style={blue,scale=1.5}]
  \node at (1,3) {\underline{pillow}};
  \node at (1,1) {\underline{-case}};
  \draw[blue,ultra thick] (0,0) -- (2,0) (0,4) -- (2,4);
\end{scope}
\end{scope}

\node at (-1.175,2.13) {\Large$\xleftarrow{\sigma}$};
\begin{scope}[xshift=-4cm]
\clip (0,0) rectangle (2,4);
\foreach \i in {0,4,8} {
  \begin{scope}[yshift=\i cm,transform shape,yscale=-1,yslant=2]
    \begin{scope}[every node/.style={blue,scale=1.5}]
      \node at (1,3) {\underline{pillow}};
      \node at (1,1) {\underline{-case}};
      \draw[blue,ultra thick] (0,0) -- (2,0) (0,4) -- (2,4);
    \end{scope}
  \end{scope}
}
\end{scope}

\end{tikzpicture}
\caption{The involutions $\sigma$ and $\tau$ of the pillowcase $X(T^2)$.}
\label{fig:sigma-tau}
\end{figure}
Our goal in this subsection is to understand the image under $\sigma$ of the image $I_K = i^*(X(E_K))$ in the pillowcase, where $K \subset Y$ is a knot as in the statement of Theorem~\ref{thm:gluing-rep}.  Indeed, the representation promised by Theorem~\ref{thm:gluing-rep} will eventually come from finding a point in the pillowcase where one image $I_{K_1}$ intersects another skewed image $\sigma(I_{K_2})$.

\begin{lemma} \label{lem:pillowcase-involution}
Define an involution of the pillowcase $X(T^2)$ by
\[ \tau(\alpha,\beta) = (\pi-\alpha, 2\pi-\beta) \]
in either of the coordinates \eqref{eq:pillowcase-quotient} or \eqref{eq:pillowcase-domain}, as shown in Figure~\ref{fig:sigma-tau}.  If $K \subset Y$ is a nullhomologous knot in an arbitrary 3-manifold, then the pillowcase image
\[ I_K := i^*(X(E_K)) \subset X(T^2) \]
is invariant (as a set) under $\tau$, meaning that $\tau(I_K) = I_K$, and so is the image $\sigma(I_K)$.
\end{lemma}

\begin{proof}
We work with \eqref{eq:pillowcase-quotient} for convenience.  Take a point $(\alpha,\beta) \in I_K$, which means that there is some representation $\rho: \pi_1(E_K) \to \SU(2)$ such that
\begin{align*}
\rho(\mu) &= \begin{pmatrix} e^{i\alpha} & 0 \\ 0 & e^{-i\alpha} \end{pmatrix}, &
\rho(\lambda) &= \begin{pmatrix} e^{i\beta} & 0 \\ 0 & e^{-i\beta} \end{pmatrix}.
\end{align*}
We take a central character
\[ \chi: \pi_1(E_K) \twoheadrightarrow H_1(E_K) \cong H_1(Y) \oplus \Z \to \{\pm1\}, \]
defined by sending $H_1(Y)$ to $+1$ and the meridian $\mu$ (which generates the $\Z$ summand) to $-1$, and then we get a new representation
\[ \tilde\rho = \chi \cdot \rho: \pi_1(E_K) \to \SU(2) \]
such that $i^*([\tilde\rho]) = (\alpha+\pi,\beta)$.  In $X(T^2)$ we can identify
\[ (\alpha+\pi,\beta) \sim (-\alpha-\pi, -\beta) = (\pi-\alpha, 2\pi - \beta) = \tau(\alpha,\beta), \]
so $\tau(\alpha,\beta)$ also lies in the image $I_K$.  This proves that $I_K$ is $\tau$-invariant.

We now claim that $\sigma \circ \tau = \tau \circ \sigma$, which we can check directly by computing
\begin{align*}
\sigma(\tau(\alpha,\beta)) &= \sigma(\pi-\alpha,2\pi-\beta) \\
&= (\pi-\alpha, 2\pi - (2(\pi-\alpha) + (2\pi-\beta))) \\
&= (\pi-\alpha, -2\pi + 2\alpha + \beta) \\
&= \tau(\alpha, 2\pi-(2\alpha-\beta)) \\
&= \tau(\sigma(\alpha,\beta)).
\end{align*}
But then we apply this to the $\tau$-invariant set $I_K$ to get
\[ \sigma(I_K) = \sigma(\tau(I_K)) = \tau(\sigma(I_K)), \]
so $\sigma(I_K)$ is $\tau$-invariant as well.
\end{proof}

We can now use the involution $\tau$ to study the skewed image $\sigma(I_K)$ of the character variety of $K$.

\begin{lemma} \label{lem:skew-image}
Let $K \subset Y$ be a nullhomologous knot in a rational homology sphere with $Y_2(K) \cong \#^n \RP^3$ for some $n \geq 0$, and suppose that $K$ has irreducible complement and that $(Y,K) \not\cong (S^3,U)$.  Then exactly one of the following must be true:
\begin{enumerate}
\item The image $\sigma(I_K) \subset X(T^2)$ contains the entire line $L_\pi = \{\beta \equiv \pi \pmod{2\pi}\}$.
\item The image $\sigma(I_K) \subset X(T^2)$ avoids the points $P=(0,\pi)$ and $Q=(\pi,\pi)$, and contains a homologically essential simple closed curve
\[ \tilde{C} \subset X(T^2) \setminus \{P,Q\} \]
such that $\tilde{C}$ is disjoint from the line $L_0 = \{\beta \in 2\pi\Z \}$.
\end{enumerate}
Moreover, the intersection $\sigma(I_K) \cap L_\pi$ is connected.
\end{lemma}

\begin{proof}
Proposition~\ref{prop:avoid-slope-2} tells us that either $I_K$ contains the line $L'_\pi = \{2\alpha+\beta \equiv \pi\pmod{2\pi}\}$, or it avoids $P$ and $Q$ and contains a homologically essential simple closed curve
\[ C \subset I_K \subset X(T^2) \setminus \{P,Q\} \]
that is disjoint from the line $L'_0 = \{2\alpha+\beta \equiv 0 \pmod{2\pi}\}$.  In the first case, $\sigma(I_K)$ contains the line $\sigma(L'_\pi) = \{\beta\equiv\pi\pmod{2\pi}\} = L_\pi$.

In the second case, we note that $\sigma$ fixes both $P$ and $Q$, hence restricts to a homeomorphism
\[ X(T^2) \setminus \{P,Q\} \xrightarrow{\cong} X(T^2) \setminus \{P,Q\}. \]
But then $\sigma(C)$ avoids $P$ and $Q$, just as $C$ does, and it remains homologically essential in their complement.  We take $\tilde{C} = \sigma(C)$, and note that $\tilde{C}$ is disjoint from the line $\sigma(L'_0)$, which is precisely $\{\beta\in2\pi\Z\} = L_0$.

Finally, in either case Proposition~\ref{prop:lines-of-slope-2} tells us that the intersection $I_K \cap L'_\pi$ is connected, so the same is true of its image under $\sigma$, which is $\sigma(I_K) \cap \sigma(L'_\pi) = \sigma(I_K) \cap L_\pi$.
\end{proof}

\subsection{Intersections of pillowcase images}

We are now ready to prove the following proposition, which will imply Theorem~\ref{thm:gluing-rep}, as discussed at the beginning of the previous subsection.

\begin{proposition} \label{prop:skew-intersection}
Let $Y_1$ and $Y_2$ be rational homology spheres, and let $K_1 \subset Y_1$ and $K_2 \subset Y_2$ be nullhomologous knots with irreducible exteriors such that $(Y_\ell)_2(K_\ell) \cong \#^{n_\ell} \RP^3$ for $\ell=1,2$, where $n_1,n_2 \geq 1$.  Suppose that neither pair $(Y_\ell,K_\ell)$ is homeomorphic to $(S^3,U)$.  Then the subsets
\[ I_{K_1},\ \sigma(I_{K_2}) \subset X(T^2) \]
intersect at some point $(\alpha,\beta)$, where neither $2\alpha+\beta$ nor $\beta$ is an integer multiple of $2\pi$.
\end{proposition}

We split the proof of Proposition~\ref{prop:skew-intersection} into two cases, which occupy the following two lemmas.

\begin{lemma} \label{lem:intersect-if-P}
Proposition~\ref{prop:skew-intersection} holds if at least one of the pillowcase images $I_{K_1}$ and $I_{K_2}$ contains the point $P = (0,\pi)$.
\end{lemma}

\begin{proof}
We note that $\sigma(P) = P$, so if $P$ belongs to both $I_{K_1}$ and $I_{K_2}$ then it also belongs to $\sigma(I_{K_2})$ and hence to $I_{K_1} \cap \sigma(I_{K_2})$; in this case we have $2\alpha+\beta = \beta = \pi \not\in 2\pi\Z$, as desired.

Now suppose that $P \in I_{K_1}$, but that $P \not\in I_{K_2}$ and hence $P = \sigma(P) \not\in \sigma(I_{K_2})$.  Then Proposition~\ref{prop:avoid-slope-2} says that $I_{K_1}$ contains the entire line
\[ L'_\pi = \{2\alpha+\beta \equiv \pi\!\!\!\pmod{2\pi} \}, \]
whose endpoints are at $P=(0,\pi)$ and $Q=(\pi,\pi)$; and Lemma~\ref{lem:skew-image} says that $\sigma(I_{K_2})$ contains a homologically essential simple closed curve
\[ \tilde{C}_2 \subset X(T^2) \setminus \{P,Q\} \]
disjoint from the line $\{\beta\in 2\pi\Z\}$.  Since $X(T^2) \setminus \{P,Q\}$ is topologically a twice-punctured sphere, with first homology $\Z$, we can measure the homology class of $\tilde{C}_2$ by counting its intersections with any arc from $P$ to $Q$.  The line $L'_\pi$ is such an arc, and since $\tilde{C}_2$ is non-zero in homology we conclude that they must intersect.  We let $(\alpha,\beta)$ be any point of the intersection $L'_\pi \cap \tilde{C}_2$; then $(\alpha,\beta)$ belongs to $I_{K_1} \cap \sigma(I_{K_2})$ by definition, and as a point of $L'_\pi$ and of $\tilde{C}_2$ it satisfies $2\alpha+\beta \not\in 2\pi\Z$ and $\beta \not\in 2\pi\Z$ respectively, as desired.

The remaining case, where $P\not\in I_{K_1}$ and $P \in \sigma(I_{K_2})$, is nearly identical.  In this case $\sigma(I_{K_2})$ contains the entire line $L_\pi = \{\beta\equiv\pi\pmod{2\pi}\}$ from $P$ to $Q$ by Lemma~\ref{lem:skew-image}, while Proposition~\ref{prop:avoid-slope-2} gives us an essential curve $C_1 \subset X(T^2)\setminus\{P,Q\}$ in the image $I_{K_1}$, with $C_1$ disjoint from $\{2\alpha+\beta \in 2\pi\Z\}$.  Since $C_1$ is essential it must intersect the arc $L_\pi$ from $P$ to $Q$, and at any point $(\alpha,\beta)$ in the intersection we have $\beta\not\in 2\pi\Z$ since $(\alpha,\beta) \in L_\pi$, and $2\alpha+\beta \not\in 2\pi\Z$ since $(\alpha,\beta) \in C_1$.
\end{proof}

\begin{lemma} \label{lem:intersect-if-not-P}
Proposition~\ref{prop:skew-intersection} holds if neither $I_{K_1}$ nor $I_{K_2}$ contains the point $P=(0,\pi)$.
\end{lemma}

\begin{proof}
Letting $Q=(\pi,\pi)$, Proposition~\ref{prop:avoid-slope-2} tells us that there is a homologically essential, simple closed curve
\[ C_1 \subset I_{K_1} \subset X(T^2) \setminus \{P,Q\} \]
that is disjoint from the line $\{2\alpha+\beta \in 2\pi\Z\}$.  Similarly, by Lemma~\ref{lem:skew-image} there is an essential curve
\[ \tilde{C}_2 \subset \sigma(I_{K_2}) \subset X(T^2) \setminus \{P,Q\} \]
that is disjoint from the line $\{\beta \in 2\pi\Z\}$.  We will let $\tau: X(T^2) \to X(T^2)$ be the involution of Lemma~\ref{lem:pillowcase-involution}, which exchanges the points $P$ and $Q$ and fixes the images $I_{K_1}$ and $\sigma(I_{K_2})$ setwise.  (In particular, we note that $\tau(\tilde{C}_2) \subset \sigma(I_{K_2})$ as well.)

First, we observe that if the intersection
\[ \big( C_1 \cup \tau(C_1) \big) \cap \big( \tilde{C}_2 \cup \tau(\tilde{C}_2) \big) \]
is nonempty, then any point $(\alpha,\beta)$ in the intersection will suffice.  It must satisfy $2\alpha+\beta \not\in 2\pi\Z$ since it lies on either $C_1$ or $\tau(C_1)$, and then $\beta\not\in 2\pi\Z$ since it lies on either $\tilde{C}_2$ or $\tau(\tilde{C}_2)$.  Thus we may assume from now on that the sets
\[ \big( C_1 \cup \tau(C_1) \big) \quad\text{and}\quad \big( \tilde{C}_2 \cup \tau(\tilde{C}_2) \big) \]
are disjoint.

Since each of the simple closed curves $C_1$, $\tau(C_1)$, $\tilde{C}_2$, and $\tau(\tilde{C}_2)$ is homologically essential in $X(T^2) \setminus \{P,Q\}$, they must all separate $P = (0,\pi)$ from $Q = (\pi,\pi)$ and intersect the line segment
\[ L_\pi = [0,\pi] \times \{\pi\} \]
from $P$ to $Q$.  Let $\alpha_0 \in [0,\pi]$ be the minimal coordinate such that at least one of these four curves passes through $(\alpha_0,\pi)$, and let $\gamma$ be the curve in question.  We split the remainder of the proof into two cases.

\vspace{1em}
\noindent{\bf Case 1}: The curve $\gamma$ is either $\tilde{C}_2$ or $\tau(\tilde{C}_2)$.

\begin{figure}
\begin{tikzpicture}[style=thick]
\begin{scope}
  \draw plot[mark=*,mark size = 0.5pt] coordinates {(0,0)(3,0)(3,6)(0,6)} -- cycle; 
  \draw[thin,|-|] (0,-0.4) node[below] {\small$0$} -- node[midway,inner sep=1pt,fill=white] {$\alpha$} ++(3,0) node[below] {\small$\vphantom{0}\pi$};
  \draw[thin,|-|] (-0.75,0) node[left] {\small$0$} -- node[midway,inner sep=1pt,fill=white] {$\beta$} ++(0,6) node[left] {\small$2\pi$};
  \coordinate (P) at (0,3);
  \coordinate (Q) at (3,3);
  \path[save path=\pqline,name path=L] (P) -- (Q);

  \begin{scope}[color=blue,style=ultra thick]
     \foreach \i in {180,0} {
       \begin{scope}[rotate around={\i:(1.5,3)}]
         \draw[save path=\gpath,name path=gp] (0,5) to[out=0,in=75] (0.75,3.5) to[out=255,in=90] (1.4,1.5) to[out=270,in=0]  coordinate[pos=0.75] (gamma\i) (0,1);
         \path[use path=\gpath,name path=gp];
         \path [name intersections={of=gp and L, by=gint\i}];
       \end{scope}
     }
     \draw[very thick] (gint0) -- (gint180);
     \node[below,inner sep=3pt] at (gamma0) {$\gamma$};
     \node[above,inner sep=2pt] at (gamma180) {$\tau(\gamma)$};
  \end{scope}
  \draw[blue,ultra thick,fill=gray!20,use path=\gpath];
  \node at (0.6,1.75) {$D$};
  
  \draw[fill=black] (P) circle (0.075) node[left] {$P$};
  \draw[fill=black] (Q) circle (0.075) node[right] {$Q$};
  \draw[dotted,use path=\pqline];
  \begin{scope}[decoration={markings,mark=at position 0.55 with {\arrow[scale=1]{>>}}}]
    \draw[postaction={decorate}] (0,0) -- (0,3);
    \draw[postaction={decorate}] (0,6) -- (0,3);
  \end{scope}
  \begin{scope}[decoration={markings,mark=at position 0.575 with {\arrow[scale=1]{>>>}}}]
    \draw[postaction={decorate}] (3,0) -- (3,3);
    \draw[postaction={decorate}] (3,6) -- (3,3);
  \end{scope}
  \begin{scope}[decoration={markings,mark=at position 0.6 with {\arrow[scale=1]{>>>>}}}]
    \draw[postaction={decorate}] (3,6) -- (0,6);
    \draw[postaction={decorate}] (3,0) -- (0,0);
  \end{scope}

  \draw[red,very thick] (0,5.5) to[out=0,in=105,looseness=0.6] node[pos=0.2,above right,inner sep=1pt] {$C_1$} (1.35,3) to[out=285,in=0,looseness=2.25] (0,0.5);
\end{scope}
\end{tikzpicture}
\caption{The case of Lemma~\ref{lem:intersect-if-not-P} where $\gamma$ is either $\tilde{C}_2$ or $\tau(\tilde{C}_2)$.  The curve $C_1 \subset I_{K_1}$ separates $\tilde{C}_2$ from $\tau(\tilde{C}_2)$, so it must meet a segment of $L_\pi=[0,\pi]\times\{\pi\}$ connecting one to the other, and this segment lies in $\sigma(I_{K_2})$, so the intersection $I_{K_1} \cap \sigma(I_{K_2})$ is non-empty.}
\label{fig:making-images-intersect}
\end{figure}
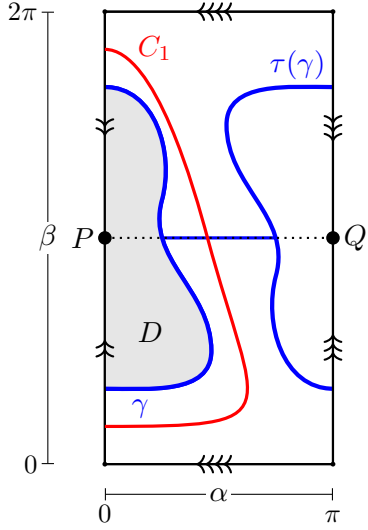
Letting $D$ be the disk component of $X(T^2) \setminus \gamma$ that contains $P$, we claim that in this case $C_1$ must be disjoint from $D$, as shown in Figure~\ref{fig:making-images-intersect}.  Assuming otherwise, it would be contained in the punctured disk $D \setminus \{P\}$ (since it is disjoint from $\gamma = \partial D$), and it is disjoint from the arc $[0,\alpha_0]\times \{\pi\}$ from $P$ to $\partial D$ by assumption, so it would necessarily be nullhomotopic in $D \setminus \{P\}$.  This would mean that $C_1$ bounds a disk $D' \subset D$ that does not contain $P$, and then $D'$ cannot contain $Q$ either since $Q \not\in D$, so $C_1 = \partial D'$ would not separate $P$ from $Q$ in $X(T^2)$, a contradiction.  Thus we know that the disk $D$ bounded by $\gamma$ is disjoint from $C_1$, and by an identical argument it is also disjoint from $\tau(C_1)$.

Applying the involution $\tau$, we know that the disk $\tau(D)$ is bounded by $\tau(\gamma)$ and contains $\tau(P) = (\pi,\pi) = Q$.  This disk must be disjoint from both $C_1$ and $\tau(C_1)$, since otherwise we could apply $\tau$ again to see that either $\tau(C_1)$ or $\tau^2(C_1) = C_1$ meets $\tau^2(D) = D$, which we know to be impossible.  So now the simple closed curve $C_1$ is disjoint from both $D$ and $\tau(D)$, whose boundaries are $\tilde{C}_2$ and $\tau(\tilde{C}_2)$ in some order, and it separates $P \in D$ from $Q \in \tau(D)$.  It follows that $D$ and $\tau(D)$ lie in different components of the complement $X(T^2) \setminus C_1$, and in particular so do their boundaries $\tilde{C}_2$ and $\tau(\tilde{C}_2)$.

Now the curves $\tilde{C}_2$ and $\tau(\tilde{C}_2)$ both intersect the line segment $L_\pi$: by assumption one of them does at $x = (\alpha_0,\pi)$, and then the other one must meet $L_\pi$ at $\tau(x) = (\pi-\alpha_0,\pi)$.  Then $x$ and $\tau(x)$ lie in different components of $X(T^2) \setminus C_1$.  They also belong to the intersection $\sigma(I_{K_2}) \cap L_\pi$, which is connected by Lemma~\ref{lem:skew-image}, so there must be some point
\[ (\alpha,\beta) \in \sigma(I_{K_2}) \cap L_\pi \]
that also lies in $C_1$.  Then $(\alpha,\beta)$ belongs to $I_{K_1} \cap \sigma(I_{K_2})$, we have $\beta = \pi \not\in 2\pi\Z$ since $(\alpha,\beta) \in L_\pi$, and similarly $2\alpha+\beta \not\in 2\pi\Z$ since $(\alpha,\beta) \in C_1$.  Thus $(\alpha,\beta)$ is our desired point of intersection, and this proves the lemma in the case where $\gamma$ is either $\tilde{C}_2$ or $\tau(\tilde{C}_2)$.

\vspace{1em}
\noindent{\bf Case 2}: The curve $\gamma$ is either $C_1$ or $\tau(C_1)$.

In this case, an identical argument shows that $C_1$ and $\tau(C_1)$ lie in different components of the complement $X(T^2) \setminus \tilde{C}_2$.  Now the line segment
\begin{align*}
L'_\pi &= \{ (\alpha, \pi-2\alpha) \mid 0 \leq \alpha \leq \pi \} \\
&= \{ 2\alpha+\beta \equiv \pi \!\!\!\pmod{2\pi} \}
\end{align*}
in $X(T^2)$ has its endpoints at $P$ and $Q$, so it must intersect any homologically essential curve in $X(T^2) \setminus \{P,Q\}$.  This includes the curve $C_1$, so we take a point
\[ x \in C_1 \cap L'_\pi \]
and note that $\tau(x) \in \tau(C_1) \cap L'_\pi$ as well, since $\tau$ fixes $L'_\pi$ setwise.

Now Proposition~\ref{prop:lines-of-slope-2} says that the intersection $I_{K_1} \cap L'_\pi$ is connected.  This intersection contains both $x$ and $\tau(x)$, which lie in different components of $X(T^2) \setminus \tilde{C}_2$ since they belong to $C_1$ and $\tau(C_1)$ respectively.  It follows that $I_{K_1} \cap L'_\pi$ must intersect $\tilde{C}_2 \subset \sigma(I_{K_2})$ at some point
\[ (\alpha,\beta) \in I_{K_1} \cap L'_\pi. \]
We have $2\alpha+\beta \not\in 2\pi\Z$ since $(\alpha,\beta) \in L'_\pi$, and $\beta \not\in 2\pi\Z$ since $(\alpha,\beta) \in \tilde{C}_2$, so $(\alpha,\beta)$ is the desired point and we are done.
\end{proof}

\begin{proof}[Proof of Proposition~\ref{prop:skew-intersection}]
One of the following must apply: either $P = (0,\pi)$ lies in at least one of $I_{K_1}$ and $I_{K_2}$, or it lies in neither of them.  In the first case we apply Lemma~\ref{lem:intersect-if-P}, and in the second case we apply Lemma~\ref{lem:intersect-if-not-P}.
\end{proof}

\subsection{Constructing a representation}

We are finally ready to prove Theorem~\ref{thm:gluing-rep}, using the information provided by Proposition~\ref{prop:skew-intersection}.  We recall the hypotheses of Theorem~\ref{thm:gluing-rep} here for convenience: we have nullhomologous knots $K_1 \subset Y_1$ and $K_2 \subset Y_2$, whose exteriors are irreducible and not solid tori, and these satisfy
\begin{align*}
(Y_1)_2(K_1) &\cong \#^k \RP^3, &
(Y_2)_2(K_2) &\cong \#^\ell \RP^3.
\end{align*}
The closed manifold $Y = E_{K_1} \cup_\partial E_{K_2}$ is then formed from their exteriors by gluing $\mu_1$ and $\lambda_1$ to $\mu_2^{-1}$ and $\mu_2^2\lambda_2$, respectively.

\begin{proof}[Proof of Theorem~\ref{thm:gluing-rep}]
Since the exteriors $E_{K_1}$ and $E_{K_2}$ are not solid tori, Proposition~\ref{prop:skew-intersection} tells us that there is a point
\[ (\alpha,\beta) \in i^*(X(E_{K_1})) \cap \sigma\big( i^*(X(E_{K_2})) \big) \]
in the pillowcase $X(T^2)$ such that $\beta \not\in 2\pi\Z$ and $2\alpha+\beta \not\in 2\pi\Z$.  Since $(\alpha,\beta)$ lies in $i^*(X(E_{K_1}))$, this means that there is a representation
\[ \rho_1: \pi_1(E_{K_1}) \to \SU(2) \]
such that
\begin{align*}
\rho_1(\mu_1) &= \begin{pmatrix} e^{i\alpha} & 0 \\ 0 & e^{-i\alpha} \end{pmatrix}, &
\rho_1(\lambda_1) &= \begin{pmatrix} e^{i\beta} & 0 \\ 0 & e^{-i\beta} \end{pmatrix},
\end{align*}
and $\rho_1$ has non-abelian image, since $\beta \not\in 2\pi\Z$ implies that $\rho_1(\lambda_1) \neq 1$.

Similarly, since $\sigma$ is an involution of the pillowcase, there is a unique point $(\gamma,\delta) \in X(T^2)$ such that
\[ \sigma(\gamma,\delta) = (\alpha,\beta). \]
Then since $(\alpha,\beta)$ lies in $\sigma\big( i^*(X(E_{K_2})) \big)$ we have $(\gamma,\delta) \in i^*(X(E_{K_2}))$, so there is a representation
\[ \rho_2: \pi_1(E_{K_2}) \to \SU(2) \]
such that
\begin{align*}
\rho_2(\mu_2) &= \begin{pmatrix} e^{i\gamma} & 0 \\ 0 & e^{-i\gamma} \end{pmatrix}, &
\rho_2(\lambda_2) &= \begin{pmatrix} e^{i\delta} & 0 \\ 0 & e^{-i\delta} \end{pmatrix}.
\end{align*}
We have $(\gamma,\delta) = \sigma(\alpha,\beta) = (-\alpha, 2\alpha+\beta)$ as well, so the condition $2\alpha+\beta \not\in 2\pi\Z$ is equivalent to $\delta \not\in 2\pi\Z$.  This means that $\rho_2(\lambda_2) \neq 1$, so $\rho_2$ also has non-abelian image.  

Now from $(\alpha,\beta) = \sigma(\gamma,\delta) = (-\gamma, 2\gamma+\delta)$ we conclude that, up to replacing $\rho_2$ with a conjugate to replace the coordinates $(\gamma,\delta)$ with the equivalent $(-\gamma,-\delta)$, we have
\begin{align*}
\rho_1(\mu_1) &= \rho_2(\mu_2^{-1}), \\
\rho_1(\lambda_1) &= \rho_2(\mu_2^2\lambda_2).
\end{align*}
This says that when we form the closed 3-manifold
\[ Y = E_{K_1} \cup_\partial E_{K_2} \]
by gluing $\mu_1$ to $\mu_2^{-1}$ and $\lambda_1$ to $\mu_2^2\lambda_2$, the representations $\rho_1$ and $\rho_2$ agree on the common torus $\partial E_{K_1} \sim \partial E_{K_2}$, and so we can glue them together to define
\[ \rho: \pi_1(Y) \to \SU(2) \]
whose restrictions to $E_{K_1}$ and $E_{K_2}$ are the non-abelian representations $\rho_1$ and $\rho_2$, as desired.
\end{proof}

\section{Toroidal manifolds as unions of knot complements} \label{sec:knot-exteriors}

In this section we prepare to prove Theorem~\ref{thm:sl2-abelian-main} by studying toroidal 3-manifolds of the form
\[ Y = M_1 \cup_T M_2, \]
where $H_1(Y;\Z)$ is $p$-torsion for some prime $p$ and each $M_i$ is a compact orientable 3-manifold with boundary $T$.  Our goal is to express the $M_i$ as complements of knots in rational homology spheres, which can be glued together in a standard way so that when $p=2$, we will be able to find representations of each $\pi_1(M_i)$ by using the results of the previous sections.

\subsection{Nullhomologous rational longitudes} \label{ssec:basis-inessential}

In this subsection, we suppose that $Y = M_1 \cup_T M_2$ has a separating incompressible torus $T$, and that the rational longitudes of $M_1$ and $M_2$ are both nullhomologous.  Under these assumptions, we will express $M_1$ and $M_2$ as the complements of nullhomologous knots in closed 3-manifolds, with a short list of standard forms for the gluing map $\partial M_1 \cong \partial M_2$.

\begin{proposition} \label{prop:nice-basis}
Let $Y = M_1 \cup_T M_2$ be a closed 3-manifold with $H_1(Y;\Z) \cong (\Z/p\Z)^r$, for some prime $p$ and integer $r \geq 0$, and with separating incompressible torus $T$.  Suppose that the rational longitudes of $M_1$ and $M_2$ are both nullhomologous.  Then there are closed 3-manifolds $Y_1$ and $Y_2$, with
\begin{align*}
H_1(Y_1;\Z) &\cong (\Z/p\Z)^k, \\
H_1(Y_2;\Z) &\cong (\Z/p\Z)^\ell
\end{align*}
for some integers $k$ and $\ell$, and nullhomologous knots $K_1 \subset Y_1$ and $K_2 \subset Y_2$ with exteriors 
\[ M_i \cong Y_i \setminus N(K_i) \qquad (i=1,2), \]
such that one of the following holds.
\begin{enumerate}
\item $k+\ell = r$, and the identification $\partial M_1 \cong \partial M_2$ sends $(\mu_1,\lambda_1)$ to $(\lambda_2,\mu_2)$.  In this case there are degree-1 maps $Y \to Y_i$ for each $i=1,2$. \label{i:y2-hrp3}

\item $k+\ell = r-1$, and the identification $\partial M_1 \cong T \cong \partial M_2$ equates
\begin{align*}
\mu_1 &= a\mu_2 + b\lambda_2, &
\lambda_1 &= p\mu_2 + c\lambda_2
\end{align*}
in $H_1(T)$, for some integers $a,b,c$ with $ac-bp=-1$ and $0 \leq b < c < p$.  Then there are degree-1 maps
\[ Y \to (Y_1)_{-p/a}(K_1) \qquad\text{and}\qquad Y \to (Y_2)_{p/c}(K_2). \]
In particular, when $p=2$ we must have $(a,b,c)=(-1,0,1)$, so that $(\mu_1,\lambda_1)$ is sent to $(-\mu_2,2\mu_2+\lambda_2)$ and there are degree-1 maps $Y \to (Y_i)_2(K_i)$ for $i=1,2$. \label{i:y2-hs3}
\end{enumerate}
Each of these degree-1 maps induces a surjection on $\pi_1$ with non-trivial kernel.
\end{proposition}

To prove Proposition~\ref{prop:nice-basis}, we begin by finding candidate pairs $(Y_i,K_i)$ somewhat arbitrarily, and then in the subsequent lemma we will use Dehn surgery to make better choices as needed.  

\begin{lemma} \label{lem:m_i-as-exteriors}
Let $Y$ be a closed 3-manifold satisfying $H_1(Y;\Z) \cong (\Z/p\Z)^r$ for some prime $p$ and integer $r \geq 0$, and suppose that $Y$ has an incompressible torus $T$ separating it into $Y = M_1 \cup_T M_2$.  Suppose in addition that the rational longitudes of $M_1$ and $M_2$ are both nullhomologous.  Then there is a pair of closed 3-manifolds $Y_1$ and $Y_2$ with
\begin{align*}
H_1(Y_1;\Z) &\cong (\Z/p\Z)^k, &
H_1(Y_2;\Z) &\cong (\Z/p\Z)^\ell,
\end{align*}
together with nullhomologous knots $K_1 \subset Y_1$ and $K_2 \subset Y_2$ whose exteriors are
\begin{align*}
M_1 &\cong Y_1 \setminus N(K_1), &
M_2 &\cong Y_2 \setminus N(K_2)
\end{align*}
respectively, such that if $(\mu_i,\lambda_i)$ are meridian-longitude pairs for each $K_i$ then either
\begin{enumerate}
\item $k+\ell = r$, and the gluing $\partial M_1 \cong \partial M_2$ identifies $\mu_1$ with $\lambda_2$ and $\lambda_1$ with $\mu_2$;
\item or $k+\ell=r-1$, and in $\partial M_1 \cong \partial M_2$ we have $\lambda_1 = p\mu_2+c\lambda_2$ for some $c\in\Z$ that is not a multiple of $p$.
\end{enumerate}
\end{lemma}

\begin{proof}
Let $j_i: T \hookrightarrow M_i$ denote inclusion, and let $\lambda_1$ and $\lambda_2$ be the rational longitudes of $M_1$ and $M_2$, so that the classes $(j_1)_*(\lambda_1)$ and $(j_2)_*(\lambda_2)$ are both zero by assumption.

We let $M_1(\lambda_2)$ denote the Dehn filling of $M_1$ along the slope $\lambda_2$.  Since $\lambda_2$ is nullhomologous in $M_2$, Proposition~\ref{prop:solid-torus-pinch} provides a degree-1 pinching map $Y \to M_1(\lambda_2)$ that collapses $M_2$ to a solid torus.  Degree-1 maps are surjective on fundamental groups and hence on first homology, so there is a surjection
\[ H_1(Y;\Z) \twoheadrightarrow H_1(M_1(\lambda_2);\Z) \]
and hence $H_1(M_1(\lambda_2);\Z)$ is a quotient of $H_1(Y;\Z) \cong (\Z/p\Z)^r$.  This means that $H_1(M_1(\lambda_2);\Z) \cong (\Z/p\Z)^s$ for some $s \leq r$.

We now pick a curve $\mu_1 \subset T$ such that $\{\lambda_1, \mu_1\}$ is an integral basis of $H_1(T)$; if $\{\lambda_1,\lambda_2\}$ is an integral basis then we will insist that $\mu_1 = \lambda_2$.  We then let
\[ Y_1 = M_1(\mu_1), \]
and we take $K_1 \subset Y_1$ to be the core of this Dehn filling.  Then $K_1$ is nullhomologous, with meridian $\mu_1$ and longitude $\lambda_1$.  The manifold $M_1(\lambda_2)$ can be built by Dehn surgery on $K_1$, say with some slope $\frac{a}{b}$, and then
\[ H_1(Y_1;\Z) \oplus (\Z/a\Z) \cong H_1(M_1(\lambda_2);\Z) \cong (\Z/p\Z)^s \]
implies that $|a|$ divides $p$.  The homology $H_1(Y_1;\Z)$ must then have the form $(\Z/p\Z)^k$, where $k$ is either $s$ or $s-1$ depending on whether $|a|$ is $1$ or $p$, respectively.

By the same argument we can pick a Dehn filling $Y_2 = M_2(\mu_2)$, with nullhomologous core $K_2$, so that $H_1(Y_2;\Z) \cong (\Z/p\Z)^\ell$.  Again, if $\{\lambda_1,\lambda_2\}$ is an integral basis of $H_1(T)$ then we will take $\mu_2 = \lambda_1$; in this case we have $\mu_1 = \lambda_2$ as well, which is consistent with the fact that the identification $\partial M_1 \cong \partial M_2$ is orientation-reversing.

We now compute that $H_2(Y;\Z) = 0$, so the Mayer-Vietoris sequence for $Y = M_1 \cup_T M_2$ consists in part of a short exact sequence
\[ 0 \to \underbrace{H_1(T)}_{\cong \Z^2} \xrightarrow{((j_1)_*,(j_2)_*)} \underbrace{H_1(M_1)}_{\cong \Z \oplus (\Z/p)^k} \oplus \underbrace{H_1(M_2)}_{\cong \Z\oplus(\Z/p)^\ell} \to \underbrace{H_1(Y)}_{\cong (\Z/p)^r} \to 0. \]
The image of $(j_1)_*$ is spanned by $(j_1)_*(\mu_1) = (1,0) \in \Z \oplus (\Z/p)^k$ and by $(j_1)_*(\lambda_1) = (0,0)$, so it is precisely the $\Z$ summand of $H_1(M_1)$, and likewise for the image of $(j_2)_*$.  Thus $H_1(Y) \cong (\Z/p\Z)^r$ is homeomorphic to $(\Z/p\Z)^{k+\ell}$ plus the cokernel of the injective map
\[ H_1(T) \xrightarrow{j} \Z\langle\mu_1\rangle \oplus \Z\langle\mu_2\rangle, \]
whose codomain is viewed as a subset of $H_1(M_1) \oplus H_1(M_2)$.

We observe that $\coker(j)$ is cyclic, since the image of $j$ contains $j(\mu_1) = (1,a)$ for some $a\in\Z$.  We also know that $\coker(j)$ is $p$-torsion, as a summand of $H_1(Y) \cong (\Z/p\Z)^r$.  Thus either
\begin{itemize}
\item $\coker(j) \cong \Z/p\Z$ and $k+\ell = r-1$, or
\item $j$ is onto and $k+\ell = r$.
\end{itemize}
If $j$ is onto then it sends some element $a\mu_1+b\lambda_1$ to $(0,1)$; we must have $a=0$, since $(j_1)_*(\lambda_1)=0$ in $H_1(M_1)$, and then $b\lambda_1 \mapsto (0,1)$ implies that $b=\pm1$.  This means that $\lambda_1$ is homologous to $\pm \mu_2$ as elements of $H_1(M_2)$, hence $\lambda_1 = \pm(\mu_2 + n\lambda_2)$ in $H_1(T)$.  But then in $H_1(T)$ we have
\begin{align*}
\Span\{\lambda_1,\lambda_2\} &= \Span\{\mu_2+n\lambda_2,\lambda_2\} \\
&= \Span\{\mu_2,\lambda_2\} \\
&= H_1(T),
\end{align*}
so we must have taken $\mu_1 = \lambda_2$ and $\lambda_1 = \mu_2$, as claimed.

In the remaining case, where $\coker(j) \cong \Z/p\Z$ and $k+\ell = r-1$, we know that
\begin{align*}
j(\mu_1) &= (1,a), \\
j(\lambda_1) &= (0,b)
\end{align*}
for some $a,b \in \Z$, and since $\{\mu_1,\lambda_1\}$ is an integral basis of $H_1(T)$ we have $|b| = |{\coker(j)}| = p$.  Up to changing the orientation of $K_2$, and hence the sign of $\mu_2$, we can take $b = p$.  This means that as classes in $H_1(T)$ we have
\[ \lambda_1 = p\mu_2 + c\lambda_2 \]
for some $c \in \Z$, and since $\lambda_1$ is primitive we must have $p \nmid c$.
\end{proof}

In what follows we will always let $\mu_1,\lambda_1$ and $\mu_2,\lambda_2$ be meridian--longitude pairs for $K_1$ and $K_2$ respectively.  We will suppose that they are identified as classes in $H_1(T)$ by the relations
\begin{align} \label{eq:ml2-from-ml1}
\mu_2 &= a\mu_1 + b\lambda_1, &
\lambda_2 &= c\mu_1 + d\lambda_1.
\end{align}
Since both $(\mu_1,\lambda_1)$ and $(\mu_2,\lambda_2)$ are integral bases of $H_1(T) \cong \Z^2$, and the gluing map $\partial M_1 \to \partial M_2$ reverses orientation, we must have $ad-bc=-1$, and then we can rewrite these relations as
\begin{align} \label{eq:ml1-from-ml2}
\mu_1 &= -d\mu_2 + b\lambda_2, &
\lambda_1 &= c\mu_2 - a\lambda_2.
\end{align}

We now study the case of Lemma~\ref{lem:m_i-as-exteriors} where $k+\ell = r-1$.

\begin{lemma} \label{lem:Y_i-homology-spheres}
Suppose that we have $Y = M_1 \cup_T M_2$ as in Lemma~\ref{lem:m_i-as-exteriors}, with $H_1(Y;\Z) \cong (\Z/p\Z)^r$, and that the resulting $K_1 \subset Y_1$ and $K_2 \subset Y_2$ satisfy
\begin{align*}
H_1(Y_1;\Z) &\cong (\Z/p\Z)^k, \\
H_1(Y_2;\Z) &\cong (\Z/p\Z)^\ell
\end{align*}
where $k+\ell=r-1$.  Then we can arrange the gluing map to identify
\begin{equation} \label{eq:gluing-skew-coefficients}
\begin{aligned}
\mu_1 &= a\mu_2 + b\lambda_2 \\
\lambda_1 &= p\mu_2 + c\lambda_2
\end{aligned}
\end{equation}
as elements of $H_1(T;\Z)$, where the coefficients satisfy $ac-bp=-1$ and $0 \leq b < c < p$.  In particular, when $p=2$ we have $(\mu_1,\lambda_1) = (-\mu_2,2\mu_2+\lambda_2)$, or equivalently $(\mu_2,\lambda_2) = (-\mu_1,2\mu_1+\lambda_1)$.
\end{lemma}

\begin{proof}
We already know from Lemma~\ref{lem:m_i-as-exteriors} that $\lambda_1 = p\mu_2 + c\lambda_2$ for some integer $c\not\equiv 0\pmod{p}$.  We write $c=qp+r$, with $0<r<p$, and then let $Y'_2$ be the result of $(1/q)$-surgery on $K_2 \subset Y_2$, with $K'_2 \subset Y'_2$ the core of this surgery.  Then $Y'_2$ has the same homology as $Y_2$, and $K'_2$ is still nullhomologous, with longitude $\lambda'_2 = \lambda_2$ and meridian
\[ \mu'_2 = \mu_2 + q\lambda_2. \]
In terms of the peripheral curves for $K'_2$, we have
\begin{align*}
\lambda_1 &= p\mu_2 + (pq+r)\lambda_2 \\
&= p(\mu_2+q\lambda_2) + r\lambda_2 \\
&= p\mu'_2 + r\lambda'_2.
\end{align*}
We thus replace $(Y_2,K_2)$ with $(Y'_2,K'_2)$, so now we have $\lambda_1 = p\mu_2+c'\lambda_2$ where $c'=r$ is strictly between $0$ and $p$.

Having arranged that $0<c<p$ as above, we have an identification
\begin{align*}
\mu_1 &= a\mu_2 + b\lambda_2 \\
\lambda_1 &= p\mu_2 + c\lambda_2
\end{align*}
for some integers $a$ and $b$.  Since the gluing map $\partial M_1 \cong \partial M_2$ is an orientation-reversing homeomorphism, we have $ac-bp = -1$.  We now write $b = nc + s$, where $n\in\Z$ and $0 \leq s < c$, and let $Y'_1$ be the result of $(-\frac{1}{n})$-surgery on $K_1 \subset Y_1$, with core $K'_1$.  Then $K'_1$ has meridian $\mu'_1 = \mu_1 - n\lambda_1$ and longitude $\lambda'_1$, so that we identify $\lambda'_1 = \lambda_1 = p\mu_2 + c\mu_2$ and
\begin{align*}
\mu'_1 &= (a\mu_2 + b\lambda_2) - n(p\mu_2 + c\lambda_2) \\
&= (a-np)\mu_2 + (b-nc)\lambda_2 \\
&= a'\mu_2 + b'\lambda_2
\end{align*}
with $a' = a-np$ and $b' = b-nc = s$.  We can thus arrange that the coefficient $b'=s$ satisfies $0 \leq b' < c$, as desired.

We now have arranged for the coefficients of \eqref{eq:gluing-skew-coefficients} to satisfy $0 < c < p$ and $0 \leq b < c$, as desired.  The relation $ac-bp = -1$ is an immediate consequence of the fact that the identification $\partial M_1 \cong \partial M_2$ is an orientation-reversing homeomorphism.  And finally, in the case $p=2$ these relations imply one after the other that $c=1$, $b=0$, and $a=-1$, as claimed.
\end{proof}

Putting the above lemmas together now allows us to prove Proposition~\ref{prop:nice-basis}.

\begin{proof}[Proof of Proposition~\ref{prop:nice-basis}]
Lemma~\ref{lem:m_i-as-exteriors} combines with Lemma~\ref{lem:Y_i-homology-spheres} (in the case $k+\ell=r-1$) to show that we can find $K_1 \subset Y_1$ and $K_2 \subset Y_2$ as claimed.

Since $\lambda_2$ is nullhomologous in $M_2 = Y_2 \setminus N(K_2)$, there is a degree-1 map
\[ h: Y \to M_1(\lambda_2), \]
built from Proposition~\ref{prop:solid-torus-pinch} by preserving $M_1$ but pinching $M_2$ to a solid torus in which $\mu_1 \subset T$ bounds a disk.  The same argument, with the roles of $M_1$ and $M_2$ switched, gives a degree-1 map $Y \to M_2(\lambda_1)$.  Now in the case $k+\ell=r$ we have
\begin{align*}
M_1(\lambda_2) &= M_1(\mu_1) \cong Y_1, \\
M_2(\lambda_1) &= M_2(\mu_2) \cong Y_2,
\end{align*}
so $Y$ admits degree-1 maps onto both $Y_1$ and $Y_2$.  Otherwise, we have $k+\ell=r-1$, with identifications
\[ (\mu_1,\lambda_1) = (a\mu_2+b\lambda_2, p\mu_2 + c\lambda_2) \]
in $H_1(T)$ where $ac-bp=-1$, implying that
\[ p\mu_1 - a\lambda_1 = (pb-ac)\lambda_2 = \lambda_2. \]
Then we can fill each of $M_1$ and $M_2$ along the rational longitudes $\lambda_2$ and $\lambda_1$ to get
\begin{align*}
M_1(\lambda_2) &= M_1(p\mu_1-a\lambda_1) \cong (Y_1)_{-p/a}(K_1), \\
M_2(\lambda_1) &= M_2(p\mu_2+c\lambda_2) \cong (Y_2)_{p/c}(K_2),
\end{align*}
so there are degree-1 maps from $Y$ onto each of these.  In both cases, each of the maps from $Y$ induce surjections on $\pi_1$, because they have degree $1$; to see that these surjections have non-trivial kernel, we note that $\pi_1(T)$ injects into $\pi_1(Y)$, and yet either $\lambda_2$ or $\lambda_1$ (whichever one is filled to produce the respective maps) is a homotopically essential curve in $T$, and hence in $Y$, that bounds a disk in the quotient.
\end{proof}

\begin{remark} \label{rem:c-at-most-p/2}
In Lemma~\ref{lem:Y_i-homology-spheres}, and hence in Proposition~\ref{prop:nice-basis}, we can replace the condition $0 \leq b < c < p$ with $0 \leq b < c \leq \frac{p}{2}$ if we are allowed to possibly reverse the orientation of $Y$.  Indeed, in the proof of the lemma we wrote $c=qp+r$ and replaced $c$ with $r$ by performing $\frac{1}{q}$-surgery on $K_2$.  We chose the remainder $r$ to satisfy $0 < r < p$, but suppose that we arrange for $-\frac{p}{2} < r \leq \frac{p}{2}$ instead.  In this case, if $c'=r$ is negative then we can replace $Y$ with $-Y$ before continuing: we reverse the orientations of $Y_1$ and $Y_2$, and also reverse the string orientation of $K_1$ but not that of $K_2$.  This reverses the peripheral curves $\lambda_1$ and $\mu_2$, but not $\mu_1$ or $\lambda_2$, and so the relation
\[ \lambda_1 = p\mu_2 + c' \lambda_2, \qquad -\tfrac{p}{2} < c' < 0 \]
becomes $\lambda_1 = p\mu_2 + (-c') \lambda_2$, where the coefficient $-c'$ is now positive and less than $\frac{p}{2}$.  We then replace $c$ with $-c'$ and follow the rest of the proof as written to achieve $0 \leq b < c \leq \frac{p}{2}$.
\end{remark}

\subsection{The homologically essential case} \label{ssec:basis-essential}

In this subsection, we consider decompositions of the form $Y=M_1\cup_T M_2$ where at least one of the rational longitudes of the $M_i$ is homologically essential.  Our goal is to prove the following.

\begin{proposition} \label{prop:nice-essential-basis}
Let $Y = M_1 \cup_T M_2$ be a toroidal manifold, where $T$ is an incompressible torus, and suppose that $H_1(Y;\Z) \cong (\Z/p\Z)^r$ for some prime $p$ and integer $r \geq 0$.  Let $\lambda_1,\lambda_2 \subset T$ be the rational longitudes of $M_1$ and $M_2$, respectively, and suppose that $\lambda_1$ is not nullhomologous in $M_1$.  Then there are closed 3-manifolds $Y_1$ and $Y_2$, with
\begin{align*}
H_1(Y_1;\Z) &\cong (\Z/p\Z)^k, &
H_1(Y_2;\Z) &\cong (\Z/p\Z)^\ell
\end{align*}
where $k \geq 2$ and $k+\ell=r$, and knots $K_1 \subset Y_1$ and $K_2 \subset Y_2$ satisfying the following.
\begin{enumerate}
\item The knot $K_1$ is homologically essential, of order $p$, while $K_2$ is nullhomologous.
\item $M_1$ and $M_2$ are the exteriors of $K_1$ and $K_2$, i.e., $M_i \cong Y_i \setminus N(K_i)$.
\item The identification $\partial M_1 \cong T \cong \partial M_2$ sends $\mu_1$ to $\lambda_2$ and $\lambda_2$ to $\mu_1$.
\item There is a degree-1 map $Y \to Y_1$, inducing a surjection $\pi_1(Y) \to \pi_1(Y_1)$ with non-trivial kernel.
\end{enumerate}
\end{proposition}

We begin with some lemmas allowing us to find nice bases for $H_1(M_1)$ and $H_1(M_2)$, and to express the rational longitudes for each in these bases.

\begin{lemma} \label{lem:rational-lambda-basis}
Let $Y = M_1 \cup_T M_2$ be a toroidal manifold, with $H_1(Y;\Z) \cong (\Z/p\Z)^r$ for some prime $p$ and integer $r \geq 0$.  Let $\lambda_1, \lambda_2 \subset T$ be the rational longitudes of $M_1$ and $M_2$, respectively.  If some $\lambda_j$ is not nullhomologous in $M_j$, then the following are true.
\begin{itemize}
\item Exactly one of the $\lambda_j$ is not nullhomologous, and it satisfies $p\cdot (i_j)_*(\lambda_j) = 0$ in $H_1(M_j;\Z)$.
\item The curves $\lambda_1$ and $\lambda_2$ form an integral basis of $H_1(T;\Z) \cong \Z^2$.
\end{itemize}
Here the maps $i_j: T \hookrightarrow M_j$ are the respective inclusions of $\partial M_j$ into $M_j$.
\end{lemma}

\begin{proof}
We examine the Mayer--Vietoris sequence for $Y = M_1 \cup_T M_2$, which reads in part
\begin{equation} \label{eq:mv-rational-lambda}
0 \to \underbrace{H_1(T)}_{\cong\Z^2} \xrightarrow{i=((i_1)_*,(i_2)_*)} H_1(M_1) \oplus H_1(M_2) \xrightarrow{q} \underbrace{H_1(Y)}_{\cong(\Z/p)^r} \to 0
\end{equation}
since $H_2(Y;\Z) \cong 0$.  Each of the $H_1(M_j)$ has rank at least $1$ by half lives half dies, and this sequence shows that their total rank is $2$, so we can write
\[ H_1(M_j;\Z) \cong \Z \oplus A_j \qquad(j=1,2) \]
where each $A_j$ is torsion.  Moreover, the image of the map $i$ is torsion-free, so $q$ sends $A_1\oplus A_2$ injectively into $(\Z/p\Z)^r$, hence we can write $A_j = (\Z/p\Z)^{n_j}$ for each $j$, and we have $n_1+n_2 \leq r$.

Suppose without loss of generality that $(i_1)_*(\lambda_1)$ is non-zero.  Since it is torsion it lies in $A_1$, hence $p\cdot (i_1)_*(\lambda_1) = 0$ as claimed.  In fact, it generates a $\Z/p\Z$ summand of the $(\Z/p\Z$)-vector space $A_1 = (\Z/p\Z)^{n_1}$, so $n_1 \geq 1$ and we can write
\[ H_1(M_1;\Z) \cong (\Z/p\Z) \oplus (\Z/p\Z)^{n_1-1} \oplus \Z \]
with the first summand generated by the rational longitude, i.e.,
\[ (i_1)_*(\lambda_1) = (1,0,0). \]
We pick a class $\mu_1 \subset H_1(T;\Z)$ that is dual to $\lambda_1$, meaning they form an integral basis of $H_1(T;\Z)$, and write $(i_1)_*(\mu_1) = (x,y,n)$ in these coordinates.  If $x\neq 0$ then we can replace $\mu_1$ with $\mu_1-x\lambda_1$ in order to arrange that $x=0$.  Moreover, applying half lives half dies over $\F=\Z/p\Z$, we see that
\[ \Span_{\Z/p\Z}\big( (i_1)_*(\lambda_1), (i_1)_*(\mu_1) \big) = \Span_{\Z/p\Z} \big( (1,0,0), (0,y, n\text{ mod }p) \} \]
is 1-dimensional, so $y$ is zero and $n$ is a multiple of $p$.  And then over $\F=\Z/\ell\Z$, where $\ell\neq p$ is prime, we have $H_1(M_1;\F) \cong \F$ and
\[ \Span_{\Z/\ell\Z}\big( (i_1)_*(\lambda_1), (i_1)_*(\mu_1) \big) = \Span_{\Z/\ell\Z} \big( 0, n\text{ mod }\ell \big), \]
which can only be $1$-dimensional if $n$ is not a multiple of $\ell$.  Thus up to changing the sign of the $\Z$ summand, we can write $n = p^e$ for some integer $e \geq 1$.

We now consider the pair
\[ z = \big( (0,0,p^{e-1}), 0 \big) \in H_1(M_1) \oplus H_1(M_2). \]
This cannot lie in the image of $i$: the element $(0,0,p^{e-1})$ is not in the image of $(i_1)_*$, since it does not belong to the span of $(i_1)_*(\lambda_1)=(1,0,0)$ and $(i_2)_*(\mu_1)=(0,0,p^e)$.  Thus $q(z)$ is nonzero by exactness.  Since $q(z) \neq 0$ lies in $H_1(Y) \cong (\Z/p\Z)^r$, it cannot be a multiple of $p$, so neither can $z$ and we must therefore have $e=1$.  Moreover, we know that $q(pz) = p\cdot q(z) = 0$, so the pair
\[ pz = \big( (0,0,p^e), 0\big) = \big( (i_1)_*(\mu_1), 0 \big) \]
lies in the image of $i$.  In other words, there is a class $\alpha \in H_1(T;\Z)$ such that
\begin{align*}
(i_1)_*(\alpha) &= (i_1)_*(\mu_1), \\
(i_2)_*(\alpha) &= 0.
\end{align*}
The first relation implies that $\alpha = \mu_1 + pk\lambda_1$ for some $k\in\Z$, so $\alpha$ and $\lambda_1$ are also dual classes.  This means that $\alpha$ is primitive, so the second relation now says that $\alpha$ is a rational longitude for $M_2$, hence $\alpha = \pm \lambda_2$.  In particular $\lambda_2$ is nullhomologous in $M_2$, and the elements $\{\lambda_1, \lambda_2\} = \{\lambda_1, \pm(\mu_1+pk\lambda_1)\}$ form an integral basis for $H_1(T;\Z)$, as claimed.
\end{proof}

\begin{lemma} \label{lem:essential-coordinates}
Under the hypotheses of Lemma~\ref{lem:rational-lambda-basis}, suppose that the rational longitude $\lambda_1$ is not nullhomologous.  Then we can write
\begin{align*}
H_1(M_1;\Z) &\cong (\Z/p\Z)^{n_1} \oplus \Z, &
H_1(M_2;\Z) &\cong (\Z/p\Z)^{n_2} \oplus \Z,
\end{align*}
with $n_1 \geq 1$, such that the integral basis $\lambda_1,\lambda_2$ of $H_1(T)$ satisfies
\begin{align*}
(i_1)_*(\lambda_1) &= ((1,0,\dots,0), 0), &
(i_2)_*(\lambda_1) &= (0,1), \\
(i_1)_*(\lambda_2) &= ((0,0,\dots,0), p), &
(i_2)_*(\lambda_2) &= (0,0)
\end{align*}
in these coordinates.  If $H_1(Y) \cong (\Z/p\Z)^r$ then we also have $n_1+n_2 = r-1$, and Dehn filling either of the $M_i$ along the other rational longitude gives us
\begin{align*}
H_1(M_1(\lambda_2);\Z) &\cong (\Z/p\Z)^{n_1+1}, &
H_1(M_2(\lambda_1);\Z) &\cong (\Z/p\Z)^{n_2}.
\end{align*}
\end{lemma}

\begin{proof}
We recall from the proof of Lemma~\ref{lem:rational-lambda-basis} that we can take coordinates
\[ H_1(M_1) \cong (\Z/p\Z) \oplus (\Z/p\Z)^{n_1-1} \oplus \Z \]
such that $n_1 \geq 1$ and 
\begin{align*}
(i_1)_*(\lambda_1) &= (1,0,0) \in H_1(M_1), \\
(i_1)_*(\lambda_2) &= (0,0,p) \in H_1(M_1)
\end{align*}
up to changing the sign of the $\Z$ summand.  We also know that we can write 
\[ H_1(M_2) \cong (\Z/p\Z)^{n_2} \oplus \Z, \]
and that the rational longitude $\lambda_2$ is nullhomologous in $M_2$, so that
\[ (i_2)_*(\lambda_2) = (0,0) \in H_1(M_2) \]
in these coordinates.  Thus by half lives half dies over $\Q$ the element $(i_2)_*(\lambda_1)$ must be non-torsion.  If we write
\[ (i_2)_*(\lambda_1) = (w,m) \in H_1(M_2) \]
then $m$ must therefore be nonzero; we choose a generator of the $\Z$ summand so that $m > 0$.  If some prime $\ell\neq p$ divides $m$, then we take $c\equiv \ell^{-1}\pmod{p}$ and we see that $(w,m) = \ell\cdot (cw,\frac{m}{\ell})$ is $\ell$ times an integral class, so $(i_2)_*$ is zero over $\F=\Z/\ell\Z$, contradicting half lives half dies.  Thus $m=p^f$ for some integer $f \geq 0$.

Returning to the Mayer--Vietoris sequence \eqref{eq:mv-rational-lambda}, we know that
\[ (\Z/p\Z)^r \cong \coker(i) \cong \frac{H_1(M_1) \oplus H_1(M_2)}{\langle i_*(\lambda_1), i_*(\lambda_2) \rangle}. \]
We can define a surjection $\coker(i) \to \Z/p^{f+1}\Z$ in the coordinates
\[ H_1(M_1) \oplus H_1(M_2) \cong \big( (\Z/p\Z) \oplus (\Z/p\Z)^{n_1-1} \oplus \Z \big) \oplus \big( (\Z/p\Z)^{n_2} \oplus \Z \big) \]
by sending $\big( (a,v_1,m_1),(v_2,m_2) \big) \mapsto m_2 - p^f a \pmod{p^{f+1}}$; this is well-defined, since $a\in \Z/p\Z$ defines a unique residue class $p^f a \in \Z/p^{f+1}\Z$, and since we have
\begin{align*}
i_*(\lambda_1) = \big((1,0,0),(w,p^f)\big) &\mapsto p^f - p^f \cdot 1 \equiv 0
\\
i_*(\lambda_2) = \big((0,0,p),(0,0)\big) &\mapsto 0.
\end{align*}
But $\coker(i) \cong (\Z/p\Z)^r$ can only surject onto $\Z/p^{f+1}\Z$ if $f=0$.  Thus $(i_2)_*(\lambda_1) = (w,1)$ for some $w \in (\Z/p\Z)^{n_2}$, and by a change of basis we can take this element rather than $(0,1)$ to be the generator of the $\Z$ summand of $H_1(M_2)$, so that $(i_2)_*(\lambda_1) = (0,1)$.

We have now found coordinates on each $H_1(M_j;\Z)$ so that the images $(i_j)_*(\lambda_1)$ and $(i_j)_*(\lambda_2)$ have the desired form.  The computations of
\begin{align*}
H_1(M_1(\lambda_2);\Z) &\cong \frac{H_1(M_1)}{\langle (i_1)_*(\lambda_2) \rangle}, &
H_1(M_2(\lambda_1);\Z) &\cong \frac{H_1(M_2)}{\langle (i_2)_*(\lambda_1) \rangle}
\end{align*}
follow immediately.  Moreover, in these coordinates we have
\[ \coker(i) \cong \frac{\big( (\Z/p\Z) \oplus (\Z/p\Z)^{n_1-1} \oplus \Z \big) \oplus \big( (\Z/p\Z)^{n_2} \oplus \Z \big)}{ \langle \big((1,0,0),(0,1)\big),\ \big((0,0,p),(0,0)\big) \rangle}, \]
which is readily checked to be isomorphic to $(\Z/p\Z)^{n_1+n_2+1}$.  But $\coker(i) \cong (\Z/p\Z)^r$ as well, so we conclude that $r = n_1+n_2+1$.
\end{proof}

We now complete the proof of Proposition~\ref{prop:nice-essential-basis}.

\begin{proof}[Proof of Proposition~\ref{prop:nice-essential-basis}]
Lemma~\ref{lem:rational-lambda-basis} says that the rational longitudes $\lambda_1$ and $\lambda_2$ form a basis of $H_1(T)$, and that $[\lambda_1] \in H_1(M_1)$ has order $p$ while $\lambda_2$ is nullhomologous in $H_1(M_2)$.  We thus define the $Y_i$ by Dehn fillings along these curves:
\begin{align*}
Y_1 &= M_1(\lambda_2), &
Y_2 &= M_2(\lambda_1)
\end{align*}
and we take $K_1 \subset Y_1$ and $K_2 \subset Y_2$ to be the cores of these Dehn fillings.  It follows that their respective meridians are $\mu_1 = \lambda_2$ and $\mu_2 = \lambda_1$, which are dual to their rational longitudes $\lambda_1$ and $\lambda_2$ respectively, so then $[K_1] \in H_1(Y_1)$ has order $p$ while $[K_2] = 0$ in $H_1(Y_2)$.

Lemma~\ref{lem:essential-coordinates} says that these $Y_i$ have homology of the form
\begin{align*}
H_1(Y_1) &= (\Z/p\Z)^k, &
H_1(Y_2) &= (\Z/p\Z)^\ell
\end{align*}
where $k=n_1+1$ and $\ell = n_2$ in the notation of that lemma, and that $k+\ell = (n_1+1)+n_2 = r$.  Since $n_1 \geq 1$ we have also $k \geq 2$, as claimed.

Finally, since $\lambda_2$ is nullhomologous in $M_2$, Proposition~\ref{prop:solid-torus-pinch} gives us a degree-1 pinching map
\[ Y \to M_1(\lambda_2) \cong Y_1 \]
in which $M_2$ is sent onto a solid torus.  The curve $\lambda_2$ lies in the image of the map $\pi_1(T) \to \pi_1(Y)$, which is injective since $T$ is incompressible; thus $\lambda_2$ is a nontrivial element of $\pi_1(Y)$, but it lies in the kernel of the homomorphism $\pi_1(Y) \to \pi_1(Y_1)$ induced by the above pinching map, so that homomorphism has nontrivial kernel.
\end{proof}

\section{Proof of Theorem~\ref{thm:sl2-abelian-main}} \label{sec:main-proof}

We will now use Proposition~\ref{prop:nice-basis} to prove Theorem~\ref{thm:sl2-abelian-main}, which we restate here for convenience.

\begin{theorem} \label{thm:sl2-abelian}
Let $Y$ be a closed, orientable 3-manifold with $H_1(Y;\Z) \cong (\Z/2\Z)^r$ for some $r \geq 0$.  If $Y$ is not homeomorphic to $\#^r \RP^3$, then there is an irreducible representation $\pi_1(Y) \to \SL(2,\C)$.
\end{theorem}

We first verify Theorem~\ref{thm:sl2-abelian} in the atoroidal case before going on to prove it in general.

\begin{lemma} \label{lem:geometric}
Suppose that $Y$ is a closed, atoroidal 3-manifold, with $H_1(Y;\Z) \cong (\Z/p\Z)^r$ for some prime $p$ and some integer $r \geq 0$.  If $Y$ is $\SL(2,\C)$-reducible, then it must be either $\#^r \RP^3$ or a lens space of order $p \geq 3$.
\end{lemma}

\begin{proof}
If $Y$ is a connected sum then each of its summands must also be $\SL(2,\C)$-reducible with first homology $(\Z/p\Z)^{r'}$ for some $r' \leq r$, so we will assume for now that $Y$ is prime.  Then $Y$ is both prime and atoroidal, so by geometrization it must be either Seifert fibered or hyperbolic.  If $Y$ is hyperbolic then it has a holonomy representation $\pi_1(Y) \hookrightarrow \PSL(2,\C)$, and this always lifts to $\SL(2,\C)$ \cite[Proposition~3.1.1]{cs-splittings}, so $Y$ cannot be $\SL(2,\C)$-reducible.  This leaves only the Seifert fibered case.

Among Seifert fibered manifolds, we know from \cite[Theorem~1.2]{sz-menagerie} that the only rational homology spheres that are $\SU(2)$-abelian are
\begin{enumerate}
\item $S^3$ and lens spaces, \label{i:lens}
\item $\RP^3 \# \RP^3$, \label{i:2rp3}
\item those with base orbifold $S^2(3,3,3)$ and with $|H_1(Y)|$ even, \label{i:333}
\item and those with base orbifold $S^2(2,4,4)$. \label{i:244}
\end{enumerate}
In case~\eqref{i:lens}, the only $Y$ such that $H_1(Y)$ is $p$-torsion are $S^3$ and lens spaces of order $p$; and we can ignore case \eqref{i:2rp3} since it is not prime.  For cases \eqref{i:333} and \eqref{i:244}, we note that given a Seifert fibration
\[ Y \cong S^2((\alpha_1,\beta_1),(\alpha_2,\beta_2),(\alpha_3,\beta_3)), \]
we then have
\begin{equation} \label{eq:sfs-homology}
H_1(Y) = \coker \begin{pmatrix} \alpha_1 & 0 & 0 & \beta_1 \\ 0 & \alpha_2 & 0 & \beta_2 \\ 0 & 0 & \alpha_3 & \beta_3 \\ 1 & 1 & 1 & 0 \end{pmatrix}
\end{equation}
and in particular
\[ |H_1(Y)| = |\alpha_1\alpha_2\beta_3 + \alpha_1\beta_2\alpha_3 + \beta_1\alpha_2\alpha_3|. \]
(See \cite[Lemma~2.9]{sz-menagerie}.)  This quickly rules out case \eqref{i:333}, since if $(\alpha_1,\alpha_2,\alpha_3) = (3,3,3)$ then $|H_1(Y)|$ is always a multiple of $18$ (recalling that it must be even), hence not a prime power.  And for case \eqref{i:244}, where $(\alpha_1,\alpha_2,\alpha_3) = (2,4,4)$, we let $x,y,z,w$ be the generators specified by the presentation \eqref{eq:sfs-homology}, and we define a surjection
\begin{align*}
H_1(Y) &\twoheadrightarrow \Z/4\Z \\
x &\mapsto 2 \\
y,z &\mapsto 1 \\
w &\mapsto 0.
\end{align*}
Since $H_1(Y)$ surjects onto $\Z/4\Z$, it cannot possibly have the form $(\Z/p\Z)^r$ with $p$ prime.  We conclude that the only prime examples are $S^3$ and lens spaces of order $p$.

Finally, we note that if $p=2$ then every prime summand of $Y$ is $\RP^3$, so $Y \cong \#^r \RP^3$ as claimed.  If instead $p \geq 3$, then each summand is a lens space of order $p$; but then there cannot be more than one summand, or else $Y$ would not be $\SL(2,\C)$-reducible by exactly the same construction as in Proposition~\ref{prop:non-prime-curve}, so we have $r \leq 1$ and $Y$ is prime after all.
\end{proof}

\begin{proof}[Proof of Theorem~\ref{thm:sl2-abelian}]
We will suppose in what follows that $H_1(Y;\Z) \cong (\Z/2\Z)^r$ for some $r\geq 0$, but that $Y \not\cong \#^r \RP^3$ is $\SL(2,\C)$-reducible.  We will also assume that $Y$ is prime: otherwise, by assumption there must be a prime summand $Y' \not\cong \RP^3$, and then $Y'$ is also $\SL(2,\C)$-reducible with $2$-torsion homology, so we might as well replace $Y$ with $Y'$.  Lemma~\ref{lem:geometric} says that if $Y$ is atoroidal then $Y \cong \#^r \RP^3$, so we may also assume that $Y$ contains an incompressible torus.

Since $Y$ is prime and contains an incompressible torus $T$, we can write
\[ Y = M_1 \cup_T M_2 \]
where each $M_i$ is irreducible and has incompressible boundary.  (The torus $T$ must separate $Y$ because $Y$ is a rational homology sphere.) We split the argument into three cases, depending on the rational longitudes $\lambda_i$ of the $M_i$: in the first two we suppose that the $\lambda_i$ are nullhomologous, so one of the conclusions of Proposition~\ref{prop:nice-basis} applies, and we number these cases according to the relevant conclusion of that proposition.  In the remaining case, at least one of the $\lambda_i$ is essential, so Proposition~\ref{prop:nice-essential-basis} applies instead.  Propositions~\ref{prop:nice-basis} and \ref{prop:nice-essential-basis} each give us closed manifolds $Y_i$ and knots $K_i \subset Y_i$ whose exteriors are the $M_i$, so we will refer freely to these pairs $(Y_i,K_i)$ in the discussion below.

\vspace{1em}
\noindent{\bf Case \ref{i:y2-hrp3}.}  In this case the $M_i$ are complements of non-trivial, nullhomologous knots that have been spliced together by gluing meridians to longitudes and vice versa.  We apply Theorem~\ref{thm:splice-rep} to get an irreducible representation $\rho: \pi_1(Y) \to \SU(2)$, hence if $Y$ is $\SL(2,\C)$-reducible then this case cannot occur.

\vspace{1em}
\noindent{\bf Case \ref{i:y2-hs3}.}  In this case we have degree-1 maps $Y \to (Y_1)_2(K_1)$ and $Y \to (Y_2)_2(K_2)$, with
\begin{align*}
H_1((Y_1)_2(K_1)) &\cong (\Z/2\Z)^{k+1}, &
H_1((Y_2)_2(K_2)) & \cong (\Z/2\Z)^{\ell+1}
\end{align*}
and $k+\ell = r-1$.  The knots $K_i \subset Y_i$ are nullhomologous, and the degree-1 maps $Y \to (Y_i)_2(K_i)$ for $i=1,2$ tell us that each of the $(Y_i)_2(K_i)$ must be $\SL(2,\C)$-reducible as well.

If $(Y_1)_2(K_1) \cong \#^{k+1} \RP^3$ and $(Y_2)_2(K_2) \cong \#^{\ell+1} \RP^3$, then Theorem~\ref{thm:gluing-rep} tells us that $Y$ cannot be $\SL(2,\C)$-reducible or even $\SU(2)$-abelian, a contradiction.  Thus without loss of generality we must have $(Y_1)_2(K_1) \not\cong \#^{k+1} \RP^3$.  In particular $(Y_1)_2(K_1)$ is $\SL(2,\C)$-reducible with first homology $(\Z/2\Z)^{k+1}$, but it is not homeomorphic to $\#^{k+1} \RP^3$.  We let $Y'$ be a prime summand of $(Y_1)_2(K_1)$ different from $\RP^3$ (which may be $(Y_1)_2(K_1)$ itself), and then by collapsing the other prime summands to $S^3$ we have a degree-1 map
\[ Y \to (Y_1)_2(K_1) \to Y'. \]
Here $Y'$ is prime by construction, it is $\SL(2,\C)$-reducible since $(Y_1)_2(K_1)$ is, and $H_1(Y')$ is $2$-torsion since it is a summand of $H_1( (Y_1)_2(K_1) ) \cong (\Z/2\Z)^{k+1}$.

\vspace{1em}
\noindent{\bf Case 3.} In this case one of the $\lambda_i$ is essential, so we suppose without loss of generality that $\lambda_1$ is nonzero in $H_1(M_1;\Z)$.  Then Proposition~\ref{prop:nice-essential-basis} applies: we have
\begin{align*}
H_1(Y_1) &\cong (\Z/2\Z)^k, &
H_1(Y_2) &\cong (\Z/2\Z)^\ell
\end{align*}
where $k \geq 2$ and $k+\ell = r$; the knot $K_1$ is homologically essential in $Y_1$, with rational longitude of order $2$, while $K_2 \subset Y_2$ is nullhomologous; and the gluing of $\partial M_1$ to $\partial M_2$ identifies $\mu_1 \sim \lambda_2$ and $\lambda_1 \sim \mu_2$.  We now apply Theorem~\ref{thm:splice-essential-rep} to see that $Y$ cannot be $\SU(2)$-abelian, a contradiction.  Thus if $Y$ is $\SL(2,\C)$-reducible then this case does not occur.

\vspace{1em}
In each of the three cases above, we have found either a contradiction or a degree-1 map of the form $f:Y \to Y'$, where $Y' \not\cong \RP^3$ is prime and $\SL(2,\C)$-reducible and $H_1(Y')$ is $2$-torsion, and the map
\[ f_*: \pi_1(Y) \to \pi_1(Y') \]
is a surjection with non-trivial kernel.  We can thus replace $Y$ with $Y'$ and repeat.

This process produces an infinite sequence of closed, prime 3-manifolds and degree-1 maps
\[ Y = Y_1 \xrightarrow{f_1} Y_2 \xrightarrow{f_2} Y_3 \xrightarrow{f_3} \cdots, \]
in which none of the $f_i$ are homotopy equivalences because the maps $(f_i)_*: \pi_1(Y) \to \pi_1(Y_{i+1})$ are not injective.  But Theorem~\ref{thm:rong} says that such a sequence cannot exist, so we conclude that our original manifold $Y \not\cong \#^r\RP^3$ could not have been $\SL(2,\C)$-reducible after all.
\end{proof}

\section{From $\Z/p\Z$ to $p$-torsion homology} \label{sec:p-torsion}

In this section we consider $\SL(2,\C)$-reducible $3$-manifolds whose first homology is $p$-torsion for some odd prime $p$.  Our goal is to show the following, which in favorable situations reduces their classification to the case where the homology is in fact cyclic.

\begin{theorem} \label{thm:mod-p-to-p-torsion}
Let $p \geq 3$ be an odd prime with the property that every closed, $\SL(2,\C)$-reducible $3$-manifold with first homology $\Z/p\Z$ is a lens space.  If $Y$ is a closed, $\SL(2,\C)$-reducible $3$-manifold with $H_1(Y;\Z) \cong (\Z/p\Z)^r$ for some $r \geq 1$, then $r=1$ and $Y$ is a lens space.
\end{theorem}

In practice one has to check even less than the stated hypothesis: in Theorem~\ref{thm:p-torsion-strong} we will give a stronger, but much less concise, version of this theorem.

\begin{remark} \label{rem:order-37}
There are many odd primes $p$ that do not satisfy the hypothesis of Theorem~\ref{thm:mod-p-to-p-torsion}.  Indeed, Motegi \cite[\S3]{motegi} produced toroidal, $\SL(2,\C)$-abelian manifolds $Y$ by gluing together the exteriors of any two torus knots $T_{a,b}$ and $T_{c,d}$, identifying the meridian of one with the Seifert fiber of the other and vice versa; then $H_1(Y)$ is cyclic of order $|abcd-1|$, which may be prime.  For example, taking $T_{2,3}$ and $T_{-2,3}$ as our torus knots shows that the hypothesis fails for $p=37$; taking $T_{2,3}$ and $T_{\pm2,5}$ rules out $p=59$ and $p=61$; and so on.
\end{remark}

One important difference from the case $p=2$ is that if $p$ is odd, then $\SL(2,\C)$-reducible manifolds with $p$-torsion homology are always prime, as the following lemma shows.

\begin{lemma} \label{lem:p-torsion-prime}
Let $p \geq 3$ be an odd prime, and suppose that $Y$ is a closed, $\SL(2,\C)$-reducible $3$-manifold such that $H_1(Y;\Z)$ is $p$-torsion.  Then $Y$ is prime.
\end{lemma}

\begin{proof}
Suppose not, and write $Y = Y_1 \# Y_2$, where neither summand is $S^3$.  Then neither $Y_1$ nor $Y_2$ can be a homology sphere, since otherwise it would not be $\SL(2,\C)$-reducible by Theorem~\ref{thm:zentner-main} and so neither would $Y$.  This means that each $H_1(Y_i)$ is $p$-torsion and non-trivial, so each $\pi_1(Y_i)$ surjects onto $H_1(Y_i)$ and hence onto $\Z/p\Z$, and then we have a surjection
\begin{equation} \label{eq:p-composite-surjection}
\pi_1(Y) \cong \pi_1(Y_1) \ast \pi_1(Y_2) \twoheadrightarrow (\Z/p\Z) \ast (\Z/p\Z).
\end{equation}
Since $p \geq 3$, there is a non-abelian homeomorphism
\begin{equation} \label{eq:p-p-su2}
(\Z/p\Z) \ast (\Z/p\Z) \to \SU(2)
\end{equation}
defined by sending generators of each $\Z/p\Z$ factor to the unit quaternions $\exp(\frac{2\pi}{p}j)$ and $\exp(\frac{2\pi}{p}k)$, respectively.  Composing \eqref{eq:p-composite-surjection} and \eqref{eq:p-p-su2} gives an irreducible representation $\pi_1(Y) \to \SU(2)$, so we have a contradiction.
\end{proof}

Lemma~\ref{lem:p-torsion-prime} simplifies some parts of the story, because we no longer have to worry about connected sums of $\SL(2,\C)$-reducible manifolds, as we did for $\#^r \RP^3$ in the $2$-torsion case.  We already encountered this fact in Lemma~\ref{lem:geometric}, where we saw that the only atoroidal examples are lens spaces of order $p$.

\subsection{Zero-surgery on knots in lens spaces}

We begin by generalizing Theorem~\ref{thm:km-zero-surgery} to nullhomologous knots in arbitrary lens spaces.

\begin{proposition} \label{prop:surgery-lens-space}
Let $K \subset Y$ be a nullhomologous knot in $S^3$ or a lens space, and let $w \in H^2(Y_0(K);\Z)$ be Poincar\'e dual to the image in $Y_0(K)$ of a meridian of $K$.  Then $I^w_*(Y_0(K)) \neq 0$ if and only if $K$ is not an unknot in $Y$.
\end{proposition}

\begin{proof}
The case $Y=S^3$ is Theorem~\ref{thm:km-zero-surgery}, so we can assume that $Y$ is a non-trivial lens space.
Since $K$ is nullhomologous in $Y$, we know that $K$ is in fact nullhomotopic in $Y$.  Hom and Lidman \cite[Corollary~1.2]{hom-lidman-spheres} proved that since $Y$ is a prime rational homology sphere and $K$ is nullhomotopic, the manifold $Y_0(K)$ contains a non-separating 2-sphere if and only if $K$ is unknotted, and then the proposition follows from Proposition~\ref{prop:nonzero-isharp} below.
\end{proof}

We devote the remainder of this subsection to proving Proposition~\ref{prop:nonzero-isharp}, which generalizes Theorem~\ref{thm:irreducible-nonzero} for manifolds $Y$ with first Betti number $1$; the key difference is that we do not require $Y$ to be irreducible.

\begin{proposition} \label{prop:nonzero-isharp}
Let $Y$ be a closed 3-manifold with $b_1(Y)=1$, and let $w \in H^2(Y;\Z)$ satisfy $w \cdot R = 1$ for some surface $R \subset Y$.  Then $I^w_*(Y) = 0$ if and only if $Y$ contains a non-separating two-sphere.
\end{proposition}

The proof of Proposition~\ref{prop:nonzero-isharp} makes use of several basic properties of framed instanton homology $I^\#(Y,\lambda)$ over a field of characteristic zero, including a connected sum theorem relating it to the usual instanton homology of an admissible bundle; we will refer to \cite{scaduto} for all of the needed results.

\begin{lemma} \label{lem:isharp-qhs3}
If $Y$ is a rational homology sphere, then $I^\#(Y,\lambda) \neq 0$ for any $\lambda$.
\end{lemma}

\begin{proof}
The invariant $I^\#(Y,\lambda)$ is equipped with a $\Z/2\Z$ grading, and its Euler characteristic with respect to this grading is
\[ \chi(I^\#(Y,\lambda)) = |H_1(Y;\Z)| > 0 \]
according to \cite[Corollary~1.4]{scaduto}, so we must have $\dim I^\#(Y,\lambda) > 0$.
\end{proof}

\begin{lemma} \label{lem:admissible-nonzero}
Let $w \to Y$ be an admissible Hermitian line bundle, and $\lambda \in H_1(Y;\Z)$ the Poincar\'e dual of $c_1(w)$.  Then $I^w_*(Y) = 0$ if and only if $I^\#(Y,\lambda) = 0$.
\end{lemma}

\begin{proof}
Scaduto \cite[Theorem~1.3]{scaduto} proved that
\[ I^\#(Y,\lambda) \cong \ker(u^2-64) \otimes H_*(S^3), \]
where $u=4\mu(\pt)$ is a degree-4 operator on the $\Z/8\Z$-graded invariant $I^w_*(Y)$, but only we take the kernel of the action of $u^2-64$ on four consecutive gradings.  The operator $u^2-64$ is nilpotent \cite{froyshov}, so $\ker(u^2-64) = 0$ if and only if $I^w_*(Y)$ is zero in those gradings; and then $u$ restricts to an isomorphism $I^w_*(Y) \xrightarrow{\cong} I^w_{*+4}(Y)$, so this is equivalent to $I^w_*(Y) = 0$ in all gradings.
\end{proof}

\begin{proof}[Proof of Proposition~\ref{prop:nonzero-isharp}]
We write the prime decomposition of $Y$ as
\[ Y \cong Y_0 \# Y_1 \# \dots \# Y_k, \]
where $Y_0$ is the unique summand with $b_1(Y_0) = 1$ and then the $Y_i$ with $i \geq 1$ are all rational homology spheres.  If we write the Poincar\'e dual $\lambda \in H_1(Y;\Z)$ of $w$ as $\lambda = \lambda_0+\dots+\lambda_k$ with $\lambda_i \in H_1(Y_i)$ for all $i$, then $I^\#$ satisfies a K\"unneth formula
\[ I^\#(Y,\lambda) \cong \bigotimes_{i=0}^k I^\#(Y_i,\lambda_i); \]
this is explained in \cite[\S7.7]{scaduto} when the $\lambda_i$ are all zero, but the same proof works in full generality.  By Lemma~\ref{lem:isharp-qhs3} we have $I^\#(Y_i,\lambda_i) \neq 0$ for all $i \geq 1$, so $I^\#(Y,\lambda) \neq 0$ if and only if $I^\#(Y_0,\lambda_0) \neq 0$.

Two applications of Lemma~\ref{lem:admissible-nonzero} now tell us that $I^w_*(Y) \neq 0$ if and only if $I^{w_0}_*(Y_0) \neq 0$, where $w_0 = w|_{Y_0}$ is the Poincar\'e dual to $\lambda_0$.  Since $Y_0$ is prime, either it is $S^1\times S^2$ and then $I^{w_0}_*(Y_0) = 0$, or it is irreducible and then $I^{w_0}_*(Y_0) \neq 0$ by Theorem~\ref{thm:irreducible-nonzero}.  Since $Y$ has a non-separating $S^2$ if and only if one of its prime summands is $S^1\times S^2$, it follows that $Y$ contains such a sphere if and only if $I^{w_0}_*(Y_0) = 0$, hence if and only if $I^w_*(Y) = 0$ as claimed.
\end{proof}

\subsection{Splicing knots in lens spaces}

This subsection is devoted to proving an analogue of Theorem~\ref{thm:splice-rep}, in which the knots can lie in lens spaces rather than in $3$-manifolds whose first homology is $2$-torsion.

\begin{proposition} \label{prop:splice-s3-lens}
Let each of $Y_1$ and $Y_2$ be either $S^3$ or a lens space, and let $K_1 \subset Y_1$ and $K_2 \subset Y_2$ be nullhomologous knots with irreducible, boundary-incompressible exteriors.  Form a closed, toroidal $3$-manifold
\[ Y = E_{K_1} \cup_{\partial} E_{K_2} \]
by gluing the meridian $\mu_1$ and longitude $\lambda_1$ of $K_1$ to the longitude $\lambda_2$ and meridian $\mu_2$ of $K_2$, respectively.  Then there is a representation
\[ \rho: \pi_1(Y) \to \SU(2) \]
with non-abelian image.
\end{proposition}

\begin{proof}
Suppose that each $Y_i$ is a lens space of order $n_i \geq 3$.  Then we can define representations
\[ \rho_i: \pi_1(E_{K_i}) \twoheadrightarrow \frac{\pi_1(E_{K_i})}{\llangle \mu_i\rrangle} \cong \pi_1(Y_i) \cong \Z/n_i\Z \hookrightarrow \SU(2),\]
satisfying $\rho_i(\mu_i) = 1$, and we have $\rho_i(\lambda_i) = 1$ since the image of $\rho_i$ is abelian.  Each $\rho_i$ restricts to the trivial representation on $\pi_1(T)$, so they glue together to give a representation $\rho: \pi_1(Y) \to \SU(2)$, and we can guarantee that $\rho$ has non-abelian image by choosing to send generators of $\pi_1(Y_1) \cong \Z/n_1\Z$ and $\pi_1(Y_2) \cong \Z/n_2\Z$ to the unit quaternions $\exp(\frac{2\pi}{n_1}j)$ and $\exp(\frac{2\pi}{n_2}k)$, respectively.

From now on we assume without loss of generality that $Y_1$ is either $S^3$ or $\RP^3$; the proof now follows essentially the same argument as Theorem~\ref{thm:splice-rep}.  Neither $K_1$ nor $K_2$ is unknotted, so if $w_1 \in H^2((Y_1)_0(K_1))$ and $w_2 \in H^2((Y_2)_0(K_2))$ are Poincar\'e dual to meridians of $K_1$ and $K_2$ then we know that
\begin{align*}
I^{w_1}_*((Y_1)_0(K_1)) &\neq 0, &
I^{w_2}_*((Y_2)_0(K_2)) &\neq 0
\end{align*}
by Proposition~\ref{prop:surgery-lens-space}.  Since both $Y_1$ and $Y_2$ are $\SU(2)$-abelian, the character varieties $X(E_{K_1})$ and $X(E_{K_2})$ have well-defined images in the cut-open pillowcase 
\[ \cP = [0,\pi] \times (\R/2\pi\Z) \]
of \S\ref{ssec:cut-open} (see Lemma~\ref{lem:alpha-pi}), and Theorem~\ref{thm:pillowcase-loop} provides us with essential closed curves
\begin{align*}
C_1 &\subset j(X(E_{K_1})), &
C_2 &\subset j(X(E_{K_2}))
\end{align*}
in the cut-open pillowcase images of each.

Just as in the proof of Theorem~\ref{thm:splice-rep}, the curves $C_1$ and $C_2$ now give rise to continuous paths
\[ \gamma^\ell_t = (\alpha^\ell_t, \beta^\ell_t): [0,1] \to [0,\pi] \times [0,2\pi], \qquad \ell=1,2 \]
such that for each $\ell$ we have
\begin{itemize}
\item $\beta^\ell_0 = 0$, $\beta^\ell_1 = 2\pi$, and $0 < \beta^\ell_t < 2\pi$ for $0 < t < 1$;
\item $0 < \alpha^\ell_t < \pi$ for $0<t<1$ by Lemma~\ref{lem:alpha-pi}, since $\beta^\ell_t \not\in 2\pi\Z$;
\item and each $\gamma^\ell_t$ is the pillowcase image of some $\rho^\ell_t: \pi_1(E_{K_\ell}) \to \SU(2)$ satisfying
\begin{align*}
\rho^\ell_t(\mu_\ell) &= \begin{pmatrix} e^{i\alpha^\ell_t} & 0 \\ 0 & e^{-i\alpha^\ell_t} \end{pmatrix}, &
\rho^\ell_t(\lambda_\ell) &= \begin{pmatrix} e^{i\beta^\ell_t} & 0 \\ 0 & e^{-i\beta^\ell_t} \end{pmatrix}.
\end{align*}
In particular $\rho^\ell_t$ is irreducible for $0<t<1$, since $\rho^\ell_t(\lambda_\ell) \neq 1$.
\end{itemize}
Lemma~\ref{lem:corners-limit} tells us slightly more about $\alpha^1_t$, namely that
\[ 0 < \alpha^1_t < \pi \quad \text{for all } t \in [0,1] \]
since $Y_1$ is $\SU(2)$-abelian and $H_1(Y_1)$ is either trivial or $\Z/2\Z$.  The paths
\[ \big\{ (\alpha^1_t, \beta^1_t) \big\}_{t\in[0,1]} \quad\text{and}\quad \big\{ (\beta^2_t,\alpha^2_t) \big\}_{t\in[0,1]} \]
must intersect exactly as before, say at some point $(\tilde\alpha,\tilde\beta)$, where $0<\tilde\alpha<\pi$ since $\tilde\alpha = \alpha^1_t$ for some $t$.  This point of intersection gives rise to a representation $\rho: \pi_1(Y) \to \SU(2)$, and the restriction $\rho|_{E_{K_2}}$ must have pillowcase coordinates $(\tilde\beta,\tilde\alpha)$.  Then $0 < \tilde\alpha < \pi$ implies that $\rho|_{E_{K_2}}(\lambda_2) \neq 1$, so $\rho|_{E_{K_2}}$ cannot have abelian image and thus neither can $\rho$.
\end{proof}

\subsection{Manifolds with $p$-torsion homology}

We are now ready to prove Theorem~\ref{thm:mod-p-to-p-torsion}, which will follow quickly from the next proposition.

\begin{proposition} \label{prop:p-torsion-dominates-something}
Let $p \geq 3$ be an odd prime such that every closed, $\SL(2,\C)$-reducible $3$-manifold with first homology $\Z/p\Z$ is a lens space.

If $Y$ is a closed, $\SL(2,\C)$-reducible $3$-manifold with $H_1(Y;\Z) \cong (\Z/p\Z)^r$ for some integer $r \geq 2$, then there is a closed $Y'$ with first homology $(\Z/p\Z)^{r'}$ for some $r' \geq 2$ and a degree-1 map
\[ Y \to Y' \]
that is not a homotopy equivalence.  Both $Y$ and $Y'$ are prime, toroidal, and $\SL(2,\C)$-reducible.
\end{proposition}

\begin{proof}
We know that $Y$ is prime by Lemma~\ref{lem:p-torsion-prime}, and that it contains an incompressible torus by Lemma~\ref{lem:geometric}: indeed, if it were atoroidal then it would have to be a lens space, but $H_1(Y)$ is not cyclic.  By the same argument, once we have constructed $Y'$ with the desired homology and degree-1 map $f: Y\to Y'$, it will follow that $Y'$ is $\SL(2,\C)$-reducible, and then that $Y'$ is prime and toroidal.  We thus focus on constructing $Y'$ and the map $f$, which will be a pinch map of the sort provided by Proposition~\ref{prop:solid-torus-pinch}; if it collapses a submanifold bounded by an incompressible torus $T$ to a solid torus, then it will not be a homotopy equivalence, since the kernel of the induced map $f_*: \pi_1(Y) \to \pi_1(Y')$ contains non-trivial elements of the subgroup $\pi_1(T) \subset \pi_1(Y)$.

Since $Y$ is prime and has an incompressible torus $T$, we can write
\[ Y \cong M_1 \cup_T M_2 \]
where each $M_i$ is irreducible, with incompressible torus boundary.  We let $\lambda_i \subset \partial M_i$ denote the rational longitude of each $M_i$.

Suppose first that each of the $\lambda_i$ is nullhomologous in its respective $M_i$.  Then Proposition~\ref{prop:nice-basis} says that we can write each $M_i$ as the exterior of some nullhomologous knot $K_i \subset Y_i$, with $H_1(Y_i;\Z) \cong (\Z/p_i\Z)^{n_i}$, such that one of two cases occurs:

\vspace{1em}\noindent
{\bf Case 1}: $n_1+n_2=r$, and we form $Y$ by gluing $\mu_1$ to $\lambda_2$ and $\lambda_2$ to $\mu_1$.

In this case we use Proposition~\ref{prop:solid-torus-pinch} to pinch either $M_2$ or $M_1$ to a solid torus, giving us degree-1 maps
\begin{align*}
Y &\to M_1(\lambda_2) \cong M_1(\mu_1) = Y_1, \\
Y &\to M_2(\lambda_1) \cong M_2(\mu_2) = Y_2.
\end{align*}
If $n_1=n_2=1$ then the manifolds $Y_1$ and $Y_2$ are $\SL(2,\C)$-reducible with first homology $\Z/p\Z$, so they are both lens spaces by our assumption on $p$, but then Proposition~\ref{prop:splice-s3-lens} says that $Y$ is not even $\SU(2)$-abelian and we have a contradiction.  Now since $n_1+n_2 = r \geq 2$ but $(n_1,n_2) \neq (1,1)$, it follows that $n_i \geq 2$ for some $i$, so we let $Y'$ be the corresponding $Y_i$ and we are done.

\vspace{1em}\noindent
{\bf Case 2}: $n_1+n_2=r-1$.

In this case Proposition~\ref{prop:nice-basis} says that for some $a,b,c$ with $ac-bp=-1$ we have a pair of degree-1 maps
\[ Y \to (Y_1)_{-p/a}(K_1) \qquad\text{and}\qquad Y \to (Y_2)_{p/c}(K_2), \]
neither of which is a homotopy equivalence.  The targets of these maps have first homology $(\Z/p\Z)^{n_1+1}$ and $(\Z/p\Z)^{n_2+1}$, respectively, and
\[ (n_1+1) + (n_2+1) = r+1 \geq 3, \]
so we must have $n_i+1 \geq 2$ for some $i$.  We take $Y'$ to be the corresponding Dehn surgery on $K_i \subset Y_i$.

\vspace{1em}
We have now proved the proposition except in the case where one of the rational longitudes is homologically essential, so we suppose without loss of generality that $\lambda_1$ is nonzero in $H_1(M_1)$.  Now we apply Proposition~\ref{prop:nice-essential-basis} to see that we can write each $M_i$ as the exterior of a knot $K_i \subset Y_i$, with meridian $\mu_i$ and rational longitude $\lambda_i$, such that
\begin{itemize}
\item $H_1(Y_1) \cong (\Z/p\Z)^k$ for some $k \geq 2$;
\item $\lambda_2$ is nullhomologous in $M_2$;
\item and the gluing map $\partial M_1 \cong \partial M_2$ identifies $\mu_1$ with $\lambda_2$.
\end{itemize}
Proposition~\ref{prop:solid-torus-pinch} then gives us a degree-1 map
\[ Y \to M_1(\lambda_2) \cong M_1(\mu_1) \cong Y_1, \]
so we take $Y' = Y_1$ and the proof is complete.
\end{proof}

\begin{proof}[Proof of Theorem~\ref{thm:mod-p-to-p-torsion}]
Let $Y_1 = Y$ be $\SL(2,\C)$-reducible with first homology $(\Z/p\Z)^{r_1}$ for some $r_1 \geq 2$.  By the hypothesis on $p$, Proposition~\ref{prop:p-torsion-dominates-something} provides us with an $\SL(2,\C)$-reducible manifold $Y_2$, whose first homology is $(\Z/p\Z)^{r_2}$ for some $r_2 \geq 2$, and a degree-1 map
\[ f_1: Y_1 \to Y_2 \]
that is not a homotopy equivalence.  We repeat with $Y_2$ in place of $Y_1$ and so on, constructing an infinite sequence
\[ Y_1 \xrightarrow{f_1} Y_2 \xrightarrow{f_2} Y_3 \xrightarrow{f_3} \cdots \]
of degree-1 maps between prime, toroidal manifolds, in which none of the maps $f_i$ is a homotopy equivalence.  This contradicts Theorem~\ref{thm:rong}, so our initial manifold $Y$ cannot exist after all.
\end{proof}

\subsection{A strengthening of Theorem~\ref{thm:mod-p-to-p-torsion}}

We can deduce the conclusion of Theorem~\ref{thm:mod-p-to-p-torsion} from a seemingly much weaker hypothesis on the prime $p$, by a similar appeal to Theorem~\ref{thm:rong}.

\begin{theorem} \label{thm:p-torsion-strong}
Fix an odd prime $p \geq 3$. For any choice of
\begin{itemize}
\item integer homology 3-spheres $Y_1$ and $Y_2$,
\item knots $K_1 \subset Y_1$ and $K_2 \subset Y_2$ with irreducible, boundary-incompressible complements,
\item and integers $a,b,c$ satisfying $ac-bp=-1$ and $0 \leq b < c < \frac{p}{2}$
\end{itemize}
such that both
\[ (Y_1)_{-p/a}(K_1) \quad\text{and}\quad (Y_2)_{p/c}(K_2) \]
are lens spaces, we can form a closed $3$-manifold
\[ Y = E_{K_1} \cup_\partial E_{K_2} \]
by gluing $\partial E_{K_1}$ to $\partial E_{K_2}$ so that
\begin{align*}
\mu_1 &= a\mu_2 + b\lambda_2 \\
\lambda_1 &= p\mu_2 + c\lambda_2
\end{align*}
in the homology of the torus $\partial E_{K_1} \sim \partial E_{K_2}$.  Suppose we have chosen $p$ so that every such $Y$ admits an irreducible representation $\pi_1(M) \to \SL(2,\C)$.

With the above assumption on $p$, if $Y$ is a closed $3$-manifold such that $H_1(Y;\Z) \cong (\Z/p\Z)^r$ for some $r\geq 1$, then either $Y$ is a lens space of order $p$ or there is an irreducible homomorphism $\pi_1(Y) \to \SL(2,\C)$.
\end{theorem}

\begin{remark}
There are exactly $\frac{p-1}{2}$ tuples $(a,b,c)$ to consider in the hypothesis of Theorem~\ref{thm:p-torsion-strong}, since once we have fixed $c$ between $1$ and $\frac{p-1}{2}$ inclusive, the condition $ac-bp=-1$ implies that $b \equiv p^{-1} \pmod{c}$.  For any $p$ this includes $(a,b,c) = (-1,0,1)$, and then for example when $p=5$ we must also consider $(a,b,c)=(2,1,2)$.
\end{remark}

\begin{proof}[Proof of Theorem~\ref{thm:p-torsion-strong}]
By Theorem~\ref{thm:mod-p-to-p-torsion}, it suffices to prove the theorem when $H_1(Y;\Z) \cong \Z/p\Z$, so we will assume from now on that $Y$ is $\SL(2,\C)$-reducible with $H_1(Y;\Z) \cong \Z/p\Z$, but that $Y$ is not a lens space.  Lemmas~\ref{lem:p-torsion-prime} and \ref{lem:geometric} respectively tell us that $Y$ is prime, and that it contains an incompressible torus $T$ since it is not a lens space, so we can write
\[ Y = M_1 \cup_T M_2 \]
where each $M_i$ is irreducible with incompressible boundary $T$.  Moreover, Proposition~\ref{prop:nice-essential-basis} guarantees that the rational longitude of each $M_i$ is nullhomologous in $M_i$, because otherwise we would have $H_1(Y;\Z) \cong (\Z/p\Z)^r$ for some $r \geq 2$.

Following Proposition~\ref{prop:nice-basis} and Remark~\ref{rem:c-at-most-p/2}, we can therefore write each $M_i$ as the exterior of a nullhomologous knot $K_i$ in some $3$-manifold $Y_i$ such that either
\begin{enumerate}
\item $H_1(Y_1) \oplus H_1(Y_2) \cong \Z/p\Z$, and $Y$ is formed by gluing $\mu_1$ to $\lambda_2$ and $\lambda_2$ to $\mu_1$;
\item $H_1(Y_1) \oplus H_1(Y_2) \cong 0$, and $Y$ is formed by some gluing such that
\begin{align*}
\mu_1 &= a\mu_2 + b\lambda_2 \\
\lambda_1 &= p\mu_2 + c\lambda_2
\end{align*}
in $H_1(T)$, where $ac-bp=-1$ and $0 \leq b < c < \frac{p}{2}$.
\end{enumerate}
For the last case, Remark~\ref{rem:c-at-most-p/2} may require us to reverse the orientation of $Y$ and of the $M_j$ in order to achieve $c \leq \frac{p}{2}$ rather than $c < p$, but this does not affect whether or not $Y$ is $\SL(2,\C)$-reducible.  We also note that the inequality $c \leq \frac{p}{2}$ is strict here because $p$ is odd.

In the first case, we have degree-1 pinching maps from $Y$ to each of $Y_1$ and $Y_2$, so $Y_1$ and $Y_2$ must be $\SL(2,\C)$-reducible as well.  One of them is a homology sphere, so it must be $S^3$ by Theorem~\ref{thm:zentner-main}, and the other has first homology $\Z/p\Z$.  If the latter is a lens space then Proposition~\ref{prop:splice-s3-lens} gives us a non-abelian representation $\pi_1(Y) \to \SU(2)$ and hence a contradiction, so it must be toroidal.

In the second case, both $Y_1$ and $Y_2$ are homology spheres, and we have degree-1 pinching maps
\[ Y \to (Y_1)_{-p/a}(K_1) \qquad\text{and}\qquad Y \to (Y_2)_{p/c}(K_2), \]
so both $(Y_1)_{-p/a}(K_2)$ and $(Y_2)_{p/c}(K_1)$ are $\SL(2,\C)$-reducible, with first homology $\Z/p\Z$.  If neither of these is toroidal then they must both be lens spaces, hence by hypothesis there is an irreducible representation $\pi_1(Y) \to \SL(2,\C)$.  But we assumed $Y$ is $\SL(2,\C)$-reducible, so at least one of $(Y_1)_{-p/a}(K_1)$ and $(Y_2)_{p/c}(K_2)$ must be toroidal after all.

In both cases, we have found (up to a possible change of orientation) a degree-1 map $Y \to Y'$, where $Y'$ is $\SL(2,\C)$-reducible and toroidal with $H_1(Y';\Z) \cong \Z/p\Z$, by pinching some submanifold with incompressible torus boundary onto a solid torus.  This is not a homotopy equivalence, so we can repeat this process indefinitely with $Y'$ in place of $Y$ and so on, to get an infinite sequence
\[ Y \to Y' \to Y'' \to \cdots \]
of degree-1 maps which are not homotopy equivalences.  This contradicts Theorem~\ref{thm:rong}, so the claimed $Y$ cannot exist after all.
\end{proof}

\section{Manifolds with $3$-torsion homology} \label{sec:3-torsion}

Our goal in this section is to prove Theorem~\ref{thm:3-torsion-main}, which we restate here.

\begin{theorem} \label{thm:3-torsion}
Let $Y$ be a closed $3$-manifold such that $H_1(Y;\Z)$ is $3$-torsion.  If $Y$ is not homeomorphic to $\pm L(3,1)$, then there is an irreducible representation $\pi_1(Y) \to \SL(2,\C)$.
\end{theorem}

We prove Theorem~\ref{thm:3-torsion} by appealing to Theorem~\ref{thm:p-torsion-strong}.  We note in the hypothesis of Theorem~\ref{thm:p-torsion-strong} that there is a unique triple of integers $(a,b,c)$ with $ac-3b=-1$ and $0 \leq b < c < \frac{3}{2}$, namely $(a,b,c)=(-1,0,1)$, so Theorem~\ref{thm:3-torsion} is now an immediate consequence of the following analogue of Theorem~\ref{thm:gluing-rep}.

\begin{theorem} \label{thm:glue-slope-3}
Let $K_1 \subset Y_1$ and $K_2 \subset Y_2$ be knots such that for each $j=1,2$:
\begin{itemize}
\item the manifold $Y_j$ is an integer homology sphere,
\item the exterior $E_{K_j}$ is irreducible with incompressible boundary, and
\item the Dehn surgery $(Y_j)_3(K_j)$ is a lens space of order $3$.
\end{itemize}
We glue the exteriors along their boundaries to form a toroidal manifold
\[ Y = E_{K_1} \cup_\partial E_{K_2} \]
by identifying $\mu_1 \sim \mu_2^{-1}$ and $\lambda_1 \sim \mu_2^3 \lambda_2$.  Then there is a representation
\[ \rho: \pi_1(Y) \to \SU(2) \]
with non-abelian image.
\end{theorem}

The proof of Theorem~\ref{thm:glue-slope-3} will be similar in spirit to the content of \S\ref{sec:glue-zhs} but simpler, largely because in Lemma~\ref{lem:p-avoid-lines} below, we will be able to put stronger restrictions on the pillowcase images of the various character varieties $X(E_{K_j})$ than we could in the corresponding Proposition~\ref{prop:lines-of-slope-2}.

To set the stage, given an odd prime $p$, we define an involution of the pillowcase by
\[ \sigma_p(\alpha,\beta) = (-\alpha,p\alpha+\beta) = (\alpha,2\pi-(p\alpha+\beta)) \]
by analogy with the map $\sigma$ of Subsection~\ref{ssec:pillowcase-symmetries}.  If we let
\begin{align*}
P &= (0,\pi), \\
Q &= (\pi,\pi)
\end{align*}
as before, then $\sigma_p(P) = P$ but $\sigma_p(Q) = (\pi,0) \neq Q$.  (Similarly, the map $\sigma_p$ does \emph{not} commute with the involution $\tau$ of Lemma~\ref{lem:pillowcase-involution}, because $\sigma_p(\tau(Q)) = P$ but $\tau(\sigma_p(Q)) = (0,0)$.)

\begin{lemma} \label{lem:look-for-intersection-3}
Suppose under the hypotheses of Theorem~\ref{thm:glue-slope-3} that the pillowcase images
\[ i^*(X(E_{K_1})) \quad\text{and}\quad \sigma_3\left(i^*(X(E_{K_2}))\right) \]
intersect at some point other than $(0,0)$ or $(\frac{2\pi}{3},0)$.  Then there is a representation
\[ \rho: \pi_1(Y) \to \SU(2) \]
with non-abelian image.
\end{lemma}

\begin{proof}
Let $(\alpha,\beta) \in X(T^2)$ be the given point of intersection, and write $\sigma_3(\alpha,\beta) = (\gamma,\delta)$; since $\sigma_3$ is an involution, this means that
\[ (\alpha,\beta) = \sigma_3(\gamma,\delta) = (\gamma,2\pi-(3\gamma+\delta)). \]
Now since $(\alpha,\beta) \in i^*(X(E_{K_1}))$ and $(\gamma,\delta) \in i^*(X(E_{K_2}))$, there are representations
\[ \rho_j: \pi_1(E_{K_j}) \to \SU(2), \quad j=1,2 \]
such that
\begin{align*}
\rho_1(\mu_1) &= \begin{pmatrix} e^{i\alpha} & 0 \\ 0 & e^{-i\alpha} \end{pmatrix}, &
\rho_1(\lambda_1) &= \begin{pmatrix} e^{i\beta} & 0 \\ 0 & e^{-i\beta} \end{pmatrix}
\end{align*}
and (using the fact that $(\gamma,\delta) \sim (-\gamma,-\delta)$ in $X(T^2)$)
\begin{align*}
\rho_2(\mu_2) &= \begin{pmatrix} e^{-i\gamma} & 0 \\ 0 & e^{i\gamma} \end{pmatrix}, &
\rho_2(\lambda_2) &= \begin{pmatrix} e^{-i\delta} & 0 \\ 0 & e^{i\delta} \end{pmatrix}.
\end{align*}
This means that 
\begin{align*}
\rho_2(\mu_2^{-1}) = \begin{pmatrix} e^{i\gamma} & 0 \\ 0 & e^{-i\gamma} \end{pmatrix} &= \begin{pmatrix} e^{i\alpha} & 0 \\ 0 & e^{-i\alpha} \end{pmatrix} = \rho_1(\mu_1), \\
\rho_2(\mu_2^3\lambda_2) = \begin{pmatrix} e^{-i(3\gamma+\delta)} & 0 \\ 0 & e^{i(3\gamma+\delta)} \end{pmatrix} &= \begin{pmatrix} e^{i\beta} & 0 \\ 0 & e^{-i\beta} \end{pmatrix} = \rho_1(\lambda_1)
\end{align*}
and so $\rho_1$ and $\rho_2$ glue together to give us the desired representation $\rho$.

At this point we need only show that $\rho$ has non-abelian image.  But if its image is abelian then each of $\rho_1$ and $\rho_2$ must have abelian image as well, hence $\beta \equiv \delta \equiv 0 \pmod{2\pi}$.  If $\beta \in 2\pi\Z$ then we know that $\delta = 2\pi - (3\alpha+\beta)$ is a multiple of $2\pi$ if and only if $3\alpha$ is, so we must have
\[ (\alpha,\beta) = (0,0) \text{ or } (\tfrac{2\pi}{3},0) \]
in the pillowcase.  Since we have assumed that our given intersection $(\alpha,\beta)$ is not either one of these points, we conclude that $\rho$ has non-abelian image after all.
\end{proof}

The remainder of this section will be devoted to finding a point of intersection to which we can apply Lemma~\ref{lem:look-for-intersection-3}.  Most of the argument applies equally well to other odd primes $p$, so we will not specialize to $p=3$ until the end.  

To summarize the upcoming argument, each character variety will provide us with a closed curve $\gamma_i$ in the pillowcase.  Each $\gamma_i$ is homologically essential in the complement of two points $P=(0,\pi)$ and $Q=(\pi,\pi)$, and is further constrained by the fact that the corresponding knots have lens space surgeries.  If we choose the $\gamma_i$ carefully, this will imply that $\gamma_1$ and $\sigma_3(\gamma_2)$ must intersect somewhere.  Now if they meet at the point $(\frac{2\pi}{3},0)$, then Lemma~\ref{lem:look-for-intersection-3} seems to say that we are stuck; but we will show that they must be transverse there, and since any pair of closed curves in the $2$-sphere $X(T^2) \cong S^2$ have intersection number zero, we can then deduce the existence of a second, more useful point of intersection.  This last part does not readily generalize to other primes $p$, unfortunately, because we have to show that the corresponding curves intersect away from one of the $\frac{p-1}{2}$ points of the form $(\frac{2k\pi}{p},0)$ where $1 \leq k \leq \frac{p-1}{2}$, and when $p > 3$ there are at least two such points.

We begin with the following analogue of Proposition~\ref{prop:lines-of-slope-2}, which is illustrated in Figure~\ref{fig:slope-p} when $p=3$.

\begin{lemma} \label{lem:p-avoid-lines}
Let $K$ be a knot in an integral homology sphere $Y$, and suppose for some prime $p \geq 3$ that $Y_p(K)$ is a lens space of order $p$.  If $\rho: \pi_1(E_K) \to \SU(2)$ has pillowcase coordinates $i^*([\rho]) = (\alpha,\beta)$ where $p\alpha+\beta \in \pi\Z$, then $\rho$ is reducible and $\beta\equiv 0\pmod{2\pi}$.  In this case there is an open neighborhood of $(\alpha,\beta) \in X(T^2)$ that does not contain the images of any irreducible representations.
\end{lemma}

\begin{figure}
\begin{tikzpicture}[style=thick]
\begin{scope}
  \draw plot[mark=*,mark size = 0.5pt] coordinates {(0,0)(3,0)(3,6)(0,6)} -- cycle; 
  \begin{scope}[decoration={markings,mark=at position 0.55 with {\arrow[scale=1]{>>}}}]
    \draw[postaction={decorate}] (0,0) -- (0,3);
    \draw[postaction={decorate}] (0,6) -- (0,3);
  \end{scope}
  \begin{scope}[decoration={markings,mark=at position 0.575 with {\arrow[scale=1]{>>>}}}]
    \draw[postaction={decorate}] (3,0) -- (3,3);
    \draw[postaction={decorate}] (3,6) -- (3,3);
  \end{scope}
  \begin{scope}[decoration={markings,mark=at position 0.6 with {\arrow[scale=1]{>>>>}}}]
    \draw[postaction={decorate}] (3,6) -- (0,6);
    \draw[postaction={decorate}] (3,0) -- (0,0);
  \end{scope}
  \draw[dotted] (0,3) -- (3,3);
  \draw[thin,|-|] (0,-0.4) node[below] {\small$0$} -- node[midway,inner sep=1pt,fill=white] {$\alpha$} ++(3,0) node[below] {\small$\vphantom{0}\pi$};
  \draw[thin,|-|] (-0.3,0) node[left] {\small$0$} -- node[midway,inner sep=1pt,fill=white] {$\beta$} ++(0,6) node[left] {\small$2\pi$};
  \begin{scope}[color=red,style=ultra thick,shorten >=-4pt]
    \begin{scope}[shorten <=-4pt]
      \draw[o-o] (0,6) -- node[above,sloped] {avoid these lines} (2,0);
      \draw[o-o] (1,6) -- (3,0);
      \draw[o-o] (0,0) (3,6);
      \draw[o-o] (3,6) (0,0);
    \end{scope}
    \draw[-o] (0,3) -- (1,0);
    \draw[-o] (3,3) -- (2,6);
  \end{scope}
\end{scope}
\end{tikzpicture}
\caption{If $Y_3(K)$ is a lens space of order $3$, then the image $i^*(X(E_K))$ must avoid the lines $\{3\alpha+\beta\in \pi\Z\}$, except where $\beta\equiv0\pmod{2\pi}$.}
\label{fig:slope-p}
\end{figure}
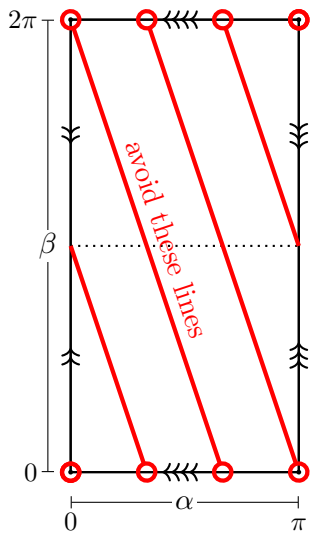

\begin{proof}
We suppose first that $i^*([\rho]) = (\alpha,\beta)$ where $p\alpha+\beta$ is an integral multiple of $2\pi$.  Then $\rho(\mu^p\lambda)=1$, so $\rho$ factors through
\[ \frac{\pi_1(E_K)}{\llangle \mu^p\lambda \rrangle} \cong \pi_1(Y_p(K)) \cong \Z/p\Z, \]
and therefore its image is cyclic.  This means that $\rho$ has abelian image, and hence $\rho(\lambda)=1$ (equivalently, $\beta=0$) as usual.

Now suppose instead that $p\alpha+\beta$ is an odd multiple of $\pi$, so $\rho(\mu^p\lambda) = -1$.  Then the central character
\[ \chi: \pi_1(E_K) \twoheadrightarrow H_1(E_K) \cong \Z \to \{\pm1\} \]
sending $\mu$ to $-1$ satisfies $\chi(\mu^p\lambda) = (-1)^p = -1$ since $p$ is odd, so $\chi\cdot\rho$ is a representation sending $\mu^p\lambda$ to $1$.  We conclude as above that $\chi\cdot \rho$ has cyclic image, hence so does $\rho$ itself, and then $\rho(\lambda)=1$ and $\beta=0$ once again.

Finally, suppose that we have a sequence of irreducible representations $\rho_n \in R^\irr(E_K)$ whose pillowcase images $i^*([\rho_n])$ converge to a point $(\alpha,\beta)$ with $p\alpha+\beta \in \pi\Z$.  Since $R(E_K)$ is compact, we can pass to a convergent subsequence, whose limit $\rho$ satisfies $i^*([\rho]) = (\alpha,\beta)$; since $p\alpha+\beta \in \pi\Z$, we deduce from above that $\rho$ is abelian, hence $\beta=0$ and $\alpha = \frac{k\pi}{p}$ for some integer $k$ with $0 \leq k \leq p$.  In addition, Lemma~\ref{lem:corners-limit} says that $\alpha$ cannot be $0$ or $\pi$ since $Y_p(K)$ is $\SU(2)$-abelian and $\rho$ is a limit of irreducible representations, and therefore $0 < k < p$.

Since $\rho$ is a reducible limit of irreducible representations, Heusener, Porti, and Su\'arez Peir\'o \cite[Theorem~2.7]{heusener-porti-suarez} also proved that the Alexander polynomial of $K$ satisfies
\[ \Delta_K(e^{2k\pi i/p}) = \Delta_K(e^{2\alpha i}) = 0. \]
(They attribute this to Klassen \cite[Theorem~19]{klassen}, who proved it for knots in $S^3$.)  Thus $\Delta_K(t)$ vanishes at a primitive $p$th root of unity, so a result of Boyer and Nicas \cite[Lemma~1.4]{boyer-nicas} says that the fundamental group
\[ \pi_1(Y_p(K)) \cong \Z/p\Z \]
is not \emph{cyclically finite}.  By definition this means that some normal subgroup of $\Z/p\Z$ has infinite abelianization, which is absurd, so $\rho$ cannot be a limit of irreducible representations after all.
\end{proof}

If $K$ is a knot in a homology sphere $Y$, and $Y_p(K)$ is a lens space of order $p$ for some prime $p \geq 3$, then Lemma~\ref{lem:p-avoid-lines} implies that the pillowcase image
\[ i^*(X(E_K)) \subset X(T^2) \]
avoids the points $P=(0,\pi)$ and $Q=(\pi,\pi)$.  In the following lemmas, we will say that a closed, embedded curve $\gamma \subset i^*(X(E_K))$ is \emph{$p$-avoiding} if it is homologically essential in $X(T^2) \setminus \{P,Q\}$.  Such curves can only intersect the lines $p\alpha+\beta \equiv 0 \pmod{\pi}$ at points of the form $(\frac{\pi k}{p},0)$, where $k$ is an integer and $0 \leq k \leq p$, and we are about to show that in fact $k$ cannot be $0$ or $p$.

\begin{lemma} \label{lem:slope-p-closed-curve}
Let $p\geq 3$ be prime, and let $K$ be a knot in a homology sphere $Y$ whose exterior is irreducible and has incompressible boundary.  If $Y_p(K)$ is a lens space of order $p$, then the pillowcase image $i^*(X(E_K)) \subset X(T^2)$ contains a $p$-avoiding curve.  Any such curve necessarily avoids $(0,0)$ and $(\pi,0)$ but intersects both of the lines $L_0 = \{\beta\equiv 0 \pmod{2\pi}\}$ and $L_\pi = \{\beta\equiv \pi \pmod{2\pi}\}$.
\end{lemma}

\begin{proof}
Proposition~\ref{prop:rp3-surgery-zero-surgery} says that $Y_0(K)$ is irreducible, so if $w \in H^2(Y_0(K))$ is Poincar\'e dual to a meridian of $K$, then Theorem~\ref{thm:irreducible-nonzero} says that
\[ I^w_*(Y_0(K)) \neq 0. \]
The image $i^*(X(E_K))$ does not contain $P$ or $Q$ by Lemma~\ref{lem:p-avoid-lines}, since these points both satisfy $p\alpha+\beta \in \pi\Z$ but $\beta=\pi \not\in 2\pi\Z$.  Now Proposition~\ref{prop:curve-in-pillowcase} gives us the desired $p$-avoiding curve $\gamma$.  Since $\gamma$ is homologically essential in the complement of $\{P,Q\}$, it must intersect any path from $P$ to $Q$, and in particular it meets the line $L_\pi$ somewhere.

To see that $\gamma$ avoids $(0,0)$ and $(\pi,0)$, we argue exactly as in the proof of Proposition~\ref{prop:avoid-slope-2}: by Lemma~\ref{lem:corners-limit} these are not limit points of the image $i^*(X^\irr(E(K)))$ of irreducible characters, so $\gamma$ can only approach $(n\pi,0)$ (where $n$ is $0$ or $1$) along the arc $\beta \equiv 0 \pmod{2\pi}$.  In particular, if we parametrize $\gamma$ as a map
\[ \gamma: \R/\Z \hookrightarrow X(T^2) \]
with $\gamma(0) = (n\pi,0)$, then $\gamma$ must embed some open interval $(-\epsilon,\epsilon)$ in the arc $[0,\pi] \times \{0\}$ as a neighborhood of the endpoint $(n\pi,0)$, and this is impossible.

Now suppose that $\gamma$ avoids the line $L_0$.  In this case, Lemma~\ref{lem:p-avoid-lines} says that $\gamma$ is disjoint from both $L_0$ and each of the lines $\{p\alpha+\beta \equiv 0\pmod{p}\}$, since it can only meet the latter along $L_0$.  But then $\gamma$ lies in the complement of all of these lines, which is a disjoint union of $p$ open disks in $X(T^2) \setminus \{P,Q\}$ as illustrated in Figure~\ref{fig:gamma-meets-L0}.
\begin{figure}
\begin{tikzpicture}[style=thick]
\begin{scope}
  \draw plot[mark=*,mark size = 0.5pt] coordinates {(0,0)(3,0)(3,6)(0,6)} -- cycle;   
  \begin{scope}[decoration={markings,mark=at position 0.55 with {\arrow[scale=1]{>>}}}]
    \draw[postaction={decorate}] (0,0) -- (0,3);
    \draw[postaction={decorate}] (0,6) -- (0,3);
  \end{scope}
  \begin{scope}[decoration={markings,mark=at position 0.575 with {\arrow[scale=1]{>>>}}}]
    \draw[postaction={decorate}] (3,0) -- (3,3);
    \draw[postaction={decorate}] (3,6) -- (3,3);
  \end{scope}
  \begin{scope}[decoration={markings,mark=at position 0.6 with {\arrow[scale=1]{>>>>}}}]
    \draw[postaction={decorate}] (3,6) -- (0,6);
    \draw[postaction={decorate}] (3,0) -- (0,0);
  \end{scope}
  \draw[dotted] (0,3) -- (3,3);
  \draw[thin,|-|] (0,-0.4) node[below] {\small$0$} -- node[midway,inner sep=1pt,fill=white] {$\alpha$} ++(3,0) node[below] {\small$\vphantom{0}\pi$};
  \draw[thin,|-|] (-0.75,0) node[left] {\small$0$} -- node[midway,inner sep=1pt,fill=white] {$\beta$} ++(0,6) node[left] {\small$2\pi$};
  \coordinate (P) at (0,3);
  \coordinate (Q) at (3,3);
  \path (P) circle (0.075) node[left] {$P\vphantom{Q}$};
  \path (Q) circle (0.075) node[right] {$Q$};
  \fill[color=gray!50, fill opacity=0.5] (0,0) -- (2,0) -- (0,6) -- cycle;
  \fill[color=gray!50, fill opacity=0.5] (3,0) -- (1,6) -- (3,6) -- cycle;
  \draw[ultra thick,red] \foreach \y in {0,6} { (0,\y) -- ++(3,0) };
  \draw[ultra thick,red] (0,3) -- (1,0) (0,6) -- (2,0) (1,6) -- (3,0) (2,6) -- (3,3);
  \node[red,above,inner sep=3pt] at (0.5,0) {$L_0$};
  \node[red,below,inner sep=3pt] at (0.5,6) {$L_0$};
\end{scope}
\begin{scope}[xshift=6cm]
  \draw plot[mark=*,mark size = 0.5pt] coordinates {(0,0)(3,0)(3,6)(0,6)} -- cycle;   
  \begin{scope}[decoration={markings,mark=at position 0.55 with {\arrow[scale=1]{>>}}}]
    \draw[postaction={decorate}] (0,0) -- (0,3);
    \draw[postaction={decorate}] (0,6) -- (0,3);
  \end{scope}
  \begin{scope}[decoration={markings,mark=at position 0.575 with {\arrow[scale=1]{>>>}}}]
    \draw[postaction={decorate}] (3,0) -- (3,3);
    \draw[postaction={decorate}] (3,6) -- (3,3);
  \end{scope}
  \begin{scope}[decoration={markings,mark=at position 0.6 with {\arrow[scale=1]{>>>>}}}]
    \draw[postaction={decorate}] (3,6) -- (0,6);
    \draw[postaction={decorate}] (3,0) -- (0,0);
  \end{scope}
  \draw[dotted] (0,3) -- (3,3);
  \draw[thin,|-|] (0,-0.4) node[below] {\small$0$} -- node[midway,inner sep=1pt,fill=white] {$\alpha$} ++(3,0) node[below] {\small$\vphantom{0}\pi$};
  \draw[thin,|-|] (-0.75,0) node[left] {\small$0$} -- node[midway,inner sep=1pt,fill=white] {$\beta$} ++(0,6) node[left] {\small$2\pi$};
  \coordinate (P) at (0,3);
  \coordinate (Q) at (3,3);
  \path (P) circle (0.075) node[left] {$P\vphantom{Q}$};
  \path (Q) circle (0.075) node[right] {$Q$};
  \fill[color=gray!50, fill opacity=0.5] (0,0) -- (1.2,0) -- (0,6) -- cycle;
  \fill[color=gray!50, fill opacity=0.5] (0.6,6) -- (1.2,6) -- (2.4,0) -- (1.8,0) -- cycle;
  \fill[color=gray!50, fill opacity=0.5] (3,0) -- (1.8,6) -- (3,6) -- cycle;
  \draw[ultra thick,red] \foreach \y in {0,6} { (0,\y) -- ++(3,0) };
  \draw[ultra thick,red] (0,3) -- (0.6,0) (0,6) -- (1.2,0) (0.6,6) -- (1.8,0) (1.2,6) -- (2.4,0) (1.8,6) -- (3,0) (2.4,6) -- (3,3);
  \node[red,above,inner sep=3pt] at (0.3,0) {$L_0$};
  \node[red,below,inner sep=3pt] at (1,6) {$L_0$};
\end{scope}
\end{tikzpicture}
\caption{If the closed curve $\gamma$ avoids $L_0$ then it must lie in a disjoint union of $p$ open disks, shown here for $p=3$ and for $p=5$ with several of the disks shaded.}
\label{fig:gamma-meets-L0}
\end{figure}
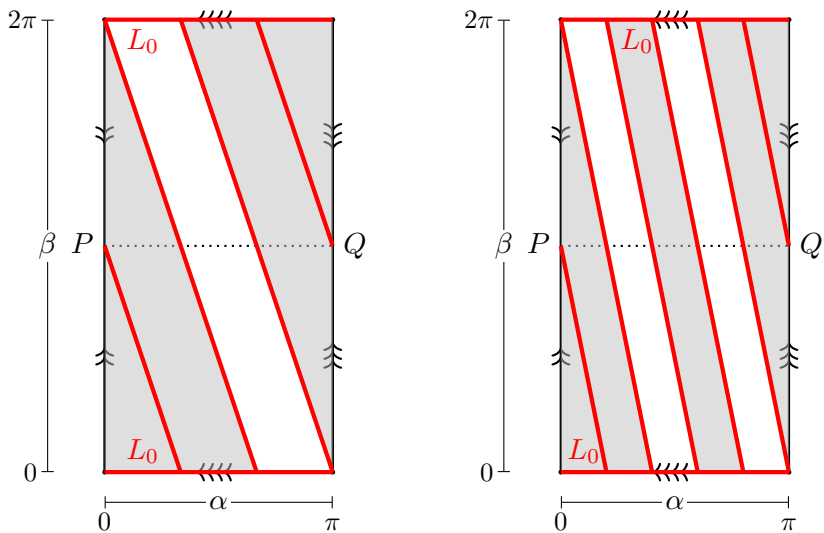
Since $\gamma$ is connected it must lie in one of these disks, which means that it is nullhomotopic in the complement of $P$ and $Q$.  This contradicts the fact that $\gamma$ is homologically essential, so $\gamma$ must meet the line $L_0$ after all.
\end{proof}

\begin{lemma} \label{lem:skew-transverse}
Let $\gamma \subset X(T^2)$ be a $p$-avoiding curve for some prime $p \geq 3$.  If $\gamma$ contains a point of the form $A_k = (\frac{2k\pi}{p}, 0)$, where $k$ is an integer satisfying $0 < k < \frac{p}{2}$, then there is an $\epsilon$-neighborhood $U$ of this point such that
\[ \gamma \cap U = \left(\tfrac{2k\pi}{p} - \epsilon, \tfrac{2k\pi}{p} + \epsilon\right) \times \{0\}. \]
In particular, if $\gamma' \subset X(T^2)$ is another $p$-avoiding curve passing through $A_k$, then $\gamma$ and $\sigma_p(\gamma')$ intersect transversely at $A_k$.
\end{lemma}

\begin{proof}
Suppose that $\gamma$ belongs to the pillowcase image of the character variety of $K \subset Y$.  Then Lemma~\ref{lem:p-avoid-lines} says that $A_k$ has some $\epsilon$-neighborhood $U$ in the pillowcase where every point of the corresponding image $i^*(X(E_K))$ is the image $(\alpha,0)$ of a reducible representation.  Since $\gamma$ is a subset of $i^*(X(E_K))$, it follows that the intersection $\gamma \cap U$ must be the open arc $\Gamma = (\frac{2k\pi}{p} - \epsilon, \frac{2k\pi}{p} + \epsilon) \times \{0\}$, as claimed.

Now if $\gamma'$ also passes through $A_k$, then it intersects some $\epsilon'$-neighborhood $U'$ of $A_k$ in the open arc 
\[ \Gamma' = (\tfrac{2k\pi}{p} - \epsilon', \tfrac{2k\pi}{p} + \epsilon') \times \{0\}, \]
where $\epsilon'$ may be different from $\epsilon$ because $\gamma'$ may come from a different knot $K' \subset Y'$.  In any case, we have $\sigma_p(A_k) = A_k$, so the image of $\Gamma'$ is an arc
\[  \sigma_p(\Gamma') = \{ (\alpha, 2\pi - p\alpha) \mid \tfrac{2k\pi}{p}-\epsilon' < \alpha < \tfrac{2k\pi}{p}+\epsilon' \} \]
of slope $-p$ through $A_k$.  See the left side of Figure~\ref{fig:arcs-near-B}.  The arcs $\Gamma$ and $\sigma_p(\Gamma')$ intersect transversely at $A_k$, hence so do the simple closed curves $\gamma$ and $\sigma_p(\gamma')$ to which they belong.
\end{proof}

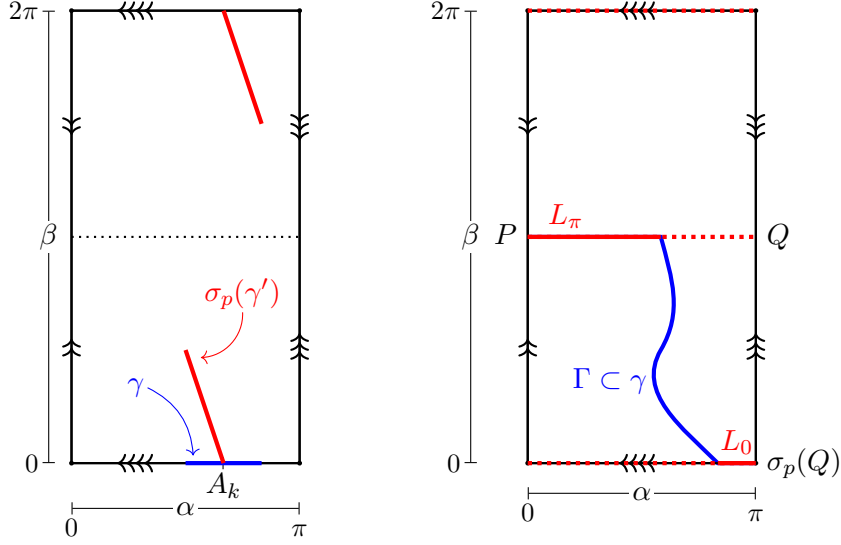
\begin{figure}
\begin{tikzpicture}[style=thick]
\begin{scope}
  \draw plot[mark=*,mark size = 0.5pt] coordinates {(0,0)(3,0)(3,6)(0,6)} -- cycle; 
  \begin{scope}[decoration={markings,mark=at position 0.55 with {\arrow[scale=1]{>>}}}]
    \draw[postaction={decorate}] (0,0) -- (0,3);
    \draw[postaction={decorate}] (0,6) -- (0,3);
  \end{scope}
  \begin{scope}[decoration={markings,mark=at position 0.575 with {\arrow[scale=1]{>>>}}}]
    \draw[postaction={decorate}] (3,0) -- (3,3);
    \draw[postaction={decorate}] (3,6) -- (3,3);
  \end{scope}
  \begin{scope}[decoration={markings,mark=at position 0.8 with {\arrow[scale=1]{>>>>}}}]
    \draw[postaction={decorate}] (3,6) -- (0,6);
    \draw[postaction={decorate}] (3,0) -- (0,0);
  \end{scope}
  \draw[dotted] (0,3) -- (3,3);
  \draw[thin,|-|] (0,-0.6) node[below] {\small$0$} -- node[midway,inner sep=1pt,fill=white] {$\alpha$} ++(3,0) node[below] {\small$\vphantom{0}\pi$};
  \draw[thin,-|] (1,0) -- (2,0) node[below] {$A_k$};
  \draw[thin,|-|] (-0.3,0) node[left] {\small$0$} -- node[midway,inner sep=1pt,fill=white] {$\beta$} ++(0,6) node[left] {\small$2\pi$};
  \draw[ultra thick,blue] (1.5,0) -- (2.5,0);
  \draw[thin,blue,->] (1,1) node[left,inner sep=1pt] {$\gamma$} to[bend left=30] (1.6,0.1);
  \draw[ultra thick,red] (1.5,1.5) -- (2,0) (2,6) -- (2.5,4.5);
  \draw[thin,red,->] (2.25,2) node[above,inner sep=1pt] {$\sigma_p(\gamma')$} to[bend left=45] (1.7,1.4);
\end{scope}
\begin{scope}[xshift=6cm]
  \begin{scope}[decoration={markings,mark=at position 0.55 with {\arrow[scale=1]{>>}}}]
    \draw[postaction={decorate}] (0,0) -- (0,3);
    \draw[postaction={decorate}] (0,6) -- (0,3);
  \end{scope}
  \begin{scope}[decoration={markings,mark=at position 0.575 with {\arrow[scale=1]{>>>}}}]
    \draw[postaction={decorate}] (3,0) -- (3,3);
    \draw[postaction={decorate}] (3,6) -- (3,3);
  \end{scope}
  \begin{scope}[decoration={markings,mark=at position 0.6 with {\arrow[scale=1]{>>>>}}}]
    \draw[postaction={decorate}] (3,6) -- (0,6);
    \draw[postaction={decorate}] (3,0) -- (0,0);
  \end{scope}
  \draw[thin,|-|] (0,-0.4) node[below] {\small$0$} -- node[midway,inner sep=1pt,fill=white] {$\alpha$} ++(3,0) node[below] {\small$\vphantom{0}\pi$};
  \draw[thin,|-|] (-0.75,0) node[left] {\small$0$} -- node[midway,inner sep=1pt,fill=white] {$\beta$} ++(0,6) node[left] {\small$2\pi$};
  \coordinate (P) at (0,3);
  \coordinate (Q) at (3,3);
  \path (P) circle (0.075) node[left] {$P\vphantom{Q}$};
  \path (Q) circle (0.075) node[right] {$Q$} -- ++(0,-3) node[right] {$\sigma_p(Q)$};
  \draw[line width=1.55pt,blue, preaction={yellow,double=yellow,double distance=2\pgflinewidth}] (0,3) -- (1.75,3) to[out=-75,in=60] (1.75,1.5) to[out=240,in=135] node[pos=0.25,left,inner sep=2pt] {$\Gamma \subset \gamma$} (2.5,0) -- (3,0);
  \draw plot[mark=*,mark size = 0.5pt] coordinates {(0,0)(3,0)(3,6)(0,6)} -- cycle;   
  \draw[ultra thick,red] (0,3) -- (1.75,3) (2.5,0) -- (3,0);
  \draw[ultra thick,dotted,red] \foreach \y in {0,3,6} { (0,\y) -- ++(3,0) };
  \node[red,above,inner sep=3pt] at (2.75,0) {$L_0$};
  \node[red,above,inner sep=3pt] at (0.5,3) {$L_\pi$};
\end{scope}
\end{tikzpicture}
\caption{Left: If the $p$-avoiding curves $\gamma$ and $\gamma'$ both pass through $A_k = (\frac{2k\pi}{p},0)$, then $\gamma$ and $\sigma_p(\gamma')$ must meet transversely at $A_k$.  Right: Using a path from $L_0$ to $L_\pi$ in $\gamma$ to construct a path $\tilde\Gamma$ from $\sigma_p(P)=P$ to $\sigma_p(Q)$, which must then intersect the closed, essential curve $\sigma_p(\gamma') \subset X(T^2) \setminus \{\sigma_p(P),\sigma_p(Q)\}$.}
\label{fig:arcs-near-B}
\end{figure}

\begin{lemma} \label{lem:pass-through-isolated-reducible}
Let $\gamma, \gamma' \subset X(T^2)$ be $p$-avoiding curves for some prime $p \geq 3$.  If the intersection
\[ \gamma \cap \sigma_p(\gamma') \]
is empty, then there are integers $k$ and $k'$ with $0 < k,k' < p$ such that $\gamma$ contains the point $(\frac{k\pi}{p},0)$ and $\gamma'$ contains the point $(\frac{k'\pi}{p},0)$.
\end{lemma}

\begin{proof}
It suffices to prove the desired conclusion for $\gamma'$, since we can use the fact that $\sigma_p$ is an involution to write
\[ \gamma' \cap \sigma_p(\gamma) = \sigma_p\big( \gamma \cap \sigma_p(\gamma') \big) = \emptyset \]
and thus freely exchange the roles of $\gamma$ and $\gamma'$.

According to Lemma~\ref{lem:slope-p-closed-curve}, the simple closed curve $\gamma$ meets both of the lines
\[ L_\pi = \{\beta \equiv 0 \!\!\!\pmod{\pi}\} \quad\text{and}\quad L_0 = \{\beta \equiv \pi \!\!\!\pmod{2\pi}\}, \]
so it contains an embedded path $\Gamma$ from $L_\pi$ to $L_0$.  Letting $P=(0,\pi)$ and $Q=(\pi,\pi)$ as usual, we form a path $\tilde\Gamma$ from $\sigma_p(P) = P$ to $\sigma_p(Q) = (\pi,0)$ by first following $L_\pi$ from $P$ until it meets $\Gamma$, then following $\Gamma$ until it meets $L_0$, and then following $L_0$ from there to $\sigma_p(Q)$. See the right side of Figure~\ref{fig:arcs-near-B}.

Since $\gamma'$ is homologically essential in $X(T^2)\setminus \{P,Q\}$, the image $\sigma_p(\gamma')$ is also homologically essential in $X(T^2) \setminus \{\sigma_p(P),\sigma_p(Q)\}$, and so it must intersect any path from $\sigma_p(P) = P$ to $\sigma_p(Q)$.  In particular, the intersection
\[ \tilde{\Gamma} \cap \sigma_p(\gamma') \]
is nonempty.  But we have assumed that $\sigma_p(\gamma')$ is disjoint from $\gamma$ and hence from the path $\Gamma \subset \tilde\Gamma$, so $\sigma_p(\gamma')$ must intersect $\tilde\Gamma$ along either $L_\pi$ or $L_0$.  This means that
\[ \sigma_p(\gamma') \cap \{\beta \in \pi\Z \} \neq \emptyset, \]
and we apply $\sigma_p$ to both sets to deduce that the intersection
\[ \gamma' \cap \sigma_p\left( \{\beta \in \pi\Z\} \right) = \gamma' \cap \{ p\alpha+\beta \in \pi\Z \} \]
is also nonempty.  Lemma~\ref{lem:p-avoid-lines} says that any point in this intersection must have the form $(\frac{k'\pi}{p},0)$, where $0 < k' < p$ by Lemma~\ref{lem:slope-p-closed-curve}, so $\gamma'$ contains such a point after all.
\end{proof}

We are now ready to specialize to $p=3$ and thus prove Theorem~\ref{thm:glue-slope-3}.

\begin{proof}[Proof of Theorem~\ref{thm:glue-slope-3}]
In order to find the desired representation $\pi_1(Y) \to \SU(2)$, Lemma~\ref{lem:look-for-intersection-3} says that it suffices to prove that
\[ i^*(X(E_{K_1})) \quad\text{and}\quad \sigma_3\left( i^*(X(E_{K_2})) \right) \]
intersect at some point of $X(T^2)$ other than $A_0 = (0,0)$ or $A_1 = (\frac{2\pi}{3},0)$.

We use Lemma~\ref{lem:slope-p-closed-curve} to find a pair of $3$-avoiding curves
\[ \gamma_j \subset i^*(X(E_{K_j})) \subset X(T^2) \setminus \{P,Q\}, \qquad j=1,2 \]
that both avoid the point $A_0$.  If they both pass through $A_1$, then Lemma~\ref{lem:skew-transverse} says that the simple closed curves $\gamma_1$ and $\sigma_3(\gamma_2)$ meet transversely at $A_1$.  We view these curves as lying in the pillowcase, which is topologically $S^2$, and then after orienting them arbitrarily their intersection number must be zero.  Their transverse intersection at $A_1$ contributes $\pm1$ to this intersection number $\gamma_1 \cdot \sigma_3(\gamma_2) = 0$, so there must be at least one other point of intersection.  This other point is neither $A_0$ nor $A_1$, so in this case we are done.

In the remaining case, the curves $\gamma_1$ and $\sigma_3(\gamma_2)$ do not intersect at $A_0$ or $A_1$.  If they have another point of intersection then we are done, so we can assume that $\gamma_1$ and $\sigma_3(\gamma_2)$ are disjoint.  But then Lemma~\ref{lem:pass-through-isolated-reducible} says that there must be integers $k_1,k_2 \in \{1,2\}$ such that
\[ \left(\frac{k_1\pi}{3},0\right) \in \gamma_1 \quad\text{and}\quad \left(\frac{k_2\pi}{3},0\right) \in \gamma_2. \]
At least one of the $k_j$ is equal to $1$, since otherwise $\gamma_1$ and $\sigma_3(\gamma_2)$ both contain $A_1 = \sigma_3(A_1)$.  For each such $j$, we use the involution
\[ \tau(\alpha,\beta) = (\pi-\alpha,2\pi-\beta) \]
of the pillowcase, which fixes each of the pillowcase images $i^*(X(E_{K_j})) \subset X(T^2)$ setwise by Lemma~\ref{lem:pillowcase-involution}, to replace $\gamma_j$ with the $3$-avoiding curve $\tau(\gamma_j)$ that passes through $\tau(\frac{\pi}{3},0) = (\frac{2\pi}{3},0)$.  But now we have found $3$-avoiding curves in $i^*(X(E_{K_1}))$ and $i^*(X(E_{K_2}))$ that both pass through $A_1$, so by the first case above they also must intersect at some point other than $A_0$ and $A_1$, completing the proof.
\end{proof}

This completes the proof of Theorem~\ref{thm:3-torsion}. \hfill\qed

\appendix

\section{The surgery exact triangle in irreducible instanton homology} \label{sec:exact-triangle-proof}

In this appendix, we verify some details needed for the proof of Theorem~\ref{thm:exact-triangle-2-torsion}, which generalizes the surgery exact triangle in instanton homology \cite{floer-surgery,braam-donaldson} to the irreducible instanton homology of surgery on a nullhomologous knot in some $Y$ such that $H_1(Y;\Z)$ is $2$-torsion.  We repeat Scaduto's proof of the surgery exact triangle \cite[\S5]{scaduto}, in which the maps in the exact triangle are induced by $2$-handle cobordisms
\begin{equation} \label{eq:triangle-cobordisms}
Y \xrightarrow{(W_0,c_0)} Y_0(K) \xrightarrow{(W_1,c_1)} Y_1(K) \xrightarrow{(W_2,c_2)} Y \to \cdots.
\end{equation}
In each case, we write $(W,c)$ to indicate that $W$ is a cobordism, and $c \subset W$ is a properly embedded surface such that $[c] \in H_2(W,\partial W)$ is Poincar\'e dual to the first Chern class of some line bundle, which specifies a $\SO(3)$-bundle on $W$ as usual.  The cobordism map associated to $(W,c)$ is then defined by counting solutions to the perturbed ASD equation on this bundle.

More generally, the proof of exactness also involves counting instantons on various compositions of these cobordisms, taken over various 1- and 2-dimensional families of metrics.  Since these define chain maps between various irreducible instanton homology groups, and chain homotopies between these, the instantons we count always have irreducible flat limits at either end of their cobordisms.  Scaduto's proof of exactness works without modification as long as the relevant moduli spaces can be compactified without reducible connections appearing in the middle of a broken flowline, since these would not be counted.

In what follows we will omit $K$ from our notation, writing $Y_n := Y_n(K)$.  We will also concatenate subscripts to denote the composition of two or more cobordisms, so that for example $W_{01} = W_0 \cup_{Y_0} W_1$ and $W_{120} = W_1 \cup_{Y_1} W_2 \cup_{Y} W_0$.

We first describe the basic cobordisms in \eqref{eq:triangle-cobordisms}.  We build $W_0$ by attaching a $0$-framed $2$-handle to $Y \times [0,1]$ along $K\times\{1\}$.  Then $W_1$ is the result of attaching a $-1$-framed $2$-handle along a meridian of $K$, and $W_2$ is likewise built out of a $-1$-framed $2$-handle attached along a meridian of the previous attaching curve.

\begin{lemma} \label{lem:triangle-cobordisms-topology}
We have $b_1(W_i) = b^+_2(W_i) = 0$ for each $i=0,1,2$, and the $W_i$ have signatures 
\[ \sigma(W_0) = \sigma(W_1) = 0, \quad \sigma(W_2) = -1. \]
\end{lemma}

\begin{proof}
The claim that $b_1(W_i) = 0$ follows from noting that $Y$ is a rational homology sphere and the knot $K \subset Y$ is nullhomologous.  Now the signatures of $W_0$, $W_{01}$, and $W_{012}$ are the same as the signatures of the linking matrices
\[ \begin{pmatrix} 0 \end{pmatrix}, \quad
\begin{pmatrix} 0 & 1 \\ 1 & -1 \end{pmatrix}, \quad
\begin{pmatrix} 0 & 1 & 0 \\ 1 & -1 & 1 \\ 0 & 1 & -1 \end{pmatrix} \]
for the Kirby diagrams of the respective cobordisms, and these signatures are $0$, $0$, and $-1$ respectively, so that $\sigma(W_0) = \sigma(W_1) = 0$ and $\sigma(W_2) = -1$ by additivity of signatures.  Since each $b_2(W_i)$ is $1$, the claim that $b_2^+(W_i) = 0$ follows immediately.
\end{proof}

Each $W_i$ is labeled with a properly embedded surface $c_i$, following \cite[\S3]{scaduto}: if the incoming end of $W_i$ is decorated with $\lambda$, then $c_i$ is (up to orientation) the union of a cylinder $\lambda \times [0,1]$ with a meridional disk of the attaching curve for the $2$-handle, pushed slightly into the interior of $W_i$ so that it is properly embedded with boundary on the outgoing end.  We note that if for some cobordism $(W,c)$ the ends of $c$ are nullhomologous in $\partial W$, as they are when $\partial W$ consists of copies of $Y$ or $Y_1$, then $[c] \in H_2(W,\partial W)$ lifts to a class in $H_2(W)$, so that $c^2 \in \Z$; in this case the class of $c\pmod{2}$ determines uniquely the value of $c^2 \pmod{4}$.

\begin{lemma} \label{lem:c-squared-mod-4}
We have $c_2^2 \equiv 0 \pmod{4}$ and $c_{01}^2 \equiv c_{012}^2 \equiv c_{201}^2 \equiv -1 \pmod{4}$.
\end{lemma}

\begin{proof}
We realize $c_0 \subset W_0$ by taking a disk in $Y\times\{1\}$ with boundary $\mu_K \times \{1\}$, where $\mu_K$ is a meridian $\mu_K$ of the attaching curve $K\times \{1\}$, and pushing its interior into the interior of $Y\times [0,1]$.  The homology $H_2(W_0)$ is generated by the union $F_0$ of a Seifert surface for $K$ and a core of the $2$-handle; we have $F_0^2 = 0$, and $c_0 \cdot F_0 = 1$.
  
Then $W_1$ is built by attaching a $-1$-framed $2$-handle to $Y_0 \times [0,1]$ along $\mu_K$, and $c_1$ is the union of $\mu_K \times [0,1]$ and a disk bounded by a meridian of $\mu_K$ with some orientation.  We observe that $W_{01}$ is diffeomorphic to a blow-up of the trace of $1$-surgery on $K$, and then $H_2(W_{01})$ is generated by a capped-off Seifert surface $F_1$ and the exceptional sphere $E$, with $F_1^2 = 1$ and $E^2 = -1$.  We have $c_{01} \cdot E \equiv 1\pmod{2}$ by \cite[\S3.4]{scaduto} (see also \cite[Lemma~2.1]{floer-surgery}), and
\[ c_{01} \cdot F_0 \equiv c_0 \cdot F_0 \equiv 1 \pmod{2} \]
since $F_0 \subset W_0$, whence
\[ c_{01} \cdot F_1 \equiv c_{01} \cdot (F_0 - E) \equiv 0 \pmod{2}. \]
Since $c_{01}$ has nullhomologous ends in $Y$ and $Y_1$, it lifts to a class in $H_2(W_{01})$.  Then the above intersection numbers tell us that $c_{01} \equiv E \pmod{2}$, and so $c_{01}^2 \equiv -1 \pmod{4}$.

Meanwhile, we have $c_2 \equiv 0 \pmod{2}$ as in \cite[\S2]{braam-donaldson}, so that $c_2^2 \equiv 0 \pmod{4}$.  And we note that the surfaces $c_{012}$ and $c_{201}$ are each homologous to a disjoint union of closed surfaces in the classes of $c_{01} \subset W_{01}$ and $c_2 \subset W_2$ in some order, so we conclude that
\[ c_{012}^2 \equiv c_{201}^2 \equiv c_{01}^2 + c_2^2 \equiv -1 \pmod{4}. \qedhere \]
\end{proof}

Next, we note that a cobordism either to or from $Y_0$, equipped with a bundle whose restriction to $Y_0$ is the admissible $w$, does not admit any reducible ASD connections at all.  This is because any such connection must limit at the $Y_0$ end to a reducible flat connection over $Y_0$, and such limiting connections do not exist by the admissibility of $w$.  Thus we can restrict our attention to the cobordisms
\begin{equation} \label{eq:exact-triangle-interesting-cobordisms}
(W_2,c_2),\ (W_{01},c_{01}),\ (W_{012},c_{012}),\ (W_{201},c_{201}),
\end{equation}
which are the only other cobordisms considered in \cite[\S5]{scaduto}.  We note from Lemma~\ref{lem:triangle-cobordisms-topology} that these all have $b_1=0$, and that by additivity of signature their signatures are $-1$, $0$, $-1$, $-1$ respectively while their second Betti numbers are $1$, $2$, $3$, $3$, so that
\begin{equation} \label{eq:cobordisms-bplus}
b^+_2(W_2)=0, \qquad b^+_2(W_{01}) = b^+_2(W_{012}) = b^+_2(W_{201}) = 1.
\end{equation}
Moreover, since $K$ is nullhomologous it follows that each of the classes $[c_s] \in H_2(W_s,\partial W_s)$ appearing in \eqref{eq:exact-triangle-interesting-cobordisms} has nullhomologous boundary in $\partial W_s$, hence lifts to a class in the corresponding $H_2(W_s)$.

From now on, we focus our attention on reducible connections on each of the cobordisms \eqref{eq:exact-triangle-interesting-cobordisms}.  We will call a reducible instanton \emph{central} if its holonomy is central, and \emph{abelian} if it is not.

An abelian instanton on any of the cobordisms $(W,c)$ of \eqref{eq:exact-triangle-interesting-cobordisms} limits to a central connection at either end, because $H_1(Y)$ and $H_1(Y_1)$ are both $2$-torsion and thus all reducible flat connections on $Y$ and $Y_1$ are central.  According to \cite[Proposition~1.7]{miller-equivariant}, the components of the space of abelian instantons on $(W,c)$ are parametrized by pairs
\[ \left\{ \{x,y\} \subset H^2(W;\Z) \mid x+y=\mathrm{PD}(c),\ x\neq y \right\}; \]
given a $U(2)$-bundle $E\to Y$ such that $\lambda = \det(E)$ has first Chern class $\mathrm{PD}(c)$, we send a reducible connection that induces a splitting into complex line bundles $E \cong \eta \oplus (\lambda\otimes \eta^{-1})$ to the set $\{x,y\} = \{c_1(\eta),c_1(\lambda\otimes \eta^{-1})\}$.  Given a perturbation $\pi_W$ on $W$ restricting to perturbations $\pi,\pi'$ on the incoming and outgoing ends, and given an abelian instanton $\Lambda$ in the component labeled by $\{x,y\}$, the component $D^\nu_{\Lambda,\pi_W}$ of the linearized ASD operator at $\Lambda$ that is normal to the reducible locus has index $N(\Lambda;\pi,\pi') \in 2\Z$ equal to
\begin{equation} \label{eq:normal-index}
N(\Lambda;\pi,\pi') = -2(x-y)^2 - 2b^+_2(W) - 2,
\end{equation}
by \cite[Proposition~3.23]{daemi-miller-eismeier}, which in turn comes from \cite[\S4.5]{miller-equivariant} (and is greatly simplified here because $b_1(W)=0$ and because the limiting flat connections at either end of $W$ must be central).  We omit perturbations from the notation from here on, but note that by taking them sufficiently small on the interior of $W$, we will always have $-2(x-y)^2 \geq 0$ by \cite[Remark~3.28]{daemi-miller-eismeier}.

We can say more about the $-2(x-y)^2$ term.  Working dually in homology, we note that $H_2(W)$ is free abelian, since $W$ is built by attaching 2-handles to either $Y$ or $Y_1$ (both of which have $H_2=0$) along nullhomologous knots.  In fact, we have a splitting
\[ H_2(W,\partial W) \cong H_2(W) \oplus \ker\big(i_*: H_1(\partial W) \to H_1(W)\big), \]
whose second term is $2$-torsion since $H_1(\partial W)$ is.  The class $c$ lies in the $H_2(W)$ summand, so if we have $x+y=c$ then we can write the summands with respect to this splitting as $x=(\alpha, \tau)$ and $y=(c-\alpha,\tau)$ for some $\alpha \in H_2(W)$ and some $2$-torsion element $\tau \in \ker(i_*)$.  But then $x-y = 2\alpha-c$, and both $\alpha$ and $c$ have integral square since they lift to $H_2(W)$.  We can thus compute that
\begin{equation} \label{eq:index-square-mod-8}
{-}2(x-y)^2 \equiv -2c^2 \pmod{8},
\end{equation}
and the right side of this congruence only depends on the mod $2$ class of $c$.  We will use this to bound $N(\Lambda)$ from below.

\begin{lemma} \label{lem:index-composite}
Fix $s \in \{01,012,201\}$.  Then there are no central instantons on $(W_s,c_s)$, and any abelian instanton $\Lambda$ on $(W_s,c_s)$ has normal index $N(\Lambda) \geq -2$.
\end{lemma}

\begin{proof}
To see that there are no central instantons, we note that they restrict to reducible flat connections over $Y_0 \subset W_s$, and the admissibility of $w \to Y_0$ rules such connections out.
 
If $\Lambda$ is reducible, then by equations~\eqref{eq:normal-index} and \eqref{eq:cobordisms-bplus}, we have
\[ N(\Lambda) = -2(x-y)^2 - 4 \]
where the $-2(x-y)^2$ term is nonnegative.  We also have $-2(x-y)^2 \equiv -2c_s^2 \equiv 2\pmod{8}$, by \eqref{eq:index-square-mod-8} and Lemma~\ref{lem:c-squared-mod-4}, so we conclude that it is at least $2$ and thus $N(\Lambda) \geq -2$.
\end{proof}

\begin{proof}[Proof of Theorem~\ref{thm:exact-triangle-2-torsion}]
As described above, we simply repeat the proof of the exact triangle in \cite[\S5]{scaduto}, and we only have to check that there are no compactness issues caused by broken flowlines with reducible (hence central) flat connections on a $3$-manifold in the middle.  If every ASD connection in a broken flowline is irreducible, but they are glued along flat connections at least one of which is central, then the index will be at least $3$.  We can thus restrict our attention to broken flowlines with at least one reducible instanton.  

We have seen that this can only happen on one of the cobordisms $(W,c)$ listed in \eqref{eq:exact-triangle-interesting-cobordisms}, and the proof does not make use of higher-dimensional families on $(W_2,c_2)$, so we really only need to consider
\[ (W,c) = (W_s,c_s), \quad s=01,012,201. \]
We remark that by Lemma~\ref{lem:index-composite}, a reducible instanton $\Lambda$ on one of these $(W,c)$ must be abelian with normal index at least $-2$, and by \cite[Proposition~4.26]{miller-equivariant} its index satisfies
\[ \ind(\Lambda) = N(\Lambda) - (1 - b_1(W) + b^+_2(W)) = N(\Lambda) - 2, \]
so then $\ind(\Lambda) \geq -4$.

We now fix a pair of irreducible flat connections $a$ and $a'$ on the ends $Y_a$ and $Y_{a'}$ of $W$.  Suppose that some sequence of instantons in the moduli space $\cM(W;a,a')$ converges to a broken flowline
\[ a \xrightarrow{B} b \xrightarrow{\Lambda} b' \xrightarrow{B'} a', \]
where $B$ and $B'$ are ASD connections on $\R \times Y_a$ and $\R \times Y_{a'}$, and $\Lambda$ is abelian; then generically $B$ and $B'$ have index at least $1$, while $\ind(\Lambda) \geq -4$.  Then by standard gluing results, the broken flowline has index
\[ \ind(B) + \ind(\Lambda) + \ind(B') + 5 \geq 3. \]
(To explain the constant on the left, we first glue $B$ to the abelian $\Lambda$ along the central $b$, with $\dim H^0_b(Y_a)=3$, and then similarly glue the irreducible result of this to $B'$ along the central $b'$.)  In particular, this broken flowline does not belong to the compactification of an at most $2$-dimensional moduli space.

The same holds for broken flowlines with longer chains of connections over $\R \times Y_a$ or $\R\times Y_{a'}$.  We conclude that the moduli spaces of dimension at most $2$ that appear in \cite[\S5]{scaduto} can be compactified without adding any broken flowlines with reducible components, so the proof of the exact triangle goes through as before.
\end{proof}

\bibliographystyle{myalpha}
\bibliography{References}

\end{document}